\documentclass{article}
\usepackage{enumitem}

\usepackage[a4paper,margin=3.0cm]{geometry}

\usepackage[T1]{fontenc}
\usepackage[utf8]{inputenc}
\usepackage{lmodern}
\usepackage[english]{babel}
\usepackage[final]{microtype}

\usepackage[dvipsnames]{xcolor}
\definecolor{linkblue}{HTML}{0000FF}
\definecolor{citegreen}{HTML}{0B6B3A}
\definecolor{urlviolet}{HTML}{7A1E76}

\usepackage{mathtools}
\usepackage{amsmath,amssymb,amsthm}
\usepackage{bbm}
\usepackage{graphicx}
\numberwithin{equation}{section}
\usepackage{esint}

\theoremstyle{plain}
\newtheorem{theorem}{Theorem}[section]
\newtheorem{proposition}[theorem]{Proposition}
\newtheorem{lemma}[theorem]{Lemma}
\newtheorem{corollary}[theorem]{Corollary}
\newtheorem{assumption}[theorem]{Assumption}

\theoremstyle{definition}
\newtheorem{definition}[theorem]{Definition}

\theoremstyle{remark}
\newtheorem{remark}[theorem]{Remark}

\usepackage{booktabs}
\usepackage{csquotes}
\usepackage{bm}

\usepackage[backend=biber,style=alphabetic]{biblatex}
\usepackage[
  colorlinks=true,
  linkcolor=linkblue,
  citecolor=citegreen,
  urlcolor=urlviolet
]{hyperref}
\usepackage[nameinlink,capitalize]{cleveref}

\crefname{theorem}{Theorem}{Theorems}
\crefname{proposition}{Proposition}{Propositions}
\crefname{lemma}{Lemma}{Lemmas}
\crefname{corollary}{Corollary}{Corollaries}
\crefname{definition}{Definition}{Definitions}
\crefname{remark}{Remark}{Remarks}
\crefname{assumption}{Assumption}{Assumptions}

\newcommand{\N}{\mathbb{N}}
\newcommand{\R}{\mathbb{R}}
\newcommand{\T}{\mathbb{T}}
\newcommand{\E}{\mathbb{E}}

\newcommand{\Pcal}{\mathcal{P}}
\newcommand{\Ccal}{\mathcal{C}}
\newcommand{\Dcal}{\mathcal{D}}

\newcommand{\Id}{\mathrm{Id}}
\newcommand{\dd}{\,\mathrm{d}}

\DeclareMathOperator{\Law}{Law}     
\DeclareMathOperator*{\esssup}{ess\,sup} 
\newcommand{\Kcal}{\mathcal{K}}

\DeclareMathOperator{\Lip}{Lip}

\DeclareMathOperator{\divg}{div}
\DeclareMathOperator{\Concat}{Concat}

\newcommand{\abs}[1]{\left\lvert #1 \right\rvert}
\newcommand{\ip}[2]{\left\langle #1,\,#2 \right\rangle}

\newcommand{\Tail}{\operatorname{Tail}}
\DeclareMathOperator{\Train}{Train}
\DeclareMathOperator{\clip}{clip}

\newcommand{\Z}{\mathbb{Z}}

\begin{document}

\title{Statistical Error Bounds for Generative Solvers of Chaotic PDEs:\\
Wasserstein Stability, Generalization, and Turbulence}
\author{Victor Armegioiu\\
Department of Mathematics, ETH Z\"urich\\
\href{mailto:victor.armegioiu@math.ethz.ch}{victor.armegioiu@math.ethz.ch}}
\date{}
\maketitle

\begin{abstract}
Statistical solutions of incompressible Euler describe turbulent dynamics as time-parameterized laws on $L^2$ whose
multi-point correlations satisfy an infinite hierarchy of weak identities. Modern generative samplers for PDE forecasting
(flow matching, rectified flows, diffusion via probability-flow ODEs) are measure-transport mechanisms and therefore
induce Markov operators on laws. We develop a law-level analysis compatible with the correlation-measure framework of
Lanthaler--Mishra--Par\'es-Pulido (LM): convergence in
$d_T(\mu,\nu)=\int_0^T W_1(\mu_t,\nu_t)\,\dd t$, compactness controlled by structure functions, and identification of limits
through hierarchy identities.

Quantitatively, we prove a $W_2$ stability estimate whose growth rate is a distance-weighted average strain under optimal
couplings, and a one-step error decomposition into a resolved mismatch term and an unavoidable high-frequency coverage
tail controlled by structure-function (spectral) bounds. These inputs propagate through multi-step rollouts via a
discrete Gr\"onwall recursion with amplification governed by the average-strain exponent rather than a worst-case
Lipschitz constant. On the qualitative side, sampler-native path controls yield LM time regularity; together with uniform
energy and structure-function bounds this gives precompactness in $d_T$ and strong convergence of LM-admissible
observables. If hierarchy residuals vanish along a sequence, every limit is an LM statistical solution, with residuals
bounded by training-native drift/score regression errors. Finally, we show how common finite-grid diagnostics--proper
distributional scores and likelihood-style certificates--admit principled interpretations as resolved observables within
the same statistical-solution framework.
\end{abstract}

\section{Introduction}

In multiscale incompressible turbulence, small perturbations can lead to substantially different fine-scale outcomes,
even when coarse features remain comparable. This makes long-horizon prediction of individual realizations a fragile
objective, while ensemble statistics can remain stable and informative. A natural alternative is therefore to study the
evolution of probability measures on the space of velocity fields and to evaluate forecasts through robust statistical
quantities (e.g.\ correlations and spectra) defined by these laws.

\paragraph{What is missing in current ML practice.}
Recent machine-learning approaches to PDE forecasting increasingly aim at \emph{distributional} prediction \cite{price2025probabilistic, bulte2025probabilistic, price2023gencast, larsson2025diffusion, andrae2024continuous, gao2024bayesian, mahesh2024huge}, by producing
ensembles of samples rather than point forecasts. In practice, however, most analyses remain tied to discretized models
and heuristic statistical comparisons (sample moments, empirical spectra, ad hoc calibration scores), without a clear
connection to a continuum quantitative notion of solution at the level of measures.
This mismatch is especially visible for Euler-type dynamics, where nonuniqueness and limited regularity make trajectory
semigroups problematic \cite{majda2002vorticity}, while statistical notions of solution are natural \cite{LMP2021}. The goal here is to provide a law-level
framework in which learned generative forecasters can be studied quantitatively and, in suitable limits, identified as
statistical solutions in a rigorous sense.

\paragraph{Why the Lanthaler--Mishra--Par\'es-Pulido framework.}
Several measure-valued/statistical formulations exist for incompressible Euler. We adopt the correlation-measure approach
of Lanthaler--Mishra--Par\'es-Pulido~\cite{LMP2021} for three concrete reasons that interface directly with learning:

\begin{enumerate}[label=(\roman*),leftmargin=2.2em]
\item \emph{A concrete topology with observable stability.}
LM work in the metric
\[
d_T(\mu,\nu)=\int_0^T W_1(\mu_t,\nu_t)\,\dd t,
\]
which is strong enough to yield convergence of expectations for a large, explicit class of admissible observables.
This matches how learned surrogates are evaluated: by statistics of samples.

\item \emph{Compactness controlled by structure functions.}
LM identify structure-function bounds as a verifiable compactness mechanism. Structure functions are directly estimable
from ensembles and constitute a natural ``physics-to-sampler'' interface.

\item \emph{A precise Euler law definition via hierarchy identities.}
The Euler dynamics is encoded as an infinite family of weak identities for correlation measures.
This provides a crisp notion of ``being Euler'' at the law level and a target for certification of learned evolutions.
\end{enumerate}

\paragraph{Samplers are law operators.}
Modern conditional generative models for PDE forecasting--flow matching, rectified flows, and diffusion models--are
intrinsically \emph{measure transport} mechanisms \cite{molinaro2024generative, armegioiu2025rectified, schiff2024dyslim, kohl2023benchmarking}.
At a fixed physical step size $\Delta t$, a sampler defines a Markov kernel $K_{\Delta t}(u,\cdot)$ on state space and
therefore a Markov operator on laws $\mathcal T_{\Delta t}:\Pcal(L^2)\to\Pcal(L^2)$.
Moreover, many samplers provide an \emph{internal-time interpolation} (an ODE in internal time for rectified flows / flow
matching; the probability-flow ODE for diffusion). This interpolation naturally defines within-step path measures and
continuity equations on law space, which are exactly the structures needed to study time regularity and compactness.

\paragraph{Two threads, one interface: quantitative bounds and statistical-solution identification.}
The paper is organized around two complementary threads that meet at a common interface of verifiable, sampler-native
quantities.

\smallskip
\noindent\textbf{Thread I: quantitative error propagation in $W_2$.}
On the quantitative side we measure one-step and rollout discrepancies in $W_2$ on $L^2$-valued fields.
The choice of $W_2$ is deliberate: projection/tail errors are naturally quadratic, and $W_2$ interacts cleanly with
$L^2$-based harmonic analysis. On bounded-energy classes, $W_1\le W_2$, so $W_2$ bounds also control the $W_1$ distances
that enter the LM topology.

Two mechanisms drive the estimates:
(i) a refined stability bound for the Euler (or reference) flow in which the growth of the squared $L^2$ distance between
two solutions is governed by the rate of strain weighted by the squared separation, and (ii) a one-step
capacity--coverage decomposition in which approximation error splits into a \emph{training} mismatch on resolved scales
and an unavoidable \emph{coverage} term due to unresolved high frequencies.
The coverage term is controlled purely by structure-function tails, yielding explicit algebraic rates under power-law
structure moduli.
These one-step bounds are then propagated through rollouts via a discrete Gr\"onwall recursion, whose amplification
depends on the distance-weighted strain exponent rather than a worst-case Lipschitz constant.

\smallskip
\noindent\textbf{Thread II: compactness and identification as LM statistical solutions.}
On the qualitative side we place sampler-induced law evolutions into the LM setting.
Under uniform energy admissibility, LM time regularity, and a uniform structure-function modulus, compactness holds in
$(L^1_t(\Pcal),d_T)$ and expectations of admissible observables converge strongly~\cite{LMP2021}.
If, in addition, the Euler hierarchy identities hold up to residuals that vanish along a sequence, then every
subsequential limit is an LM statistical solution.

A distinctive feature of the present setting is that residual control can be expressed in \emph{training-native} terms.
For drift-driven law curves (e.g.\ probability-flow ODE sampling), the hierarchy residual for tensor-product tests is an
exact expected defect obtained by replacing the Euler drift by the learned drift. This yields explicit bounds in terms of
drift regression losses. Since training occurs at finite resolution, certification is naturally formulated on resolved
scales, and diffusion score regression is connected to drift regression via the probability-flow identity.

\paragraph{Scope and reference dynamics.}
The quantitative rollout analysis compares the learned one-step operator $\mathcal T_{\Delta t}$ to a reference one-step
pushforward $(S_{\Delta t})_\#$.
Depending on the application, $S_{\Delta t}$ can be interpreted as the exact Euler flow map on a well-posed smooth class
over a short time window, or as a deterministic reference solver/truncation used to generate training data.
The compactness/identification results are formulated at the level of Euler correlation hierarchies in the LM sense,
and do not rely on uniqueness of trajectories.

\paragraph{Contributions.}
Our contributions are grouped into three threads that meet at a common law-level interface.

\begin{enumerate}[label=\textbf{(\Roman*)},leftmargin=2.4em]
\item \textbf{Quantitative rollout analysis at the level of laws (Sections~\ref{sec:approx}--\ref{sec:rollout}).}
We derive finite-horizon error guarantees for learned one-step operators on laws.
Section~\ref{sec:stability} proves a $W_2$ stability estimate for the reference dynamics in which the amplification rate is a
\emph{distance-weighted average strain} evaluated along coupled pairs, rather than a worst-case Lipschitz constant.
Section~\ref{sec:approx} isolates the finite-resolution obstruction via a capacity--coverage decomposition:
one-step law error splits into a resolved mismatch term and an unavoidable high-frequency tail, with the tail controlled by
structure-function (spectral) bounds.
Section~\ref{sec:rollout} combines these ingredients into a discrete Gr\"onwall recursion that cleanly separates stability
amplification from injected one-step defects, yielding interpretable multi-step rollout bounds.

\item \textbf{Compactness and identification as LM statistical solutions (Sections~\ref{sec:LM}, \ref{sec:regularity}--\ref{sec:identify}).}
We place sampler-induced law evolutions into the Lanthaler--Mishra--Par\'es-Pulido framework.
Section~\ref{sec:LM} fixes the topology $d_T$ and the admissible observable class in which convergence is meaningful, and
recalls the correlation-hierarchy identities that define Euler at the law level.
Section~\ref{sec:regularity} shows that sampler-native within-step trajectories can be packaged into segment kernels and
concatenated into a global path measure, producing the measurable couplings required by LM time-regularity from a uniform
expected-speed/straightness bound.
Section~\ref{sec:identify} then performs the identification step: under LM compactness inputs and vanishing hierarchy
residuals (plus incompressibility), every subsequential $d_T$ limit is an LM statistical solution.

\item \textbf{Training-native certification and application-facing diagnostics (Sections~\ref{sec:training} and \ref{sec:applications}).}
We connect the abstract identification criterion to quantities that appear naturally in training and evaluation.
Section~\ref{sec:training} expresses (resolved) hierarchy residuals as \emph{drift defects} and bounds them explicitly by an
$L^2$ drift regression loss; for diffusion models, Subsection~\ref{subsec:pfode} records the exact identity converting score
regression into the corresponding probability-flow drift regression.
Section~\ref{sec:applications} then shows how the same law-level viewpoint subsumes common practice:
proper scores such as CRPS/energy score become resolved Lipschitz observables controlled by $d_T$, while diffusion-style
likelihood certificates can be treated as LM-admissible observables and, under a strong convexity hypothesis, can be turned
into quantitative mean-square error certificates in the coupled pipeline.
\end{enumerate}

\subsection{Motivation and related work}

A growing ML literature now targets distributional PDE forecasting, producing ensembles and evaluating uncertainty with proper scoring rules such as CRPS or energy score, often entirely at the discretized-field level. Representative examples include diffusion-based ensemble weather models \cite{price2023gencast, bulte2025probabilistic, price2025probabilistic, larsson2025diffusion, molinaro2024generative, armegioiu2025rectified} and CRPS-trained operational-style systems \cite{lang2026aifs}, as well as probabilistic neural operator frameworks using proper scoring rules on function outputs. In our setting,
these scores like CRPS and its multi-dimensional variant (the energy score) can be written as expectations of Lipschitz \emph{resolved observables} of the law (e.g.\ mollified point
evaluations or low-dimensional projections). Consequently, convergence in the LM metric $d_T$ yields quantitative control
of such scores: time-integrated CRPS/energy-score discrepancies are bounded by $d_T$ up to the Lipschitz constant of the
chosen observable. We make this link explicit in Section~\ref{subsec:crps}, showing that the LM observable framework
subsumes common distributional metrics used in ML PDE forecasting.
While these works demonstrate impressive empirical calibration and skill, the evaluation and interpretation of “distributional correctness” is typically detached from a continuum notion of law evolution and from the weak identities that define the target PDE in measure-valued form. The works \cite{molinaro2024generative, armegioiu2025rectified} make partial progress towards quantitatively closing this gap via spectral analysis estimates. This paper supplies that missing bridge: it analyzes generative samplers as operators on laws, quantifies finite-resolution error propagation, and connects training losses to certification of statistical-solution identities in the LM framework.

\section{The LM framework: laws, observables, compactness, statistical solutions}
\label{sec:LM}

This section fixes the precise notion of statistical solution and the topology in which compactness and limit passages
will be performed. The framework of Lanthaler--Mishra--Par\'es-Pulido~\cite{LMP2021} is adopted for two reasons.

\smallskip
\noindent\textbf{(i) Observable stability in a concrete topology.}
The Euler hierarchy identities are expressed as expectations of multi-point observables. The LM topology
\[
d_T(\mu,\nu)=\int_0^T W_1(\mu_t,\nu_t)\,\dd t
\]
is strong enough to yield convergence of a large admissible observable class, once a verifiable small-scale compactness
criterion is imposed.

\smallskip
\noindent\textbf{(ii) Compactness controlled by structure functions.}
The spatial compactness input is formulated through second-order structure functions, which measure the distribution of
small-scale increments. These quantities are directly estimable from ensemble samples and therefore form a natural
interface between theory and learned simulators.

\subsection{Basic notation and phase space}

Let $D=\T^d$ be the flat torus and set $L^2_x:=L^2(D;\R^d)$ with norm $\|u\|_2:=\|u\|_{L^2(D)}$.
Write $L^2_\sigma\subset L^2_x$ for divergence-free vector fields. The theory below is stated on $L^2_x$ and
incompressibility is imposed through a correlation constraint (equivalently, concentration on $L^2_\sigma$).

Let $\Pcal(L^2_x)$ be the set of Borel probability measures on $L^2_x$. For $p\ge 1$ write $\Pcal_p(L^2_x)$ for the
subclass of measures with finite $p$th moment $\int \|u\|_2^p\,\dd\mu(u)<\infty$.

\subsection{Wasserstein distances}

For $\mu,\nu\in\Pcal_1(L^2_x)$ define the $1$-Wasserstein distance with cost $\|u-v\|_2$ by
\[
W_1(\mu,\nu)
:=
\inf_{\pi\in\Pi(\mu,\nu)}
\int_{L^2_x\times L^2_x}\|u-v\|_2\,\dd\pi(u,v),
\]
where $\Pi(\mu,\nu)$ denotes couplings with marginals $\mu$ and $\nu$. When $\mu,\nu\in\Pcal_2(L^2_x)$, define $W_2$ by
\[
W_2^2(\mu,\nu)
:=
\inf_{\pi\in\Pi(\mu,\nu)}
\int_{L^2_x\times L^2_x}\|u-v\|_2^2\,\dd\pi(u,v).
\]
The $W_1$ distance is used to define the LM topology on time-parameterized laws; the $W_2$ distance is used in the
quantitative stability and approximation analysis later.

\subsection{$L^1_t(\Pcal)$ and the LM metric $d_T$}

The LM framework treats a law evolution as a curve $t\mapsto \mu_t$ in $\Pcal(L^2_x)$. The natural topology compares such
curves by integrating the Wasserstein distance in time.

\begin{definition}[$L^1_t(\Pcal)$ and $d_T$]\label{def:Lt1}
Fix $T>0$. A curve $\mu_\cdot$ belongs to $L^1([0,T);\Pcal(L^2_x))$ (denoted $L^1_t(\Pcal)$) if it is weak-$\ast$ measurable and
\begin{equation}\label{eq:Lt1}
\int_0^T\int_{L^2_x}\|u\|_2\,\dd\mu_t(u)\,\dd t<\infty.
\end{equation}
On $L^1_t(\Pcal)$ define
\begin{equation}\label{eq:dT}
d_T(\mu,\nu):=\int_0^T W_1(\mu_t,\nu_t)\,\dd t.
\end{equation}
\end{definition}

\begin{remark}[Why $L^1_t(\Pcal)$ is the right ambient space]
Condition \eqref{eq:Lt1} is the minimal integrability ensuring that $W_1(\mu_t,\nu_t)$ is finite for a.e.\ $t$ and hence
that $d_T$ is well-defined. Convergence in $d_T$ is stronger than narrow convergence at each time, but weak enough to be
compatible with the Euler nonlinearity once the structure-function compactness input is imposed.
\end{remark}

\subsection{Time-regularity}

Spatial compactness alone does not yield compactness of time-parameterized laws; a temporal control is needed. LM (Definition 2.2 in \cite{LMP2021}) encode
time regularity by requiring explicit couplings between times whose expected displacement is small in a negative Sobolev
norm. This formulation is robust under limited regularity and is tailored to compactness in $d_T$.

\begin{definition}[Time-regularity]\label{def:time-regular}
A curve $\mu_\cdot\in L^1_t(\Pcal)$ is \emph{time-regular} if there exist $L\in\N$, $C>0$, and a measurable assignment
$(s,t)\mapsto\pi_{s,t}\in\Pcal(L^2_x\times L^2_x)$ such that for a.e.\ $s,t\in[0,T)$:
\begin{enumerate}[label=(\roman*),leftmargin=2.2em]
\item $\pi_{s,t}\in\Pi(\mu_s,\mu_t)$;
\item \begin{equation}\label{eq:time-reg}
\int_{L^2_x\times L^2_x}\|u-v\|_{H^{-L}}\,\dd\pi_{s,t}(u,v)\le C|t-s|.
\end{equation}
\end{enumerate}
A family $\{\mu^\Delta_\cdot\}_{\Delta>0}$ is \emph{uniformly time-regular} if the same $(C,L)$ works for all $\Delta$.
\end{definition}

\begin{remark}[How time-regularity will be verified later]
Later sections derive time-regularity from sampler-native controls by constructing couplings from path measures (superposition)
and estimating increments by integrating velocities. This provides a direct route from action/straightness bounds to
\eqref{eq:time-reg}.
\end{remark}

\begin{lemma}[LM time-regularity is closed under $d_T$ limits]
\label{lem:time-reg-closed}
Let $\mu^m_\cdot\in L^1([0,T);\Pcal(L^2_x))$ be time-regular with the same constants $(C,L)$ in the sense of
Definition~\ref{def:time-regular}. Assume
\[
d_T(\mu^m,\mu)\to 0
\quad\text{and}\quad
\sup_m\int_0^T\int_{L^2_x}\|u\|_2\,\dd\mu^m_t(u)\,\dd t<\infty.
\]
Assume $D=\T^d$ so that the embedding $L^2_x\hookrightarrow H^{-L}(D)$ is compact.
Then $\mu_\cdot$ is time-regular with the same $(C,L)$.
\end{lemma}

\begin{proof}
Let $\lambda:=\frac{1}{T^2}\,\dd s\,\dd t$ on $[0,T]^2$.
For each $m$, let $(s,t)\mapsto\pi^m_{s,t}\in\Pi(\mu^m_s,\mu^m_t)$ be a measurable assignment such that
\[
\int_{L^2_x\times L^2_x}\|u-v\|_{H^{-L}}\,\dd\pi^m_{s,t}(u,v)\le C|t-s|
\quad\text{for a.e.\ }(s,t)\in[0,T]^2.
\]

\smallskip
\noindent\textbf{Step 1: build averaged couplings and extract a limit in $H^{-L}$.}
Define $\Theta^m\in\Pcal([0,T]^2\times L^2_x\times L^2_x)$ by
\[
\Theta^m(\dd s\,\dd t\,\dd u\,\dd v)
:=\lambda(\dd s\,\dd t)\,\pi^m_{s,t}(\dd u,\dd v).
\]
View $L^2_x\times L^2_x$ as embedded in $H^{-L}\times H^{-L}$.
Let $B_R:=\{u\in L^2_x:\|u\|_2\le R\}$. Since $B_R$ is relatively compact in $H^{-L}$ and
\[
\Theta^m(\{(s,t,u,v):\|u\|_2>R\})
=\frac1T\int_0^T \mu^m_s(\{\|u\|_2>R\})\,\dd s
\le \frac{1}{RT}\int_0^T\int\|u\|_2\,\dd\mu^m_s\,\dd s,
\]
(and similarly for $v$), the family $\{\Theta^m\}$ is tight on $[0,T]^2\times H^{-L}\times H^{-L}$.
Hence, after extracting a subsequence (not relabeled),
\[
\Theta^m\Rightarrow\Theta
\quad\text{weakly in }\Pcal\big([0,T]^2\times H^{-L}\times H^{-L}\big).
\]
The $(s,t)$-marginal of each $\Theta^m$ is $\lambda$, hence the $(s,t)$-marginal of $\Theta$ is also $\lambda$.
Disintegrate $\Theta$ w.r.t.\ $\lambda$:
there exists a $\lambda$-a.e.\ defined measurable family $(s,t)\mapsto\pi_{s,t}\in\Pcal(H^{-L}\times H^{-L})$ such that
\[
\Theta(\dd s\,\dd t\,\dd u\,\dd v)=\lambda(\dd s\,\dd t)\,\pi_{s,t}(\dd u,\dd v).
\]

\smallskip
\noindent\textbf{Step 2: identify the marginals using $H^{-L}$-continuous tests.}
Let $\iota:L^2_x\to H^{-L}$ denote the continuous embedding and set
\[
\mu^{m,H}_t:=\iota_\#\mu^m_t\in\Pcal(H^{-L}),
\qquad
\mu^{H}_t:=\iota_\#\mu_t\in\Pcal(H^{-L}).
\]
Since $d_T(\mu^m,\mu)=\int_0^T W_1(\mu^m_t,\mu_t)\,\dd t\to0$, after extracting a further subsequence we may assume
\[
W_1(\mu^m_t,\mu_t)\to0 \quad\text{for a.e.\ }t\in[0,T].
\]
For such $t$, $W_1(\mu^m_t,\mu_t)\to0$ implies $\mu^m_t\Rightarrow\mu_t$ narrowly in $L^2_x$, and therefore (by continuity of $\iota$)
\[
\mu^{m,H}_t=\iota_\#\mu^m_t \Rightarrow \iota_\#\mu_t=\mu^H_t
\quad\text{narrowly in }H^{-L},\ \text{for a.e.\ }t.
\]

Fix $\psi\in C([0,T]^2)$ and $\varphi\in C_b(H^{-L})$.
Define $F(s,t,u,v):=\psi(s,t)\varphi(u)$, which is bounded continuous on $[0,T]^2\times H^{-L}\times H^{-L}$.
Then $\Theta^m\Rightarrow\Theta$ gives
\begin{equation}\label{eq:step2-theta-limit}
\int F\,\dd\Theta^m \longrightarrow \int F\,\dd\Theta.
\end{equation}
Compute the left-hand side using $\Theta^m=\lambda\otimes\pi^m_{s,t}$ and the fact that the first marginal of $\pi^m_{s,t}$ is $\mu^m_s$:
\[
\int F\,\dd\Theta^m
=
\int_{[0,T]^2}\psi(s,t)\Big(\int_{L^2_x}\varphi(\iota(u))\,\dd\mu^m_s(u)\Big)\dd\lambda(s,t)
=
\int_{[0,T]^2}\psi(s,t)\Big(\int_{H^{-L}}\varphi(u)\,\dd\mu^{m,H}_s(u)\Big)\dd\lambda(s,t).
\]
For a.e.\ $s$, $\mu^{m,H}_s\Rightarrow\mu^H_s$, hence $\int\varphi\,\dd\mu^{m,H}_s\to\int\varphi\,\dd\mu^H_s$.
Since $|\int\varphi\,\dd\mu^{m,H}_s|\le\|\varphi\|_\infty$, dominated convergence yields
\[
\int F\,\dd\Theta^m
\longrightarrow
\int_{[0,T]^2}\psi(s,t)\Big(\int_{H^{-L}}\varphi(u)\,\dd\mu^{H}_s(u)\Big)\dd\lambda(s,t).
\]
On the other hand, disintegration of $\Theta$ gives
\[
\int F\,\dd\Theta
=
\int_{[0,T]^2}\psi(s,t)\Big(\int_{H^{-L}\times H^{-L}}\varphi(u)\,\dd\pi_{s,t}(u,v)\Big)\dd\lambda(s,t).
\]
Comparing with \eqref{eq:step2-theta-limit} and using arbitrariness of $\psi$, we conclude that for $\lambda$-a.e.\ $(s,t)$,
\[
\int_{H^{-L}\times H^{-L}}\varphi(u)\,\dd\pi_{s,t}(u,v)=\int_{H^{-L}}\varphi(u)\,\dd\mu^H_s(u)
\qquad\forall \varphi\in C_b(H^{-L}),
\]
i.e.\ the first marginal of $\pi_{s,t}$ is $\mu^H_s$. The same argument with $F(s,t,u,v)=\psi(s,t)\varphi(v)$ shows that
the second marginal is $\mu^H_t$.

Since $\mu^H_s,\mu^H_t$ are supported on $\iota(L^2_x)$, $\pi_{s,t}$ is supported on $\iota(L^2_x)\times\iota(L^2_x)$.
Because $\iota$ is continuous and injective between Polish spaces, Lusin--Souslin implies $\iota(L^2_x)$ is Borel in $H^{-L}$
and $\iota^{-1}:\iota(L^2_x)\to L^2_x$ is Borel. Define
\[
\widetilde\pi_{s,t}:=(\iota^{-1}\times \iota^{-1})_\#\pi_{s,t}\in\Pcal(L^2_x\times L^2_x).
\]
Then $\widetilde\pi_{s,t}\in\Pi(\mu_s,\mu_t)$ for $\lambda$-a.e.\ $(s,t)$.

\smallskip
\noindent\textbf{Step 3: pass the increment bound to the limit.}
Fix $\psi\ge0$ in $C([0,T]^2)$ and set $G(s,t,u,v):=\psi(s,t)\|u-v\|_{H^{-L}}$ on $[0,T]^2\times H^{-L}\times H^{-L}$.
Then $G\ge0$ is continuous, hence by Portmanteau,
\[
\int G\,\dd\Theta
\le \liminf_{m\to\infty}\int G\,\dd\Theta^m.
\]
But for each $m$,
\[
\int G\,\dd\Theta^m
=
\int_{[0,T]^2}\psi(s,t)\Big(\int\|u-v\|_{H^{-L}}\,\dd\pi^m_{s,t}\Big)\dd\lambda(s,t)
\le C\int_{[0,T]^2}\psi(s,t)|t-s|\,\dd\lambda(s,t).
\]
Therefore,
\[
\int G\,\dd\Theta
\le C\int_{[0,T]^2}\psi(s,t)|t-s|\,\dd\lambda(s,t).
\]
Disintegrating $\Theta$ yields that for $\lambda$-a.e.\ $(s,t)$,
\[
\int_{H^{-L}\times H^{-L}}\|u-v\|_{H^{-L}}\,\dd\pi_{s,t}(u,v)\le C|t-s|.
\]
Since $\iota$ is the canonical injection of $L^2_x$ into $H^{-L}(D)$ (identifying $L^2$ functions with distributions),
we have $\|\iota(u)-\iota(v)\|_{H^{-L}}=\|u-v\|_{H^{-L}}$ for all $u,v\in L^2_x$. Therefore, pushing forward by $\iota^{-1}\times\iota^{-1}$ yields, for $\lambda$-a.e.\ $(s,t)$,
\[
\int_{L^2_x\times L^2_x}\|u-v\|_{H^{-L}}\,\dd\widetilde\pi_{s,t}(u,v)\le C|t-s|.
\]

Thus $\mu_\cdot$ is time-regular with constants $(C,L)$.
\end{proof}

\subsection{Structure functions and a pointwise-to-time-averaged link}

Structure functions quantify the mean-square size of increments. They measure how much energy sits at spatial scales
$\lesssim r$ and are central in turbulence diagnostics. In the LM theory, uniform control of a second-order structure
function is the compactness mechanism at the law level.

The LM time-averaged second-order structure function is
\begin{equation}\label{eq:SF}
S_r^2(\mu_\cdot;T)
:=
\Bigg(\int_0^T\int_{L^2_x}\int_D\fint_{B_r(0)}|u(x+h)-u(x)|^2\,\dd h\,\dd x\,\dd\mu_t(u)\,\dd t\Bigg)^{1/2}.
\end{equation}

In applications it is common to control structure functions pointwise in time. The following lemma makes the relation to
\eqref{eq:SF} explicit.

\begin{lemma}[Pointwise structure modulus implies the LM time-averaged bound]
\label{lem:pointwise-to-timeavg}
Assume there exists a modulus $\omega:[0,\infty)\to[0,\infty)$ such that for a.e.\ $t\in[0,T]$ and all $r>0$,
\begin{equation}\label{eq:pointwise-struct}
\int_{L^2_x}\int_D\fint_{B_r(0)}|u(x+h)-u(x)|^2\,\dd h\,\dd x\,\dd\mu_t(u)\le \omega(r)^2.
\end{equation}
Then $S_r^2(\mu_\cdot;T)\le \sqrt{T}\,\omega(r)$ for all $r>0$.
\end{lemma}

\begin{proof}
By \eqref{eq:SF},
\[
S_r^2(\mu_\cdot;T)^2
=
\int_0^T\left(\int_{L^2_x}\int_D\fint_{B_r(0)}|u(x+h)-u(x)|^2\,\dd h\,\dd x\,\dd\mu_t(u)\right)\dd t.
\]
Insert \eqref{eq:pointwise-struct} and integrate:
\[
S_r^2(\mu_\cdot;T)^2 \le \int_0^T \omega(r)^2\,\dd t = T\,\omega(r)^2.
\]
Taking square roots yields the claim.
\end{proof}

\subsection{Correlation measures and incompressibility (LM Definition 2.3 and Lemma 3.1)}

LM encode multi-point statistics by correlation measures: a hierarchy of Young measures $\nu^k_{t,x}$ describing the joint
law of $(u(x_1),\dots,u(x_k))$ for almost every spatial tuple $x=(x_1,\dots,x_k)\in D^k$. The hierarchy satisfies symmetry,
consistency, and a diagonal continuity condition ensuring compatibility with $L^2$ increments. The precise definition is
given in~\cite{LMP2021} (Definition~2.3); it will be used implicitly via the correspondence theorem (Theorem~2.3).

Incompressibility is imposed through a constraint on the two-point correlation marginal. LM show that this condition is
equivalent to concentration of $\mu_t$ on divergence-free fields (Lemma~3.1).

\subsection{LM compactness and strong convergence of admissible observables}

Two black-box results from~\cite{LMP2021} drive the compactness/identification pipeline:

\smallskip
\noindent\textbf{(1) Compactness from structure functions.}
Uniform time-regularity, uniform $L^2$ boundedness, and a uniform structure-function modulus imply relative compactness
in $(L^1_t(\Pcal),d_T)$ (Theorem~2.2).

\smallskip
\noindent\textbf{(2) Strong convergence of admissible observables.}
Along a convergent subsequence, expectations of LM-admissible observables converge strongly (Theorem~2.4). The admissible
class is defined next.

\begin{definition}[LM-admissible observables]\label{def:LM-adm}
Fix $k\in\N$ and write $\xi=(\xi_1,\dots,\xi_k)\in(\R^d)^k$. For $\xi,\xi'\in(\R^d)^k$ define, for each $i=1,\dots,k$,
\[
\Pi_i(\xi,\xi'):=\prod_{j\ne i}\bigl(1+|\xi_j|^2+|\xi'_j|^2\bigr).
\]
A function $g\in C\big([0,T)\times D^k\times(\R^d)^k\big)$ is \emph{LM-admissible} if there exists $C>0$ such that for all
$(t,x,\xi)\in[0,T)\times D^k\times(\R^d)^k$ and all $\xi,\xi'\in(\R^d)^k$,
\begin{align}
|g(t,x,\xi)|
&\le C\prod_{i=1}^k\bigl(1+|\xi_i|^2\bigr),
\label{eq:adm-growth}\\
|g(t,x,\xi)-g(t,x,\xi')|
&\le
C\sum_{i=1}^k
\Pi_i(\xi,\xi')\,\sqrt{1+|\xi_i|^2+|\xi_i'|^2}\,|\xi_i-\xi_i'|.
\label{eq:adm-lip}
\end{align}
\end{definition}

We crystallize as a lemma the following useful consequence of the marginal consistency property:
\begin{lemma}[Marginalization of correlation measures]\label{lem:corr-marg}
Let $U=\R^q$ and let $\{\nu^k_{t,x}\}_{k\ge 1}$ be the correlation-measure hierarchy associated with a law $\mu_t$ on $L^2(D;U)$.
Then for each $k\ge 1$, each $i\in\{1,\dots,k\}$, and a.e.\ $x=(x_1,\dots,x_k)\in D^k$,
\[
(\mathrm{pr}_i)_\#\nu^k_{t,x}=\nu^1_{t,x_i},
\]
where $\mathrm{pr}_i:U^k\to U$ is the projection onto the $i$-th component.
Consequently, for any Borel $\psi:U\to\R$ with $\int_D\langle \nu^1_{t,y},|\psi|\rangle\,\dd y<\infty$,
\[
\int_{D^k}\langle \nu^k_{t,x},\psi(\xi_i)\rangle\,\dd x
=
|D|^{k-1}\int_D\langle \nu^1_{t,y},\psi\rangle\,\dd y.
\]
\end{lemma}

\begin{proof}
Fix $k\ge1$ and $i\in\{1,\dots,k\}$. Let $\sigma$ be a permutation of $\{1,\dots,k\}$ such that $\sigma(1)=i$,
and let $\Sigma_\sigma:U^k\to U^k$ be the coordinate permutation
$\Sigma_\sigma(\xi_1,\dots,\xi_k)=(\xi_{\sigma(1)},\dots,\xi_{\sigma(k)})$.

\smallskip
\noindent\textbf{Step 1: reduce to the first coordinate by symmetry.}
By the symmetry axiom of correlation measures, for a.e.\ $x=(x_1,\dots,x_k)\in D^k$,
\[
\nu^k_{t,(x_{\sigma(1)},\dots,x_{\sigma(k)})}=(\Sigma_\sigma)_\#\nu^k_{t,x}.
\]
Hence for any bounded Borel $\psi:U\to\R$,
\[
\langle \nu^k_{t,x},\psi(\xi_i)\rangle
=
\langle \nu^k_{t,(x_{\sigma(1)},\dots,x_{\sigma(k)})},\psi(\xi_1)\rangle.
\]

\smallskip
\noindent\textbf{Step 2: eliminate variables using consistency.}
Define $f_k(\xi_1,\dots,\xi_k):=\psi(\xi_1)$ on $U^k$.
Since $f_k$ depends only on the first $k-1$ variables, the LM consistency axiom gives
\[
\langle \nu^k_{t,(y_1,\dots,y_k)}, f_k\rangle
=
\langle \nu^{k-1}_{t,(y_1,\dots,y_{k-1})}, f_{k-1}\rangle,
\qquad
f_{k-1}(\xi_1,\dots,\xi_{k-1}):=\psi(\xi_1).
\]
Iterating this reduction $k-1$ times yields
\[
\langle \nu^k_{t,(y_1,\dots,y_k)},\psi(\xi_1)\rangle
=
\langle \nu^1_{t,y_1},\psi\rangle.
\]
Applying this with $y_j=x_{\sigma(j)}$ gives
\[
\langle \nu^k_{t,(x_{\sigma(1)},\dots,x_{\sigma(k)})},\psi(\xi_1)\rangle
=
\langle \nu^1_{t,x_{\sigma(1)}},\psi\rangle
=
\langle \nu^1_{t,x_i},\psi\rangle.
\]
Combining with Step 1 gives
\[
\langle \nu^k_{t,x},\psi(\xi_i)\rangle=\langle \nu^1_{t,x_i},\psi\rangle
\quad\text{for a.e.\ }x\in D^k,
\]
i.e.\ $(\mathrm{pr}_i)_\#\nu^k_{t,x}=\nu^1_{t,x_i}$.

\smallskip
\noindent\textbf{Step 3: integrate over $D^k$.}
Integrating the pointwise identity and using Fubini,
\[
\int_{D^k}\langle \nu^k_{t,x},\psi(\xi_i)\rangle\,\dd x
=
\int_{D^k}\langle \nu^1_{t,x_i},\psi\rangle\,\dd x
=
|D|^{k-1}\int_D\langle \nu^1_{t,y},\psi\rangle\,\dd y,
\]
which is the claimed formula (under the stated integrability condition).
\end{proof}

\subsection{Statistical solutions of Euler (LM Definition 3.1)}

Let $F(\xi)=\xi\otimes\xi$. For divergence-free $\phi_i\in C^\infty([0,T)\times D;\R^d)$ define the tensor test
$\phi(t,x)=\phi_1(t,x_1)\otimes\cdots\otimes\phi_k(t,x_k)$. LM define statistical solutions by requiring the following
hierarchy identity for all $k$ and all divergence-free tests, together with time-regularity and incompressibility.

\begin{definition}[LM statistical solution]\label{def:LM-SS}
Let $\bar\mu\in\Pcal_2(L^2_x)$ and let $\bar\nu$ be its correlation measure.
A curve $\mu_\cdot\in L^1_t(\Pcal)$ is a statistical solution of incompressible Euler with initial law $\bar\mu$ if:
\begin{enumerate}[label=(\roman*),leftmargin=2.2em]
\item $\mu_\cdot$ is time-regular in the sense of \cref{def:time-regular};
\item letting $\nu_t$ denote the correlation hierarchy associated to $\mu_t$, for every $k\in\N$ and every divergence-free
$\phi_1,\dots,\phi_k\in C^\infty([0,T)\times D;\R^d)$,
\begin{align}
&\int_0^T\int_{D^k}\Big[
\ip{\nu^k_{t,x}}{\xi_1\otimes\cdots\otimes\xi_k}:\partial_t\phi(t,x)
+\sum_{i=1}^k \ip{\nu^k_{t,x}}{\xi_1\otimes\cdots\otimes F(\xi_i)\otimes\cdots\otimes\xi_k}:\nabla_{x_i}\phi(t,x)
\Big]\dd x\,\dd t \nonumber\\
&\qquad\qquad +\int_{D^k}\ip{\bar\nu^k_x}{\xi_1\otimes\cdots\otimes\xi_k}:\phi(0,x)\,\dd x =0;
\label{eq:LM-hierarchy}
\end{align}
\item the incompressibility constraint holds for a.e.\ $t$ (equivalent to concentration on $L^2_\sigma$):
for all $\psi\in C_c^\infty(D)$,
\begin{equation}\label{eq:LM-divfree}
\int_{D^2}\ip{\nu^2_{t,x_1,x_2}}{\xi_1\otimes\xi_2}:(\nabla\psi(x_1)\otimes\nabla\psi(x_2))\,\dd x_1\,\dd x_2 =0.
\end{equation}
\end{enumerate}
\end{definition}

\subsection{Energy admissibility}

\begin{definition}[Energy admissible]\label{def:energy-adm}
A curve $\mu_\cdot\in L^1_t(\Pcal)$ is energy admissible if
\[
\sup_{t\in[0,T)}\int_{L^2_x}\|u\|_2^2\,\dd\mu_t(u)<\infty.
\]
\end{definition}

\begin{remark}[Role of energy admissibility]
Energy admissibility is not part of the LM definition of statistical solution, but it is a natural physics-informed
constraint and a basic uniform integrability input. In the present work it serves as an interface condition: it is used
to control admissible observable growth and to define physically meaningful compactness classes.
\end{remark}

\section{From samplers to law evolutions}
\label{sec:sampler}

This section formalizes the law-level viewpoint for generative samplers. The central observation is that a sampler does
not merely output point predictions; it defines conditional output distributions and therefore a Markov kernel on state
space \cite{schiff2024dyslim}. This kernel induces a Markov operator on laws, which can be compared directly to the Euler-induced pushforward on
$\Pcal(L^2)$.

In addition, most samplers provide an \emph{internal-time interpolation} connecting a reference law to the output law.
For flow matching and rectified flows this interpolation is generated by an ODE in internal time and satisfies a
continuity equation on state space. For diffusion models, an associated probability-flow ODE provides an analogous
deterministic transport representation. These internal-time curves are the natural objects for establishing time
regularity and compactness properties.

\subsection{Forecasting as a Markov operator on laws}

Fix a physical step size $\Delta t>0$. A one-step conditional simulator is naturally described by a Markov kernel
\[
K_{\Delta t}(u,A)\in[0,1],\qquad u\in L^2_x,\ A\subset L^2_x\ \text{Borel},
\]
where $K_{\Delta t}(u,\cdot)$ is the model output law at time $t+\Delta t$ conditioned on the input state $u$ at time $t$.
The induced operator on laws is
\begin{equation}\label{eq:law-op}
(\mathcal T_{\Delta t}\mu)(A)=\int_{L^2_x}K_{\Delta t}(u,A)\,\dd\mu(u),
\end{equation}
so that the model rollout at discrete times $t_n=n\Delta t$ is
\[
\widehat\mu_{t_{n+1}} = \mathcal T_{\Delta t}\widehat\mu_{t_n}.
\]
For deterministic simulators $K_{\Delta t}(u,\cdot)=\delta_{\widehat S_{\Delta t}(u)}$, this reduces to pushforward by a map:
$\widehat\mu_{t+\Delta t}=(\widehat S_{\Delta t})_\#\widehat\mu_t$.

\subsection{Flow matching and rectified flows: internal-time ODE sampling}

Many modern probabilistic PDE solvers are implemented as conditional generative samplers--notably diffusion-based rollouts for turbulence, latent diffusion generators for turbulent flow fields, and diffusion-based ensemble weather systems--yet are rarely formalized as Markov operators on laws \cite{kohl2023benchmarking,du2024conditional, price2025probabilistic}.
The operator viewpoint in \eqref{eq:law-op} and the mixture continuity equation in Proposition~\ref{prop:mixture-CE} make this explicit: a sampler induces a law evolution and (under mild integrability) a closed continuity equation with an averaged drift. This provides the right mathematical object for both stability analysis and for compatibility with LM time-regularity/compactness, and it puts “sampling trajectories” (PF-ODE/flow matching paths) on the same footing as law dynamics.\\

Flow matching and rectified flow samplers generate samples by integrating an internal-time ODE. For the one-step forecast
conditional on the input state $u$, consider an internal time parameter $\tau\in[0,1]$ and the dynamics
\begin{equation}\label{eq:fm-ode}
\frac{\dd}{\dd\tau}U_\tau = v_\theta(U_\tau,\tau;u),\qquad U_0\sim \nu_0(\cdot;u),
\end{equation}
where $\nu_0(\cdot;u)$ is a reference law (often Gaussian in a latent/physical representation) and $v_\theta$ is a learned
velocity field. The terminal law of $U_1$ defines the one-step kernel via
\[
K_{\Delta t}(u,\cdot) = \Law(U_1\,|\,u).
\]

To analyze the sampler at the law level, it is convenient to consider the \emph{internal-time interpolated laws}
\[
\nu_\tau(\cdot;u) := \Law(U_\tau\,|\,u),\qquad \tau\in[0,1].
\]
Under mild integrability assumptions, these laws satisfy a continuity equation on state space in weak form.

\subsection{Weak continuity equation for the conditional internal-time laws}

Fix $u$ and a cylindrical observable
\[
\Phi(x) = \varphi(\ip{x}{\phi_1},\dots,\ip{x}{\phi_m}),
\qquad \phi_j\in C^\infty(D;\R^d),\ \varphi\in C_b^1(\R^m).
\]
Then
\[
D\Phi(x)=\sum_{j=1}^m \partial_j\varphi(\ip{x}{\phi_1},\dots,\ip{x}{\phi_m})\,\phi_j.
\]

\begin{proposition}[Weak continuity equation for the conditional internal-time laws]
\label{prop:cond-CE}
Assume that for $\nu_0(\cdot;u)$-a.e.\ initial condition the ODE \eqref{eq:fm-ode} admits an absolutely continuous solution
on $[0,1]$, and that
\begin{equation}\label{eq:cond-speed}
\int_0^1\int_{L^2_x}\|v_\theta(x,\tau;u)\|_2\,\dd\nu_\tau(x;u)\,\dd\tau <\infty.
\end{equation}
Then for every cylindrical $\Phi$, the map $\tau\mapsto \int\Phi\,\dd\nu_\tau(\cdot;u)$ is absolutely continuous and for a.e.\ $\tau$,
\begin{equation}\label{eq:cond-CE}
\frac{\dd}{\dd\tau}\int_{L^2_x}\Phi(x)\,\dd\nu_\tau(x;u)
=
\int_{L^2_x}\ip{D\Phi(x)}{v_\theta(x,\tau;u)}\,\dd\nu_\tau(x;u).
\end{equation}
\end{proposition}

\begin{proof}
Let $U_\tau$ be an absolutely continuous solution of \eqref{eq:fm-ode}. By the chain rule for Fr\'echet differentiable $\Phi$,
\[
\frac{\dd}{\dd\tau}\Phi(U_\tau)=D\Phi(U_\tau)[\dot U_\tau]=\ip{D\Phi(U_\tau)}{v_\theta(U_\tau,\tau;u)}.
\]
Integrate from $0$ to $\tau$:
\[
\Phi(U_\tau)-\Phi(U_0)=\int_0^\tau \ip{D\Phi(U_s)}{v_\theta(U_s,s;u)}\,\dd s.
\]
Take expectation with respect to $U_0\sim\nu_0(\cdot;u)$, use Tonelli justified by \eqref{eq:cond-speed}, and note
$\Law(U_s)=\nu_s(\cdot;u)$:
\[
\int\Phi\,\dd\nu_\tau - \int\Phi\,\dd\nu_0
=
\int_0^\tau \int \ip{D\Phi(x)}{v_\theta(x,s;u)}\,\dd\nu_s(x;u)\,\dd s.
\]
Thus $\tau\mapsto \int\Phi\,\dd\nu_\tau$ is absolutely continuous and its derivative is \eqref{eq:cond-CE}.
\end{proof}

\subsection{Mixture interpolation for a physical step and a closed continuity equation}

The conditional internal-time laws $\nu_\tau(\cdot;u)$ are defined for each fixed input state $u$. To obtain an
unconditional law curve over the physical interval $[t_n,t_{n+1}]$, define the \emph{mixture interpolation}:
for $\tau\in[0,1]$,
\begin{equation}\label{eq:mixture}
\widehat\mu_{t_n+\tau\Delta t}(\cdot) := \int_{L^2_x} \nu_\tau(\cdot;u)\,\dd\widehat\mu_{t_n}(u).
\end{equation}
Equivalently, for any bounded measurable $F$,
\[
\int_{L^2_x} F(x)\,\dd\widehat\mu_{t_n+\tau\Delta t}(x)
=
\int_{L^2_x}\left(\int_{L^2_x}F(x)\,\dd\nu_\tau(x;u)\right)\dd\widehat\mu_{t_n}(u).
\]
This mixture is the law-level object naturally associated to the sampler within one physical step.

The next proposition derives a \emph{closed} continuity equation for the mixture interpolation, with an averaged drift
field obtained by disintegrating the two-stage sampling distribution. 

\begin{proposition}[Closed continuity equation for the mixture interpolation]\label{prop:mixture-CE}
Assume \eqref{eq:cond-speed} holds for $\widehat\mu_{t_n}$-a.e.\ $u$ and that the map
$u\mapsto \int_0^1\int\|v_\theta(x,\tau;u)\|_2\,\dd\nu_\tau(x;u)\dd\tau$ is integrable under $\widehat\mu_{t_n}$.
Then there exists a Borel map $\bar v_\tau:L^2_x\to L^2_x$ such that for every cylindrical $\Phi$,
the map $\tau\mapsto \int\Phi\,\dd\widehat\mu_{t_n+\tau\Delta t}$ is absolutely continuous and for a.e.\ $\tau$,
\begin{equation}\label{eq:mixture-CE}
\frac{\dd}{\dd\tau}\int_{L^2_x}\Phi(x)\,\dd\widehat\mu_{t_n+\tau\Delta t}(x)
=
\int_{L^2_x}\ip{D\Phi(x)}{\bar v_\tau(x)}\,\dd\widehat\mu_{t_n+\tau\Delta t}(x).
\end{equation}
Moreover, $\bar v_\tau$ can be chosen as the conditional expectation
\[
\bar v_\tau(x)=\E\big[v_\theta(x,\tau;U)\mid X=x\big],
\]
where $(U,X)$ is the two-stage sampling pair: first $U\sim\widehat\mu_{t_n}$, then $X\sim\nu_\tau(\cdot;U)$.
\end{proposition}

\begin{proof}
Fix a cylindrical $\Phi$ and define
\[
G(\tau):=\int_{L^2_x}\Phi(x)\,\dd\widehat\mu_{t_n+\tau\Delta t}(x)
=
\int_{L^2_x}\left(\int_{L^2_x}\Phi(x)\,\dd\nu_\tau(x;u)\right)\dd\widehat\mu_{t_n}(u).
\]
For each fixed $u$, $\tau\mapsto \int\Phi\,\dd\nu_\tau(\cdot;u)$ is absolutely continuous and its derivative is given by
\eqref{eq:cond-CE}. By the integrability assumption, differentiation under the $u$-integral is justified, and for a.e.\ $\tau$,
\begin{align}
G'(\tau)
&=
\int_{L^2_x}\frac{\dd}{\dd\tau}\left(\int_{L^2_x}\Phi(x)\,\dd\nu_\tau(x;u)\right)\dd\widehat\mu_{t_n}(u)\nonumber\\
&=
\int_{L^2_x}\left(\int_{L^2_x}\ip{D\Phi(x)}{v_\theta(x,\tau;u)}\,\dd\nu_\tau(x;u)\right)\dd\widehat\mu_{t_n}(u).
\label{eq:Gprime}
\end{align}

Define the joint law $\rho_\tau$ on $L^2_x\times L^2_x$ by two-stage sampling:
\begin{equation}\label{eq:rho}
\rho_\tau(\dd u,\dd x):=\dd\widehat\mu_{t_n}(u)\,\dd\nu_\tau(x;u).
\end{equation}
Then the second marginal of $\rho_\tau$ is $\widehat\mu_{t_n+\tau\Delta t}$:
\[
\rho_\tau(L^2_x,\dd x)=\int \dd\widehat\mu_{t_n}(u)\,\dd\nu_\tau(x;u)=\dd\widehat\mu_{t_n+\tau\Delta t}(x).
\]
Using \eqref{eq:rho}, rewrite \eqref{eq:Gprime} as
\begin{equation}\label{eq:Gprime-rho}
G'(\tau)=\int_{L^2_x\times L^2_x}\ip{D\Phi(x)}{v_\theta(x,\tau;u)}\,\dd\rho_\tau(u,x).
\end{equation}

Since $L^2_x$ is Polish, disintegrate $\rho_\tau$ with respect to its second marginal:
there exists a measurable family $\{\rho_\tau^x\}_{x\in L^2_x}$ of probability measures on $L^2_x$ such that
\[
\rho_\tau(\dd u,\dd x)=\rho_\tau^x(\dd u)\,\dd\widehat\mu_{t_n+\tau\Delta t}(x).
\]
Define
\[
\bar v_\tau(x):=\int_{L^2_x} v_\theta(x,\tau;u)\,\rho_\tau^x(\dd u).
\]
Then, by Fubini and the definition of disintegration,
\begin{align*}
\int_{L^2_x\times L^2_x}\ip{D\Phi(x)}{v_\theta(x,\tau;u)}\,\dd\rho_\tau(u,x)
&=\int_{L^2_x}\left(\int_{L^2_x}\ip{D\Phi(x)}{v_\theta(x,\tau;u)}\,\dd\rho_\tau^x(u)\right)\dd\widehat\mu_{t_n+\tau\Delta t}(x)\\
&=\int_{L^2_x}\ip{D\Phi(x)}{\bar v_\tau(x)}\,\dd\widehat\mu_{t_n+\tau\Delta t}(x).
\end{align*}
Combine with \eqref{eq:Gprime-rho} to obtain \eqref{eq:mixture-CE}.
The conditional expectation interpretation is exactly the definition of $\bar v_\tau$.
\end{proof}

\subsection{Sampler-induced regularity inputs compatible with LM compactness}

The compactness and identification results later require uniform versions of the LM hypotheses: uniform $L^2$ bounds,
uniform time-regularity, and a uniform structure-function modulus. These hypotheses are not asserted to hold for all
samplers; rather, they are treated as the analytic interface: they are either enforced/monitored during training or
verified empirically.

\begin{assumption}[Uniform LM inputs]\label{ass:LM-inputs}
A family $\{\widehat\mu^\Delta_\cdot\}_{\Delta>0}\subset L^1_t(\Pcal)$ satisfies:
\begin{enumerate}[label=(\roman*),leftmargin=2.2em]
\item (\emph{Uniform $L^2$ support}) There exists $M>0$ such that $\widehat\mu^\Delta_t(B_M(0))=1$ for all $\Delta$ and a.e.\ $t$.
\item (\emph{Uniform time-regularity}) The family is uniformly time-regular in the sense of \cref{def:time-regular}.
\item (\emph{Uniform structure-function modulus}) There exists a modulus $\omega$ such that
\[
S_r^2(\widehat\mu^\Delta_\cdot;T)\le \omega(r)\qquad\forall r>0,\ \forall \Delta.
\]
\end{enumerate}
\end{assumption}

\begin{remark}[Interpretation in learned surrogates]
Uniform $L^2$ support is an energy constraint. Structure-function control quantifies small-scale content and is directly
measurable from samples. Time-regularity is a mild temporal coupling property; later sections provide sufficient
conditions in terms of action bounds and superposition representations that are natural for sampler interpolations.
\end{remark}


\section{Approximation: estimates via structure functions}
\label{sec:approx}

This section isolates the \emph{resolution obstruction} in law space.
A finite-resolution model (e.g.\ a grid, a truncated spectral representation, or a network whose output
is effectively band-limited) can only represent modes up to some scale $K$. A common empirical phenomenon in probabilistic PDE surrogates is that “distributional scores improve on the grid” while small-scale fidelity degrades out of distribution or at longer horizons; this is reported across diffusion-based turbulence rollouts and large-scale probabilistic weather models \cite{kohl2023benchmarking, bulte2025probabilistic}.
The capacity–coverage decomposition (Theorem~\ref{thm:cap-cov}) makes precise what is otherwise treated heuristically: even perfect matching of resolved statistics leaves an unavoidable unresolved tail. Proposition~\ref{prop:tail-omega} and Corollary~\ref{cor:powerlaw-coverage} show that this tail is quantitatively controlled by structure-function moduli--exactly the same sample-computable quantities used in turbulence diagnostics--thereby turning “spectral fidelity / small-scale detail” (see \cite{armegioiu2025rectified, price2023gencast, oommen2024integrating}) into a bound that can be propagated through rollout analysis. \\

Hence, even if the model matches the \emph{resolved} (low-frequency) statistics perfectly, the mismatch in the
\emph{unresolved} (high-frequency) tail remains. We quantify this in $W_2$ on $L^2$-valued random fields.
Throughout we work on the $d$-torus $\mathbb{T}^d=(\mathbb{R}/2\pi\mathbb{Z})^d$ and set
\[
\|f\|_2 := \|f\|_{L^2(\mathbb{T}^d)}.
\]
Let $\mathcal{P}_2(L^2(\mathbb{T}^d))$ denote Borel probability measures on the Hilbert space
$L^2(\mathbb{T}^d)$ with finite second moment.

\subsection{Fourier projections and a Bernstein inequality}

For $u\in L^2(\mathbb{T}^d)$, write its Fourier series
\[
u(x)=\sum_{k\in\mathbb{Z}^d}\widehat u(k)e^{ik\cdot x},
\qquad
\widehat u(k)=\frac{1}{(2\pi)^d}\int_{\mathbb{T}^d}u(x)e^{-ik\cdot x}\,\mathrm{d}x.
\]
Parseval's identity reads
\[
\|u\|_2^2=(2\pi)^d\sum_{k\in\mathbb{Z}^d}|\widehat u(k)|^2.
\]

For $K\ge 1$, define the sharp Fourier projector
\[
P_{\le K}u := \sum_{|k|\le K}\widehat u(k)e^{ik\cdot x},
\qquad
P_{>K}:=\operatorname{Id}-P_{\le K}.
\]
Then $P_{\le K}$ is an orthogonal projection on $L^2(\mathbb{T}^d)$ and
\[
\|u\|_2^2=\|P_{\le K}u\|_2^2+\|P_{>K}u\|_2^2.
\]

\begin{lemma}[Bernstein for band-limited functions on $\mathbb{T}^d$]
\label{lem:bernstein}
There exists $C=C(d)$ such that for every $K\ge 1$ and every $u$ with $u=P_{\le K}u$,
\[
\|\nabla u\|_{L^\infty(\mathbb{T}^d)} \le C\,K^{1+d/2}\,\|u\|_2.
\]
\end{lemma}

\begin{proof}
Since $u=P_{\le K}u$, we have
\[
\nabla u(x)=\sum_{|k|\le K}(ik)\widehat u(k)e^{ik\cdot x}.
\]
Taking absolute values and using $|e^{ik\cdot x}|=1$,
\[
|\nabla u(x)|\le \sum_{|k|\le K}|k|\,|\widehat u(k)|,
\qquad\text{hence}\qquad
\|\nabla u\|_{L^\infty}\le \sum_{|k|\le K}|k|\,|\widehat u(k)|.
\]
Apply Cauchy--Schwarz:
\[
\sum_{|k|\le K}|k|\,|\widehat u(k)|
\le
\Big(\sum_{|k|\le K}|k|^2\Big)^{1/2}
\Big(\sum_{|k|\le K}|\widehat u(k)|^2\Big)^{1/2}.
\]
By Parseval, $\sum_{|k|\le K}|\widehat u(k)|^2 \le (2\pi)^{-d}\|u\|_2^2$.
For the first factor, use $|k|^2\le K^2$ and the lattice-point bound
$\#\{k\in\mathbb{Z}^d:\ |k|\le K\}\le C_d K^d$:
\[
\sum_{|k|\le K}|k|^2 \le K^2\cdot C_d K^d = C_d K^{d+2}.
\]
Taking square roots yields
\[
\|\nabla u\|_{L^\infty}\le C\,K^{1+d/2}\,\|u\|_2,
\]
with $C$ depending only on $d$.
\end{proof}

\subsection{$W_2$ distance to a projection}

For $\mu,\nu\in\mathcal{P}_2(L^2(\mathbb{T}^d))$, define the quadratic Wasserstein distance
\[
W_2(\mu,\nu)^2
:=\inf_{\pi\in\Gamma(\mu,\nu)}
\int_{L^2\times L^2}\|u-v\|_2^2\,\mathrm{d}\pi(u,v),
\]
where $\Gamma(\mu,\nu)$ is the set of couplings with marginals $\mu$ and $\nu$.

\begin{lemma}[Projection controls $W_2$]
\label{lem:projW2}
For every $\mu\in\mathcal{P}_2(L^2(\mathbb{T}^d))$ and every $K\ge 1$,
\[
W_2\big(\mu,(P_{\le K})_\#\mu\big)
\le
\Big(\int_{L^2}\|P_{>K}u\|_2^2\,\mathrm{d}\mu(u)\Big)^{1/2}.
\]
\end{lemma}

\begin{proof}
Define the measurable map $T:L^2\to L^2\times L^2$ by $T(u)=(u,P_{\le K}u)$
and let $\pi:=T_\#\mu$. Then the first marginal of $\pi$ is $\mu$ and the second marginal is
$(P_{\le K})_\#\mu$, so $\pi\in\Gamma(\mu,(P_{\le K})_\#\mu)$.
By definition of $W_2$,
\begin{align*}
W_2\big(\mu,(P_{\le K})_\#\mu\big)^2
&\le \int\|u-P_{\le K}u\|_2^2\,\mathrm{d}\mu(u)
= \int\|P_{>K}u\|_2^2\,\mathrm{d}\mu(u).
\end{align*}
Taking square roots gives the claim.
\end{proof}

\subsection{Capacity--coverage decomposition}

The next definition names the two quantities that will propagate through the rest of the paper.

\begin{definition}[Coverage tail and projected mismatch at resolution $K$]
\label{def:cov-train-errors}
For $\mu,\nu\in\mathcal{P}_2(L^2(\mathbb{T}^d))$ and $K\ge 1$, define
\[
\Tail_K(\mu)
:=\Big(\int_{L^2}\|P_{>K}u\|_2^2\,\mathrm{d}\mu(u)\Big)^{1/2},
\qquad
\Train_K(\mu,\nu)
:= W_2\big((P_{\le K})_\#\mu,\ (P_{\le K})_\#\nu\big).
\]
\end{definition}

\begin{theorem}[Capacity--coverage decomposition]
\label{thm:cap-cov}
For any $\mu,\nu\in\mathcal{P}_2(L^2(\mathbb{T}^d))$ and any $K\ge 1$,
\[
W_2(\mu,\nu)\ \le\ \Tail_K(\mu)\ +\ \Train_K(\mu,\nu)\ +\ \Tail_K(\nu).
\]
In particular, if $\nu$ is $K$-band-limited (i.e.\ $\Tail_K(\nu)=0$), then
\[
W_2(\mu,\nu)\ \le\ \Tail_K(\mu)\ +\ \Train_K(\mu,\nu).
\]
\end{theorem}

\begin{proof}
Insert the intermediate projected laws and apply the triangle inequality:
\begin{align*}
W_2(\mu,\nu)
&\le W_2\big(\mu,(P_{\le K})_\#\mu\big)
   +W_2\big((P_{\le K})_\#\mu,(P_{\le K})_\#\nu\big)
   +W_2\big((P_{\le K})_\#\nu,\nu\big).
\end{align*}
The middle term is $\Train_K(\mu,\nu)$ by definition.
For the first and third terms apply Lemma~\ref{lem:projW2} to $\mu$ and to $\nu$ respectively:
\[
W_2\big(\mu,(P_{\le K})_\#\mu\big)\le \Tail_K(\mu),
\quad
W_2\big((P_{\le K})_\#\nu,\nu\big)\le \Tail_K(\nu).
\]
Combining the inequalities yields the claim.
\end{proof}

\subsubsection{One-step specialization (PDE law vs.\ model law)}
Let $\mu_{t+\Delta t}=(S_{\Delta t})_\#\mu_t$ be the PDE pushforward law over one step and
let $\widehat\mu_{t+\Delta t}$ be the model one-step law at the same time.
Given $K\ge 1$, the projected mismatch is
\[
\varepsilon_{\mathrm{train}}(t;K):=\Train_K(\mu_{t+\Delta t},\widehat\mu_{t+\Delta t})
= W_2\big((P_{\le K})_\#\mu_{t+\Delta t},\ (P_{\le K})_\#\widehat\mu_{t+\Delta t}\big).
\]
Then Theorem~\ref{thm:cap-cov} gives the explicit one-step bound
\begin{equation}\label{eq:onestep-cap-cov}
W_2(\mu_{t+\Delta t},\widehat\mu_{t+\Delta t})
\le
\Tail_K(\mu_{t+\Delta t})+\varepsilon_{\mathrm{train}}(t;K)+\Tail_K(\widehat\mu_{t+\Delta t}).
\end{equation}
If $\widehat\mu_{t+\Delta t}$ is $K$-band-limited, then $\Tail_K(\widehat\mu_{t+\Delta t})=0$.

\subsection{Structure functions control spectral tails}

For $r\in\mathbb{R}^d$ with $|r|\le 1$ (identified with its class in $\mathbb{T}^d$), define increments
\[
\delta_r u(x):=u(x+r)-u(x).
\]
Assume a law-level second-order structure modulus: there exists a function
$\omega:[0,1]\to[0,\infty)$ such that
\begin{equation}\label{eq:omega-ass}
\int_{L^2}\|\delta_r u\|_2^2\,\mathrm{d}\mu(u)\ \le\ \omega(|r|)^2
\qquad\text{for all }|r|\le 1.
\end{equation}

\subsubsection{Littlewood--Paley dyadic blocks}
Fix radial cutoffs $\chi,\varphi\in C_c^\infty(\mathbb{R}^d)$ such that
\[
\chi(\xi)=1\ \text{for }|\xi|\le 1,\qquad \chi(\xi)=0\ \text{for }|\xi|\ge 2,
\qquad
\varphi(\xi):=\chi(\xi)-\chi(2\xi),
\]
so $\varphi$ is supported in $\{1/2\le|\xi|\le 2\}$.
Define Fourier multipliers on $\mathbb{T}^d$ by
\[
\widehat{\Delta_{-1}u}(k)=\chi(k)\widehat u(k),\qquad
\widehat{\Delta_j u}(k)=\varphi(2^{-j}k)\widehat u(k)\quad (j\ge 0).
\]
Then for each $k\neq 0$,
\[
\sum_{j\ge 0}\varphi(2^{-j}k)=1,
\]
and the supports have finite overlap: there exists $M\in\mathbb{N}$ (depending only on $\varphi$) such that
for each $k\neq 0$, the set $\{j\ge 0:\ \varphi(2^{-j}k)\neq 0\}$ has cardinality at most $M$.
In particular, there exist constants $0<c_*\le C_*<\infty$ such that for all $k\neq 0$,
\begin{equation}\label{eq:finite-overlap}
c_*\ \le\ \sum_{j\ge 0}\varphi(2^{-j}k)^2\ \le\ C_*.
\end{equation}

\subsubsection{Dyadic blocks are controlled by increments}

\begin{lemma}[Increment control of dyadic blocks]
\label{lem:inc-dyadic}
There exist constants $c\in(0,1)$ and $C<\infty$, depending only on $d$ and the chosen cutoffs,
such that for every $j\ge 0$ and every $u\in L^2(\mathbb{T}^d)$,
\[
\|\Delta_j u\|_2^2 \le
C\int_{|r|\le c\,2^{-j}}\frac{\|\delta_r u\|_2^2}{|r|^d}\,\mathrm{d}r.
\]
\end{lemma}

\begin{proof}
By Parseval,
\[
\|\Delta_j u\|_2^2=(2\pi)^d\sum_{k\in\mathbb{Z}^d}\varphi(2^{-j}k)^2\,|\widehat u(k)|^2.
\]
Also, $\widehat{\delta_r u}(k)=(e^{ik\cdot r}-1)\widehat u(k)$, hence by Parseval again,
\[
\|\delta_r u\|_2^2=(2\pi)^d\sum_{k\in\mathbb{Z}^d}|e^{ik\cdot r}-1|^2\,|\widehat u(k)|^2.
\]
Therefore, by Tonelli,
\begin{align*}
\int_{|r|\le c2^{-j}}\frac{\|\delta_r u\|_2^2}{|r|^d}\,\mathrm{d}r
&=(2\pi)^d\sum_{k\in\mathbb{Z}^d}
\Big(\int_{|r|\le c2^{-j}}\frac{|e^{ik\cdot r}-1|^2}{|r|^d}\,\mathrm{d}r\Big)\,|\widehat u(k)|^2.
\end{align*}
Thus it suffices to show that for all $k\in\mathbb{Z}^d$,
\begin{equation}\label{eq:mult-ineq-sec4}
\varphi(2^{-j}k)^2
\le
C\int_{|r|\le c2^{-j}}\frac{|e^{ik\cdot r}-1|^2}{|r|^d}\,\mathrm{d}r.
\end{equation}

Fix $k$ and $j$. If $\varphi(2^{-j}k)=0$, \eqref{eq:mult-ineq-sec4} is trivial. Assume $\varphi(2^{-j}k)\neq 0$.
By the support of $\varphi$, this implies $2^{j-1}\le |k|\le 2^{j+1}$.

We use the elementary inequality (valid for all $t\in\mathbb{R}$)
\[
|e^{it}-1|^2 = 2(1-\cos t)\ \ge\ \frac{2}{\pi^2}\,\min\{t^2,\pi^2\}.
\]
Choose $c\in(0,1)$ so small that $|k||r|\le \pi$ for all $|r|\le c2^{-j}$ and all $|k|\le 2^{j+1}$.
For instance, it suffices to take $c\le \pi/2$ because then
\[
|k||r|\le (2^{j+1})(c2^{-j}) = 2c \le \pi.
\]
With this choice, $\min\{(k\cdot r)^2,\pi^2\}=(k\cdot r)^2$, so
\[
|e^{ik\cdot r}-1|^2 \ge c_0 (k\cdot r)^2,
\qquad c_0:=\frac{2}{\pi^2}.
\]
Hence
\begin{align*}
\int_{|r|\le c2^{-j}}\frac{|e^{ik\cdot r}-1|^2}{|r|^d}\,\mathrm{d}r
&\ge c_0\int_{|r|\le c2^{-j}}\frac{(k\cdot r)^2}{|r|^d}\,\mathrm{d}r.
\end{align*}
Write $r=\rho\theta$ with $\rho\in(0,c2^{-j}]$ and $\theta\in S^{d-1}$.
Then $\mathrm{d}r=\rho^{d-1}\,\mathrm{d}\rho\,\mathrm{d}\theta$ and $|r|^{-d}=\rho^{-d}$, so
\begin{align*}
\int_{|r|\le c2^{-j}}\frac{(k\cdot r)^2}{|r|^d}\,\mathrm{d}r
&=\int_0^{c2^{-j}}\int_{S^{d-1}} (k\cdot(\rho\theta))^2\,\rho^{-d}\,\rho^{d-1}\,\mathrm{d}\theta\,\mathrm{d}\rho \\
&=\int_0^{c2^{-j}}\rho\left(\int_{S^{d-1}}(k\cdot\theta)^2\,\mathrm{d}\theta\right)\mathrm{d}\rho.
\end{align*}
By rotational symmetry,
\[
\int_{S^{d-1}}(k\cdot\theta)^2\,\mathrm{d}\theta
=|k|^2\int_{S^{d-1}}\theta_1^2\,\mathrm{d}\theta
=:c_1(d)\,|k|^2,
\qquad c_1(d)>0.
\]
Also,
\[
\int_0^{c2^{-j}}\rho\,\mathrm{d}\rho = \frac12 c^2 2^{-2j}.
\]
Therefore,
\[
\int_{|r|\le c2^{-j}}\frac{|e^{ik\cdot r}-1|^2}{|r|^d}\,\mathrm{d}r
\ge c_0\,c_1(d)\,\frac12 c^2\,|k|^2\,2^{-2j}.
\]
On the support of $\varphi(2^{-j}k)$, we have $|k|\ge 2^{j-1}$, hence $|k|^2 2^{-2j}\ge 2^{-2}$.
Thus the right-hand side is bounded below by a positive constant depending only on $d$ and the cutoffs.
Since $0\le \varphi\le \|\varphi\|_{L^\infty}$, we can choose $C$ so that
\eqref{eq:mult-ineq-sec4} holds for all such $k$. This proves the lemma.
\end{proof}

\subsubsection{Sharp spectral tails are controlled by dyadic tails}

\begin{lemma}[Sharp cutoff versus dyadic tail]
\label{lem:sharp-vs-dyadic}
Let $J\in\N$ and set $K:=2^J$.
There exists $C<\infty$ depending only on the cutoffs such that for every $u\in L^2(\T^d)$,
\[
\|P_{>K}u\|_2^2 \le C\sum_{j\ge J-1}\|\Delta_j u\|_2^2.
\]
Consequently, for every $K\ge 2$ and $J:=\lfloor \log_2 K\rfloor$ (so $2^J\le K<2^{J+1}$),
\[
\|P_{>K}u\|_2^2 \le C\sum_{j\ge J-1}\|\Delta_j u\|_2^2.
\]
\end{lemma}

\begin{proof}
\textbf{Step 1: dyadic $K=2^J$.}
By Parseval,
\[
\|P_{>2^J}u\|_2^2=(2\pi)^d\sum_{|k|>2^J}|\widehat u(k)|^2.
\]
Fix $k\neq 0$. The dyadic partition satisfies $\sum_{j\ge 0}\varphi(2^{-j}k)=1$, and by finite overlap
\eqref{eq:finite-overlap} we have $\sum_{j\ge 0}\varphi(2^{-j}k)^2\ge c_*$.

If $|k|>2^J$ and $\varphi(2^{-j}k)\neq 0$, then $|2^{-j}k|\le 2$, hence $|k|\le 2^{j+1}$, so $j\ge J-1$.
Therefore, for $|k|>2^J$,
\[
\sum_{j\ge J-1}\varphi(2^{-j}k)^2=\sum_{j\ge 0}\varphi(2^{-j}k)^2\ge c_*,
\]
and hence
\[
\mathbf 1_{\{|k|>2^J\}} \le c_*^{-1}\sum_{j\ge J-1}\varphi(2^{-j}k)^2.
\]
Multiply by $|\widehat u(k)|^2$ and sum over $k$:
\begin{align*}
\sum_{|k|>2^J}|\widehat u(k)|^2
&\le c_*^{-1}\sum_{k\in\Z^d}\sum_{j\ge J-1}\varphi(2^{-j}k)^2|\widehat u(k)|^2 \\
&= c_*^{-1}\sum_{j\ge J-1}\sum_{k\in\Z^d}\varphi(2^{-j}k)^2|\widehat u(k)|^2.
\end{align*}
By Parseval, $\sum_{k}\varphi(2^{-j}k)^2|\widehat u(k)|^2=(2\pi)^{-d}\|\Delta_j u\|_2^2$, so
\[
\|P_{>2^J}u\|_2^2 \le c_*^{-1}\sum_{j\ge J-1}\|\Delta_j u\|_2^2.
\]

\textbf{Step 2: general $K\ge 2$.}
Let $J=\lfloor\log_2 K\rfloor$, so $2^J\le K$. Then $\{|k|>K\}\subset\{|k|>2^J\}$, hence
\[
\|P_{>K}u\|_2^2=(2\pi)^d\sum_{|k|>K}|\widehat u(k)|^2
\le (2\pi)^d\sum_{|k|>2^J}|\widehat u(k)|^2
=\|P_{>2^J}u\|_2^2.
\]
Apply Step 1 to $\|P_{>2^J}u\|_2^2$.
\end{proof}

\subsubsection{Tail bound from the structure modulus}

\begin{proposition}[Spectral tail bound from a structure modulus]
\label{prop:tail-omega}
Assume \eqref{eq:omega-ass}. Then for every $j\ge 0$,
\[
\int_{L^2}\|\Delta_j u\|_2^2\,\mathrm{d}\mu(u)
\le
C\int_{|r|\le c\,2^{-j}}\frac{\omega(|r|)^2}{|r|^d}\,\mathrm{d}r.
\]
Consequently, for every $K\ge 2$ and $J:=\lfloor\log_2 K\rfloor$,
\[
\int_{L^2}\|P_{>K}u\|_2^2\,\mathrm{d}\mu(u)
\le
C\sum_{j\ge J-1}\int_{|r|\le c\,2^{-j}}\frac{\omega(|r|)^2}{|r|^d}\,\mathrm{d}r.
\]
\end{proposition}

\begin{proof}
Start from Lemma~\ref{lem:inc-dyadic} and integrate against $\mu$:
\[
\int\|\Delta_j u\|_2^2\,\mathrm{d}\mu(u)
\le
C\int\left(\int_{|r|\le c2^{-j}}\frac{\|\delta_r u\|_2^2}{|r|^d}\,\mathrm{d}r\right)\mathrm{d}\mu(u).
\]
By Tonelli,
\[
\int\left(\int_{|r|\le c2^{-j}}\frac{\|\delta_r u\|_2^2}{|r|^d}\,\mathrm{d}r\right)\mathrm{d}\mu(u)
=
\int_{|r|\le c2^{-j}}\frac{\left(\int\|\delta_r u\|_2^2\,\mathrm{d}\mu(u)\right)}{|r|^d}\,\mathrm{d}r.
\]
Apply \eqref{eq:omega-ass} to bound the inner expectation by $\omega(|r|)^2$, obtaining the first claim.

For the sharp tail, apply Lemma~\ref{lem:sharp-vs-dyadic} pointwise in $u$:
\[
\|P_{>K}u\|_2^2 \le C\sum_{j\ge J-1}\|\Delta_j u\|_2^2.
\]
Integrate against $\mu$ and use the first claim to bound each
$\int\|\Delta_j u\|_2^2\,\mathrm{d}\mu(u)$ by the increment integral.
\end{proof}

\subsection{Closed-form coverage rates under power laws}

The integral bound in Proposition~\ref{prop:tail-omega} becomes explicit under a power-law modulus.
This is the statement one typically uses in the main text.

\begin{corollary}[Power-law structure modulus $\Rightarrow$ algebraic coverage rate]
\label{cor:powerlaw-coverage}
Assume \eqref{eq:omega-ass} holds with $\omega(r)^2\le C_0 r^{2s}$ for some $s>0$ and all $r\in(0,1]$.
Then there exists $C=C(d,s)$ such that for all $K\ge 2$,
\[
\Tail_K(\mu)\le C\,\sqrt{C_0}\,K^{-s}.
\]
\end{corollary}

\begin{proof}
Let $J:=\lceil\log_2 K\rceil$. By Proposition~\ref{prop:tail-omega},
\[
\int\|P_{>K}u\|_2^2\,\mathrm{d}\mu(u)
\le
C\sum_{j\ge J-1}\int_{|r|\le c\,2^{-j}}\frac{\omega(|r|)^2}{|r|^d}\,\mathrm{d}r.
\]
Use $\omega(|r|)^2\le C_0 |r|^{2s}$:
\[
\int_{|r|\le c\,2^{-j}}\frac{\omega(|r|)^2}{|r|^d}\,\mathrm{d}r
\le
C_0\int_{|r|\le c\,2^{-j}}|r|^{2s-d}\,\mathrm{d}r.
\]
Writing $r=\rho\theta$, $\rho\in(0,c2^{-j}]$, $\theta\in S^{d-1}$, we have $\mathrm{d}r=\rho^{d-1}\mathrm{d}\rho\,\mathrm{d}\theta$, hence
\[
\int_{|r|\le c\,2^{-j}}|r|^{2s-d}\,\mathrm{d}r
=
|S^{d-1}|\int_0^{c2^{-j}}\rho^{2s-1}\,\mathrm{d}\rho
=
\frac{|S^{d-1}|}{2s}\,(c2^{-j})^{2s}.
\]
Therefore,
\[
\int_{|r|\le c\,2^{-j}}\frac{\omega(|r|)^2}{|r|^d}\,\mathrm{d}r
\le
C'(d,s)\,C_0\,2^{-2sj}.
\]
Summing the geometric series,
\[
\sum_{j\ge J-1}2^{-2sj}
\le C''(s)\,2^{-2sJ}.
\]
Since $J=\lceil\log_2 K\rceil$, we have $2^{J-1}<K\le 2^J$, hence $2^{-2sJ}\le C(s)\,K^{-2s}$, so
\[
\int\|P_{>K}u\|_2^2\,\mathrm{d}\mu(u)\le C(d,s)\,C_0\,K^{-2s}.
\]
Taking square roots yields $\Tail_K(\mu)\le C(d,s)\sqrt{C_0}\,K^{-s}$.
\end{proof}

\subsection{One-step $W_2$ bound from structure (band-limited model)}

\begin{corollary}[One-step $W_2$ bound: coverage + projected training]
\label{cor:onestep}
Fix $K\ge 2$. Let $\mu_{t+\Delta t}$ be the PDE one-step law and $\widehat\mu_{t+\Delta t}$ the model one-step law.
Assume $\widehat\mu_{t+\Delta t}$ is $K$-band-limited and that
\[
\varepsilon_{\mathrm{train}}(t;K)
=
W_2\big((P_{\le K})_\#\mu_{t+\Delta t},\ (P_{\le K})_\#\widehat\mu_{t+\Delta t}\big)
<\infty.
\]
Then
\[
W_2(\mu_{t+\Delta t},\widehat\mu_{t+\Delta t})
\le
\Tail_K(\mu_{t+\Delta t})+\varepsilon_{\mathrm{train}}(t;K).
\]
If moreover $\mu_{t+\Delta t}$ satisfies \eqref{eq:omega-ass} with $\omega(r)^2\le C_0 r^{2s}$ for some $s>0$,
then
\[
W_2(\mu_{t+\Delta t},\widehat\mu_{t+\Delta t})
\le
C(d,s)\sqrt{C_0}\,K^{-s}+\varepsilon_{\mathrm{train}}(t;K).
\]
\end{corollary}

\begin{proof}
Apply \eqref{eq:onestep-cap-cov}. Since $\widehat\mu_{t+\Delta t}$ is $K$-band-limited,
$\Tail_K(\widehat\mu_{t+\Delta t})=0$, yielding the first inequality.
The second inequality follows by bounding $\Tail_K(\mu_{t+\Delta t})$ with
Corollary~\ref{cor:powerlaw-coverage}.
\end{proof}

\section{Wasserstein stability via distance-weighted average strain}
\label{sec:stability}

This section proves the $W_2$ stability mechanism stated in the abstract. Rollout instability and horizon degradation are repeatedly emphasized in autoregressive diffusion-based fluid solvers and in probabilistic weather forecasting, but the analysis is typically empirical (sample metrics vs lead time) rather than a law-level stability statement \cite{kohl2023benchmarking, lippe2308pde, schiff2024dyslim}.\\

For two Euler solutions $u$ and $v$, the standard $L^2$ stability estimate bounds
$\|u(t)-v(t)\|_2$ by $\exp(\int_0^t \|\nabla v(\tau)\|_{L^\infty}\dd\tau)\|u(0)-v(0)\|_2$.
This is a worst-case bound: it uses the maximum strain of the reference solution.
For law evolutions, this is overly pessimistic because $W_2$ is governed by \emph{transport-relevant pairs} under an optimal coupling.
In Theorem~\ref{thm:W2-avg-strain} we show that the growth exponent can be written as an average of the strain weighted by the squared distance under the coupling.

\subsection{A pointwise $L^2$ stability identity}

Consider the incompressible Euler equations on $D=\T^d$:
\begin{equation}\label{eq:euler}
\partial_t u + (u\cdot\nabla)u + \nabla p = 0,
\qquad
\divg u = 0.
\end{equation}
Let $u$ and $v$ be (classical) solutions on $[0,T]$ with associated pressures $p$ and $q$.
Assume $v\in L^1(0,T;W^{1,\infty}(D))$ (this is automatic for smooth solutions on a fixed interval).
Set $w:=u-v$ and $\pi:=p-q$.

\begin{lemma}[$L^2$ difference identity and strain control]\label{lem:L2-diff-strain}
Let $u,v$ solve \eqref{eq:euler} as above and set $w=u-v$.
Define the rate-of-strain tensor of $v$ by
\[
\mathsf S(v):=\tfrac12(\nabla v+\nabla v^\top).
\]
Then for a.e.\ $t\in(0,T)$,
\begin{equation}\label{eq:L2-diff}
\frac12\frac{\dd}{\dd t}\|w(t)\|_2^2
=
-\int_D (w\otimes w):\nabla v(t)\,\dd x
=
-\int_D (w\otimes w):\mathsf S(v(t))\,\dd x,
\end{equation}
and in particular
\begin{equation}\label{eq:L2-diff-ineq}
\frac{\dd}{\dd t}\|w(t)\|_2^2
\le 2\int_D |\mathsf S(v(t,x))|\,|w(t,x)|^2\,\dd x.
\end{equation}
\end{lemma}

\begin{proof}
\textbf{Step 1: equation for the difference.}
Subtract the equations for $u$ and $v$:
\[
\partial_t w + (u\cdot\nabla)u - (v\cdot\nabla)v + \nabla\pi = 0.
\]
Rewrite the nonlinear difference by inserting $u=v+w$:
\[
(u\cdot\nabla)u - (v\cdot\nabla)v
= ((v+w)\cdot\nabla)(v+w) - (v\cdot\nabla)v
= (v\cdot\nabla)w + (w\cdot\nabla)v + (w\cdot\nabla)w.
\]
Hence
\begin{equation}\label{eq:w-eq}
\partial_t w + (v\cdot\nabla)w + (w\cdot\nabla)v + (w\cdot\nabla)w + \nabla\pi = 0.
\end{equation}

\textbf{Step 2: take the $L^2$ inner product with $w$.}
Multiply \eqref{eq:w-eq} by $w$ and integrate over $D$:
\[
\frac12\frac{\dd}{\dd t}\|w\|_2^2
+\int_D (v\cdot\nabla)w\cdot w\,\dd x
+\int_D (w\cdot\nabla)v\cdot w\,\dd x
+\int_D (w\cdot\nabla)w\cdot w\,\dd x
+\int_D \nabla\pi\cdot w\,\dd x
=0.
\]

\textbf{Step 3: cancel the transport terms.}
Using $\divg v=0$ and periodicity,
\[
\int_D (v\cdot\nabla)w\cdot w\,\dd x
=\frac12\int_D v\cdot\nabla(|w|^2)\,\dd x
=-\frac12\int_D (\divg v)\,|w|^2\,\dd x=0.
\]
Similarly,
\[
\int_D (w\cdot\nabla)w\cdot w\,\dd x
=\frac12\int_D w\cdot\nabla(|w|^2)\,\dd x
=-\frac12\int_D (\divg w)\,|w|^2\,\dd x=0
\]
because $\divg w=\divg u-\divg v=0$.

\textbf{Step 4: cancel the pressure term.}
By integration by parts and $\divg w=0$,
\[
\int_D \nabla\pi\cdot w\,\dd x = -\int_D \pi\,\divg w\,\dd x=0.
\]

\textbf{Step 5: identify the remaining term.}
We are left with
\[
\frac12\frac{\dd}{\dd t}\|w\|_2^2 + \int_D (w\cdot\nabla)v\cdot w\,\dd x=0.
\]
In index notation, $(w\cdot\nabla)v\cdot w = w_i(\partial_i v_j)w_j=(w\otimes w):\nabla v$.
This gives the first identity in \eqref{eq:L2-diff}.

For the symmetric part, decompose $\nabla v=\mathsf S(v)+\mathsf A(v)$ with $\mathsf A(v)=\tfrac12(\nabla v-\nabla v^\top)$.
Since $w\otimes w$ is symmetric and $\mathsf A(v)$ is antisymmetric, $(w\otimes w):\mathsf A(v)=0$.
Thus $(w\otimes w):\nabla v=(w\otimes w):\mathsf S(v)$, giving the second identity in \eqref{eq:L2-diff}.
Finally, \eqref{eq:L2-diff-ineq} follows from $-(w\otimes w):\mathsf S(v)\le | \mathsf S(v)|\,|w|^2$ pointwise and multiplying by $2$.
\end{proof}

\subsection{Distance-weighted strain exponent}

Define, for $u(t)\neq v(t)$,
\begin{equation}\label{eq:Lambda-pointwise}
\Lambda(u,v;t)
:=
\frac{\int_D |\mathsf S(v(t,x))|\,|u(t,x)-v(t,x)|^2\,\dd x}{\|u(t)-v(t)\|_2^2},
\qquad
\Lambda(u,v;t):=0 \text{ if }u(t)=v(t).
\end{equation}

\begin{corollary}[Pointwise stability with distance-weighted strain]\label{cor:pointwise-stab}
Under the hypotheses of Lemma~\ref{lem:L2-diff-strain},
\[
\|u(t)-v(t)\|_2
\le
\exp\Big(\int_0^t \Lambda(u,v;\tau)\,\dd\tau\Big)\,\|u(0)-v(0)\|_2
\qquad\forall t\in[0,T].
\]
\end{corollary}

\begin{proof}
From \eqref{eq:L2-diff-ineq} and the definition \eqref{eq:Lambda-pointwise},
\[
\frac{\dd}{\dd t}\|w(t)\|_2^2 \le 2\,\Lambda(u,v;t)\,\|w(t)\|_2^2.
\]
If $\|w(t)\|_2=0$ there is nothing to prove. Otherwise divide by $\|w(t)\|_2^2$ and integrate in time:
\[
\log\frac{\|w(t)\|_2^2}{\|w(0)\|_2^2}\le 2\int_0^t \Lambda(u,v;\tau)\,\dd\tau.
\]
Exponentiate and take square roots.
\end{proof}

\subsection{Law-level $W_2$ stability under an optimal coupling}

Let $S_t:L^2_x\to L^2_x$ denote the Euler solution map at time $t$ on a class of initial data for which the above calculations apply
(e.g.\ smooth divergence-free data on a short time interval). Assume $S_t$ is Borel on that class.
Let $\mu_0,\nu_0\in\Pcal_2(L^2_x)$ be supported on this class and define $\mu_t:=(S_t)_\#\mu_0$, $\nu_t:=(S_t)_\#\nu_0$.

For any coupling $\pi_0\in\Pi(\mu_0,\nu_0)$, define its pushforward coupling
\[
\pi_t := (S_t\times S_t)_\#\pi_0\in\Pi(\mu_t,\nu_t).
\]
Define the distance-weighted average strain under $\pi_t$ by
\begin{equation}\label{eq:Lambda-pi}
\overline\Lambda_{\pi_0}(t)
:=
\begin{cases}
\displaystyle
\frac{\int_{L^2_x\times L^2_x}\int_D |\mathsf S(v(t,x))|\,|u(t,x)-v(t,x)|^2\,\dd x\,\dd\pi_0(u_0,v_0)}
{\int_{L^2_x\times L^2_x}\|u(t)-v(t)\|_2^2\,\dd\pi_0(u_0,v_0)},
& \text{if the denominator $>0$,}\\[2.2ex]
0,& \text{otherwise.}
\end{cases}
\end{equation}
Equivalently, in terms of $\pi_t$,
\[
\overline\Lambda_{\pi_0}(t)
=
\frac{\int_{L^2\times L^2}\int_D |\mathsf S(v(x))|\,|u(x)-v(x)|^2\,\dd x\,\dd\pi_t(u,v)}
{\int_{L^2\times L^2}\|u-v\|_2^2\,\dd\pi_t(u,v)}
\quad\text{(when the denominator is nonzero)}.
\]

\begin{theorem}[$W_2$ stability with distance-weighted average strain]\label{thm:W2-avg-strain}
Let $\mu_t=(S_t)_\#\mu_0$ and $\nu_t=(S_t)_\#\nu_0$ as above.
Fix an \emph{optimal} coupling $\pi_0^\star\in\Pi(\mu_0,\nu_0)$ for $W_2(\mu_0,\nu_0)$ and define
$\pi_t^\star:=(S_t\times S_t)_\#\pi_0^\star$ and $\overline\Lambda^\star(t):=\overline\Lambda_{\pi_0^\star}(t)$.
Then for all $t\in[0,T]$,
\begin{equation}\label{eq:W2-avg-strain}
W_2(\mu_t,\nu_t)
\le
\exp\Big(\int_0^t \overline\Lambda^\star(\tau)\,\dd\tau\Big)\,W_2(\mu_0,\nu_0).
\end{equation}
\end{theorem}

\begin{proof}
\textbf{Step 1: control the transported second moment under the pushed coupling.}
For $\pi_0^\star$-a.e.\ pair $(u_0,v_0)$, let $u(t)=S_tu_0$ and $v(t)=S_tv_0$ be the corresponding solutions.
Apply Corollary~\ref{cor:pointwise-stab} pointwise and square it:
\[
\|u(t)-v(t)\|_2^2
\le
\exp\Big(2\int_0^t \Lambda(u,v;\tau)\,\dd\tau\Big)\,\|u_0-v_0\|_2^2.
\]
This is a correct bound, but we want the \emph{averaged} exponent. For that we instead integrate the differential inequality.

From Lemma~\ref{lem:L2-diff-strain} we have, for a.e.\ $t$,
\[
\frac{\dd}{\dd t}\|u(t)-v(t)\|_2^2
\le
2\int_D |\mathsf S(v(t,x))|\,|u(t,x)-v(t,x)|^2\,\dd x.
\]
Integrate this inequality with respect to $\pi_0^\star$ and use Tonelli:
\begin{equation}\label{eq:moment-ineq}
\frac{\dd}{\dd t}\int \|u(t)-v(t)\|_2^2\,\dd\pi_0^\star
\le
2\int \int_D |\mathsf S(v(t,x))|\,|u(t,x)-v(t,x)|^2\,\dd x\,\dd\pi_0^\star.
\end{equation}
By definition \eqref{eq:Lambda-pi}, the right-hand side equals
$2\,\overline\Lambda^\star(t)\,\int \|u(t)-v(t)\|_2^2\,\dd\pi_0^\star$ (whenever the denominator is nonzero; otherwise the inequality is trivial).
Thus we obtain a Gr\"onwall inequality for
\[
M(t):=\int \|u(t)-v(t)\|_2^2\,\dd\pi_0^\star:
\qquad
M'(t)\le 2\,\overline\Lambda^\star(t)\,M(t).
\]
Integrating gives
\begin{equation}\label{eq:M-growth}
M(t)\le \exp\Big(2\int_0^t \overline\Lambda^\star(\tau)\,\dd\tau\Big)\,M(0).
\end{equation}

\textbf{Step 2: relate $M(t)$ to $W_2(\mu_t,\nu_t)$.}
Since $\pi_t^\star$ is a coupling of $(\mu_t,\nu_t)$,
\[
W_2(\mu_t,\nu_t)^2 \le \int_{L^2\times L^2}\|u-v\|_2^2\,\dd\pi_t^\star(u,v)=M(t),
\]
where the equality is exactly the definition of $\pi_t^\star$ as the pushforward of $\pi_0^\star$.

Similarly, $M(0)=\int\|u_0-v_0\|_2^2\,\dd\pi_0^\star = W_2(\mu_0,\nu_0)^2$ because $\pi_0^\star$ is optimal.

Combine this with \eqref{eq:M-growth} and take square roots to obtain \eqref{eq:W2-avg-strain}.
\end{proof}

\begin{remark}[One-step version used in rollouts]\label{rem:one-step-from-avg-strain}
Applying Theorem~\ref{thm:W2-avg-strain} on an interval $[t_n,t_{n+1}]$ (time-shifted) yields
\[
W_2\big((S_{\Delta t})_\#\rho_1,(S_{\Delta t})_\#\rho_2\big)
\le
\exp\Big(\int_{t_n}^{t_{n+1}}\overline\Lambda^\star(\tau)\,\dd\tau\Big)\,W_2(\rho_1,\rho_2),
\]
where $\overline\Lambda^\star$ is computed from an optimal coupling of $\rho_1$ and $\rho_2$ at time $t_n$ pushed through the flow.
This is the promised ``distance-weighted average strain'' exponent: it depends on the strain along the coupled pairs,
weighted by their squared separation, not on $\|\nabla v\|_{L^\infty}$.
\end{remark}

\section{Rollout bounds}
\label{sec:rollout}

We bound the discrepancy between the PDE law rollout and the model law rollout over $N$ steps.
Write $t_n:=n\Delta t$.
Let $S_{\Delta t}:L^2(\T^d)\to L^2(\T^d)$ denote the PDE solution map over one step
(so that $u(t_{n+1})=S_{\Delta t}(u(t_n))$ at the level of states),
and let $\mathcal T_{\Delta t}:\Pcal_2(L^2)\to\Pcal_2(L^2)$ denote the model one-step map on laws.
We consider the two sequences of laws
\[
\mu_{t_{n+1}} := (S_{\Delta t})_\# \mu_{t_n},
\qquad
\widehat\mu_{t_{n+1}} := \mathcal T_{\Delta t}\,\widehat\mu_{t_n},
\qquad n\ge 0.
\]
Define the rollout discrepancy
\[
\delta_n := W_2(\mu_{t_n},\widehat\mu_{t_n}),\qquad n\ge 0.
\]

\subsection{Stability--defect splitting}

The basic mechanism is: (a) the PDE map propagates existing mismatch by a stability factor,
and (b) each step injects a fresh mismatch because $\mathcal T_{\Delta t}$ is not exactly
the PDE pushforward on the current model law.

\begin{lemma}[Rollout split: stability term + one-step defect]
\label{lem:rollout-split}
For every $n\ge 0$,
\[
\delta_{n+1}
\le
W_2\big((S_{\Delta t})_\#\mu_{t_n},(S_{\Delta t})_\#\widehat\mu_{t_n}\big)
+
W_2\big((S_{\Delta t})_\#\widehat\mu_{t_n},\,\mathcal T_{\Delta t}\widehat\mu_{t_n}\big).
\]
\end{lemma}

\begin{proof}
By the definitions of $\mu_{t_{n+1}}$ and $\widehat\mu_{t_{n+1}}$,
\[
\delta_{n+1}
=
W_2\big(\mu_{t_{n+1}},\widehat\mu_{t_{n+1}}\big)
=
W_2\big((S_{\Delta t})_\#\mu_{t_n},\,\mathcal T_{\Delta t}\widehat\mu_{t_n}\big).
\]
Insert the intermediate law $(S_{\Delta t})_\#\widehat\mu_{t_n}$ and apply the triangle inequality
for the metric $W_2$:
\begin{align*}
W_2\big((S_{\Delta t})_\#\mu_{t_n},\,\mathcal T_{\Delta t}\widehat\mu_{t_n}\big)
&\le
W_2\big((S_{\Delta t})_\#\mu_{t_n},\,(S_{\Delta t})_\#\widehat\mu_{t_n}\big)
+
W_2\big((S_{\Delta t})_\#\widehat\mu_{t_n},\,\mathcal T_{\Delta t}\widehat\mu_{t_n}\big).
\end{align*}
This is exactly the claimed inequality.
\end{proof}

\subsection{One-step stability of the PDE pushforward}

\begin{proposition}[One-step stability from the average-strain theorem]\label{prop:pde-stability-step}
Assume the hypotheses of Theorem~\ref{thm:W2-avg-strain} on the time window $[t_n,t_{n+1}]$ for the Euler solution map.
Then for any $\rho_1,\rho_2\in\Pcal_2(L^2_x)$ supported on the admissible data class,
\[
W_2\big((S_{\Delta t})_\#\rho_1,\,(S_{\Delta t})_\#\rho_2\big)
\le
\exp\!\Big(\int_{t_n}^{t_{n+1}}\overline\Lambda^\star(\tau)\,\dd\tau\Big)\,
W_2(\rho_1,\rho_2),
\]
where $\overline\Lambda^\star(\tau)$ is the distance-weighted average strain defined in \eqref{eq:Lambda-pi}
computed from an optimal coupling of $\rho_1$ and $\rho_2$ at time $t_n$, pushed through the Euler flow on $[t_n,t_{n+1}]$.
Define
\[
\alpha_n:=\int_{t_n}^{t_{n+1}}\overline\Lambda^\star(\tau)\,\dd\tau.
\]
\end{proposition}

\subsection{One-step defect and the closed recursion}

Define the \emph{one-step defect functional} of the model, evaluated on an input law $\rho$:
\[
\eta_{n+1}(\rho)
:=
W_2\big((S_{\Delta t})_\#\rho,\ \mathcal T_{\Delta t}\rho\big),
\qquad \rho\in\Pcal_2(L^2).
\]
By Lemma~\ref{lem:rollout-split} and Theorem~\ref{prop:pde-stability-step},
\begin{align*}
\delta_{n+1}
&\le
W_2\big((S_{\Delta t})_\#\mu_{t_n},(S_{\Delta t})_\#\widehat\mu_{t_n}\big)
+
W_2\big((S_{\Delta t})_\#\widehat\mu_{t_n},\,\mathcal T_{\Delta t}\widehat\mu_{t_n}\big) \\
&\le e^{\alpha_n}W_2(\mu_{t_n},\widehat\mu_{t_n}) + \eta_{n+1}(\widehat\mu_{t_n})
= e^{\alpha_n}\delta_n + \eta_{n+1}(\widehat\mu_{t_n}),
\qquad n\ge 0.
\end{align*}
Thus we obtain the recursion
\begin{equation}\label{eq:delta-recursion}
\delta_{n+1}\le e^{\alpha_n}\delta_n + \eta_{n+1}(\widehat\mu_{t_n}),
\qquad n\ge 0.
\end{equation}

\subsection{Uniform control of the injected defect (rollout class)}

To close \eqref{eq:delta-recursion} we require a uniform bound on the defect along the rollout.
We encode this by specifying a class of laws $\mathfrak C$ that contains both trajectories up to time $N$
and on which the one-step defect is uniformly bounded.

\begin{assumption}[Invariant rollout class and uniform one-step defect bound]
\label{ass:rollout-class}
Fix a horizon $N\in\N$. There exists a set $\mathfrak C\subset\Pcal_2(L^2)$ such that:
\begin{enumerate}[label=(\roman*),leftmargin=2.2em]
\item (\emph{invariance up to time $N$}) $\mu_{t_n}\in\mathfrak C$ and $\widehat\mu_{t_n}\in\mathfrak C$
for all $0\le n\le N$;
\item (\emph{uniform defect bound}) there exist nonnegative numbers $\varepsilon_{n+1}$, $0\le n\le N-1$, such that
\[
\sup_{\rho\in\mathfrak C}\eta_{n+1}(\rho)
=
\sup_{\rho\in\mathfrak C}W_2\big((S_{\Delta t})_\#\rho,\ \mathcal T_{\Delta t}\rho\big)
\le \varepsilon_{n+1}.
\]
\end{enumerate}
\end{assumption}

Under Assumption~\ref{ass:rollout-class}, \eqref{eq:delta-recursion} becomes the closed inequality
\begin{equation}\label{eq:closed-recursion}
\delta_{n+1}\le e^{\alpha_n}\delta_n + \varepsilon_{n+1},
\qquad 0\le n\le N-1.
\end{equation}

\subsection{Discrete Gr\"onwall with multiplicative stability}

\begin{lemma}[Discrete Gr\"onwall with variable coefficients]
\label{lem:discrete-gronwall}
Let $(\delta_n)_{n\ge 0}$ be nonnegative and assume that for $0\le n\le N-1$,
\[
\delta_{n+1}\le L_n\,\delta_n + \varepsilon_{n+1}
\]
with given numbers $L_n\ge 0$ and $\varepsilon_{n+1}\ge 0$.
Define the empty product by $\prod_{m=a}^{b}(\cdot):=1$ when $a>b$.
Then for every $N\ge 1$,
\[
\delta_N
\le
\Big(\prod_{m=0}^{N-1}L_m\Big)\delta_0
+
\sum_{j=1}^{N}\varepsilon_j\Big(\prod_{m=j}^{N-1}L_m\Big).
\]
\end{lemma}

\begin{proof}
We prove the claim by induction on $N$.

\emph{Base case $N=1$.}
The assumption gives $\delta_1\le L_0\delta_0+\varepsilon_1$, which matches the formula
since $\prod_{m=0}^{0}L_m=L_0$ and the sum has only the term
$\varepsilon_1\prod_{m=1}^{0}L_m=\varepsilon_1$.

\emph{Induction step.} Assume the formula holds for $N$.
Using the one-step inequality at time $N$,
\[
\delta_{N+1}\le L_N\delta_N+\varepsilon_{N+1}.
\]
Insert the induction hypothesis bound on $\delta_N$:
\begin{align*}
\delta_{N+1}
&\le L_N\Big[\Big(\prod_{m=0}^{N-1}L_m\Big)\delta_0
+
\sum_{j=1}^{N}\varepsilon_j\Big(\prod_{m=j}^{N-1}L_m\Big)\Big]+\varepsilon_{N+1} \\
&=
\Big(\prod_{m=0}^{N}L_m\Big)\delta_0
+
\sum_{j=1}^{N}\varepsilon_j\Big(\prod_{m=j}^{N}L_m\Big)
+\varepsilon_{N+1}.
\end{align*}
Since $\varepsilon_{N+1}=\varepsilon_{N+1}\prod_{m=N+1}^{N}L_m$ by the empty-product convention,
this is exactly the desired formula at level $N+1$.
\end{proof}

\begin{theorem}[Rollout bound under PDE stability and uniform defect control]
\label{thm:rollout}
Assume \cref{prop:pde-stability-step,ass:rollout-class} up to horizon $N$.
Then the rollout discrepancy satisfies
\[
\delta_N
\le
\exp\Big(\sum_{m=0}^{N-1}\alpha_m\Big)\,\delta_0
+
\sum_{j=1}^{N}\varepsilon_j\,
\exp\Big(\sum_{m=j}^{N-1}\alpha_m\Big).
\]
\end{theorem}

\begin{proof}
By \eqref{eq:closed-recursion}, we are in the setting of Lemma~\ref{lem:discrete-gronwall} with
$L_n=e^{\alpha_n}$. Therefore
\[
\delta_N
\le
\Big(\prod_{m=0}^{N-1}e^{\alpha_m}\Big)\delta_0
+
\sum_{j=1}^{N}\varepsilon_j\Big(\prod_{m=j}^{N-1}e^{\alpha_m}\Big).
\]
Using $\prod_{m=a}^{b}e^{\alpha_m}=\exp(\sum_{m=a}^{b}\alpha_m)$ gives the stated bound.
\end{proof}

\begin{corollary}[Constant-coefficient simplification]
\label{cor:rollout-constant}
If $\alpha_n\le \bar\alpha$ for all $n$ and $\varepsilon_n\le \bar\varepsilon$ for all $1\le n\le N$, then
\[
\delta_N
\le
e^{N\bar\alpha}\delta_0
+
\bar\varepsilon\sum_{k=0}^{N-1}e^{k\bar\alpha}
=
\begin{cases}
e^{N\bar\alpha}\delta_0 + \bar\varepsilon\,\dfrac{e^{N\bar\alpha}-1}{e^{\bar\alpha}-1},
& \bar\alpha>0,\\[1.2ex]
\delta_0 + N\bar\varepsilon,
& \bar\alpha=0.
\end{cases}
\]
\end{corollary}

\begin{proof}
Apply \cref{thm:rollout} and bound $\sum_{m=j}^{N-1}\alpha_m\le (N-j)\bar\alpha$, so
$\exp(\sum_{m=j}^{N-1}\alpha_m)\le e^{(N-j)\bar\alpha}$. Then
\[
\sum_{j=1}^{N}\varepsilon_j e^{\sum_{m=j}^{N-1}\alpha_m}
\le
\bar\varepsilon\sum_{j=1}^{N}e^{(N-j)\bar\alpha}
=\bar\varepsilon\sum_{k=0}^{N-1}e^{k\bar\alpha},
\]
and the finite geometric sum is explicit.
\end{proof}

\begin{remark}[Where the approximation bounds enter the defect]
\label{rem:defect-via-approx}
The nontrivial input in \cref{ass:rollout-class} is an upper bound on
\[
\eta_{n+1}(\rho)=W_2\big((S_{\Delta t})_\#\rho,\ \mathcal T_{\Delta t}\rho\big)
\]
that holds uniformly for $\rho\in\mathfrak C$.
This can be obtained from the approximation results of \cref{sec:approx} by applying the
capacity--coverage decomposition (Theorem~\ref{thm:cap-cov}) to the pair of laws
$\mu:=(S_{\Delta t})_\#\rho$ and $\nu:=\mathcal T_{\Delta t}\rho$.
For example, if $\mathcal T_{\Delta t}\rho$ is $K$-band-limited and one controls the projected mismatch
$\Train_K((S_{\Delta t})_\#\rho,\mathcal T_{\Delta t}\rho)$, then
\[
\eta_{n+1}(\rho)\le \Tail_K\big((S_{\Delta t})_\#\rho\big)+
\Train_K\big((S_{\Delta t})_\#\rho,\mathcal T_{\Delta t}\rho\big),
\]
and $\Tail_K((S_{\Delta t})_\#\rho)$ can be bounded via the structure-function tail estimates in
\cref{sec:approx} (e.g.\ Corollary~\ref{cor:powerlaw-coverage} under a power-law modulus).
\end{remark}

\section{Within-step time regularity in the LM sense from sampler paths}
\label{sec:regularity}

This section verifies the LM \emph{time-regularity} hypothesis (Definition~\ref{def:time-regular}) for sampler-induced law
curves. This input is needed, together with uniform energy control and a uniform structure-function modulus, to apply the
LM compactness theorem in the metric
\[
d_T(\mu,\nu)=\int_0^T W_1(\mu_t,\nu_t)\,\dd t
\]
and to pass to limits in the correlation-hierarchy identities in the identification argument of
Section~\ref{sec:identify}. Concretely, LM time-regularity asks for a measurable assignment of couplings
$(s,t)\mapsto \pi_{s,t}\in\Pi(\widehat\mu_s,\widehat\mu_t)$ whose expected displacement is controlled in a negative Sobolev
norm:
\begin{equation}\label{eq:LM-time-reg-goal-sec7}
\int_{L^2_x\times L^2_x}\|u-v\|_{H^{-L}}\,\dd\pi_{s,t}(u,v)\le C|t-s|
\qquad\text{for a.e.\ }s,t\in[0,T).
\end{equation}
In the sampler setting the required couplings are most naturally constructed from the within-step trajectories already
generated during sampling (internal-time ODEs, probability-flow ODEs): we package these trajectories into a global path
measure $\eta^{\Delta t}$ on $\Ccal([0,T];L^2_x)$, and then define the couplings as joint time marginals of
$\eta^{\Delta t}$. The increment bound \eqref{eq:LM-time-reg-goal-sec7} is then reduced to a bound on the expected physical
speed of the sampled paths; for internal-time samplers we record simple sufficient conditions for this speed bound in
terms of a time-change identity and a pointwise straightness criterion.

\subsection{Target: LM time-regularity and the path-space construction of couplings}

Recall Definition~\ref{def:time-regular}. We must construct a measurable assignment
$(s,t)\mapsto \pi_{s,t}\in\Pi(\widehat\mu_s,\widehat\mu_t)$ and constants $C>0$, $L\in\N$ such that for a.e.\ $s,t\in[0,T)$,
\begin{equation}\label{eq:LM-time-reg-goal-again-sec7}
\int_{L^2_x\times L^2_x}\|u-v\|_{H^{-L}}\,\dd\pi_{s,t}(u,v)\le C|t-s|.
\end{equation}
We will take $L=1$. On $\T^d$, by the Fourier definition of $H^{-1}$,
\begin{equation}\label{eq:Hminus1-le-L2-secReg}
\|f\|_{H^{-1}}\le \|f\|_{L^2}\qquad\forall f\in L^2(\T^d),
\end{equation}
since $(1+|k|^2)^{-1/2}\le 1$ for every Fourier mode $k\in\Z^d$.
Therefore, it suffices to construct $\pi_{s,t}$ and bound the expected $L^2$ displacement under $\pi_{s,t}$.

The path-space viewpoint provides the couplings in a canonical way. Let
\[
\Gamma:=\Ccal([0,T];L^2_x)
\]
and denote by
\[
E_t:\Gamma\to L^2_x,\qquad E_t(\gamma):=\gamma(t),
\]
the evaluation map at time $t$. If $\eta$ is a probability measure on $\Gamma$ with time marginals
$(E_t)_\#\eta=\widehat\mu_t$, then for each $s,t$ the pushforward
\[
(E_s,E_t)_\#\eta\in \Pcal(L^2_x\times L^2_x)
\]
is a coupling of $(\widehat\mu_s,\widehat\mu_t)$, and measurability of $(s,t)\mapsto (E_s,E_t)_\#\eta$ follows from
elementary Fubini-type arguments (proved below). Thus the main task becomes constructing such a path measure
$\eta=\eta^{\Delta t}$ from the sampler, and then estimating increments of sampled paths under $\eta^{\Delta t}$.

\subsection{Segment kernels on within-step path space}

Fix a step size $\Delta t>0$ and write $t_n:=n\Delta t$. For clarity we assume $T=N\Delta t$ with $N\in\N$
(the minor change when $T$ is not an integer multiple is purely notational).

\paragraph{Choice of path space.}
We work with paths valued in $L^2_x$. Since $L^2_x$ is separable Hilbert (hence Polish), the within-step path space
$\Ccal([0,\Delta t];L^2_x)$ is also Polish. Set
\[
\Gamma_{\Delta t}:=\Ccal([0,\Delta t];L^2_x)
\]
with its Borel $\sigma$-algebra. For $r\in[0,\Delta t]$, let $e_r:\Gamma_{\Delta t}\to L^2_x$ denote evaluation,
$e_r(\gamma)=\gamma(r)$.

\paragraph{Segment kernels.}
A \emph{segment kernel} is a Markov kernel
\[
\mathsf{Q}_{\Delta t}:L^2_x\times\mathcal B(\Gamma_{\Delta t})\to[0,1],
\]
meaning:
(i) for each $u\in L^2_x$, $A\mapsto \mathsf{Q}_{\Delta t}(u,A)$ is a probability measure on $\Gamma_{\Delta t}$, and
(ii) for each Borel set $A\subset\Gamma_{\Delta t}$, the map $u\mapsto \mathsf{Q}_{\Delta t}(u,A)$ is Borel measurable.

For each $r\in[0,\Delta t]$, define the intermediate-time state kernels
\begin{equation}\label{eq:intermediate-kernels}
K^{r}_{\Delta t}(u,\cdot):=(e_r)_\#\mathsf{Q}_{\Delta t}(u,\cdot)\in \Pcal(L^2_x).
\end{equation}

\begin{assumption}[Well-posed sampler segment kernel]\label{ass:segment-kernel}
The sampler provides a segment kernel $\mathsf{Q}_{\Delta t}$ such that:
\begin{enumerate}[label=(\roman*),leftmargin=2.2em]
\item (\emph{start at the input}) for every $u\in L^2_x$,
\begin{equation}\label{eq:segment-start}
K^0_{\Delta t}(u,\cdot)=\delta_u ;
\end{equation}
equivalently, $\mathsf{Q}_{\Delta t}(u,\{\gamma:\gamma(0)=u\})=1$.
\item (\emph{end law equals the one-step kernel}) the endpoint pushforward agrees with the one-step state kernel:
\begin{equation}\label{eq:segment-end}
K^{\Delta t}_{\Delta t}(u,\cdot)=K_{\Delta t}(u,\cdot)\qquad\forall u\in L^2_x.
\end{equation}
\item (\emph{within-step absolute continuity}) for every $u$, $\mathsf{Q}_{\Delta t}(u,\cdot)$ is concentrated on
$\gamma\in AC([0,\Delta t];L^2_x)$, so that $\dot\gamma(r)\in L^2_x$ exists for a.e.\ $r$.
\end{enumerate}
\end{assumption}

\begin{remark}[Where $\mathsf{Q}_{\Delta t}$ comes from in practice]
In rectified flows / flow matching, one samples an internal-time path $(X_\tau)_{\tau\in[0,1]}$ solving
$\dot X_\tau=v_\theta(X_\tau,\tau;u)$ and then sets $\gamma(r):=X_{r/\Delta t}$.
This produces $\gamma\in AC([0,\Delta t];L^2_x)$ and therefore defines $\mathsf{Q}_{\Delta t}(u,\cdot)$.
For probability-flow ODE sampling in diffusion, the same construction holds with $v_\theta$ replaced by the PF drift.
\end{remark}

\subsection{Constructing a global rollout path measure by concatenation}

The segment kernel describes one physical step. To obtain couplings between \emph{arbitrary} physical times
$s,t\in[0,T]$, we construct a \emph{global} path measure describing the entire rollout. Let $\Gamma:=\Ccal([0,T];L^2_x)$ with evaluation maps $E_t(\gamma)=\gamma(t)$.

\paragraph{Concatenation map.}
For a tuple of segments $\bm\gamma=(\gamma^0,\dots,\gamma^{N-1})\in(\Gamma_{\Delta t})^N$, define
\begin{equation}\label{eq:concat-map}
(\Concat(\bm\gamma))(t):=\gamma^n(t-t_n)\qquad\text{for }t\in[t_n,t_{n+1}],\ \ n=0,\dots,N-1.
\end{equation}
Because each $\gamma^n$ is continuous in $L^2_x$, the right-hand side defines an $L^2$-continuous path on each interval.
To obtain global continuity at the junction times, we will ensure that the endpoint of $\gamma^n$ equals the start of $\gamma^{n+1}$
almost surely. This matching is exactly why we define the intermediate states recursively by $u_{n+1}=\gamma^n(\Delta t)$.

\begin{proposition}[Global rollout path measure]\label{prop:global-path-measure}
Assume Assumption~\ref{ass:segment-kernel}. Let $\widehat\mu_0\in\Pcal_2(L^2_x)$.
Then there exists a probability measure $\eta^{\Delta t}\in\Pcal(\Gamma)$ such that for every $n$ and $r\in[0,\Delta t]$,
\begin{equation}\label{eq:marginal-identity}
(E_{t_n+r})_\#\eta^{\Delta t}=\widehat\mu_{t_n+r},
\qquad
\widehat\mu_{t_n+r}(\cdot)=\int_{L^2_x}K^r_{\Delta t}(u,\cdot)\,\dd\widehat\mu_{t_n}(u),
\qquad
\widehat\mu_{t_{n+1}}=\mathcal T_{\Delta t}\widehat\mu_{t_n}.
\end{equation}
Moreover, $\eta^{\Delta t}$-a.e.\ path is absolutely continuous on each interval $[t_n,t_{n+1}]$.
\end{proposition}

\begin{proof}
We build a probability measure on segment tuples by iterated conditioning and then push it forward by $\Concat$.

\smallskip
\noindent\textbf{Step 1: define the probability space at the level of cylinders.}
Let $\mathcal A$ be the algebra of cylinder sets in
$L^2_x\times(\Gamma_{\Delta t})^N$ of the form
\[
A_0\times B_0\times\cdots\times B_{N-1},
\qquad A_0\in\mathcal B(L^2_x),\ \ B_n\in\mathcal B(\Gamma_{\Delta t}).
\]
For such a cylinder set define
\begin{equation}\label{eq:cylinder-measure}
\begin{aligned}
\mathbf P(A_0\times B_0\times\cdots\times B_{N-1})
:=
\int_{A_0}\widehat\mu_0(\dd u_0)
\int_{B_0}\mathsf Q_{\Delta t}(u_0,\dd\gamma^0)
\int_{B_1}\mathsf Q_{\Delta t}(u_1,\dd\gamma^1)\cdots
\int_{B_{N-1}}\mathsf Q_{\Delta t}(u_{N-1},\dd\gamma^{N-1}),
\end{aligned}
\end{equation}
where the intermediate states are defined recursively by
\begin{equation}\label{eq:state-recursion}
u_{n+1}:=\gamma^n(\Delta t)\in L^2_x.
\end{equation}
This recursion is meaningful because $\gamma^n(\Delta t)\in L^2_x$ and the map $\gamma\mapsto\gamma(\Delta t)$ is Borel on
$\Gamma_{\Delta t}=\Ccal([0,\Delta t];L^2_x)$.

\smallskip
\noindent\textbf{Step 2: measurability of the iterated integrand.}
We must check that the function being integrated at each stage is measurable so the iterated integral makes sense.
Fix $B\in\mathcal B(\Gamma_{\Delta t})$. By the kernel property, $u\mapsto \mathsf Q_{\Delta t}(u,B)$ is Borel.
Also, $(u,\gamma)\mapsto \gamma(\Delta t)$ is Borel, hence the map
\[
(u,\gamma)\mapsto \mathsf Q_{\Delta t}(\gamma(\Delta t),B)
\]
is Borel as a composition of Borel maps. By induction, the integrand in \eqref{eq:cylinder-measure} is measurable and nonnegative,
so the iterated integral is well-defined.

\smallskip
\noindent\textbf{Step 3: $\mathbf P$ is a pre-measure on $\mathcal A$.}
If we fix all cylinder coordinates except one $B_m$ and split $B_m$ into a disjoint union,
countable additivity of $\mathsf Q_{\Delta t}(u,\cdot)$ in its second argument implies countable additivity of $\mathbf P$ in that
coordinate after integrating the remaining coordinates. Since $\mathcal A$ is generated by finite intersections and such
coordinate-wise decompositions, $\mathbf P$ is a pre-measure on $\mathcal A$.
Moreover, taking $A_0=L^2_x$ and $B_n=\Gamma_{\Delta t}$ gives $\mathbf P=1$, so it has total mass $1$.

\smallskip
\noindent\textbf{Step 4: extend $\mathbf P$ to the full product $\sigma$-algebra.}
The algebra $\mathcal A$ generates the product Borel $\sigma$-algebra on $L^2_x\times(\Gamma_{\Delta t})^N$.
By Carath\'eodory's extension theorem, $\mathbf P$ extends uniquely to a probability measure (still denoted $\mathbf P$)
on the full product $\sigma$-algebra.

\smallskip
\noindent\textbf{Step 5: push forward to obtain a global path measure.}
Let $\Pi$ be the projection $(u_0,\gamma^0,\dots,\gamma^{N-1})\mapsto(\gamma^0,\dots,\gamma^{N-1})$ and set
$\widetilde\eta^{\Delta t}:=\Pi_\#\mathbf P$ on $(\Gamma_{\Delta t})^N$.
Define the global path measure
\[
\eta^{\Delta t}:=\Concat_\#\widetilde\eta^{\Delta t}\in\Pcal(\Gamma).
\]
By construction of $u_{n+1}=\gamma^n(\Delta t)$ and Assumption~\ref{ass:segment-kernel}(i) (start at the input),
the next segment $\gamma^{n+1}$ satisfies $\gamma^{n+1}(0)=u_{n+1}=\gamma^n(\Delta t)$ almost surely.
Hence the concatenated path is continuous at junction times, so $\Concat(\bm\gamma)\in\Gamma$ $\widetilde\eta^{\Delta t}$-a.s.

\smallskip
\noindent\textbf{Step 6: compute the marginals.}
Fix $n$ and $r\in[0,\Delta t]$. For bounded measurable $\Phi:L^2_x\to\R$,
\[
\int \Phi\,\dd(E_{t_n+r})_\#\eta^{\Delta t}
=
\int \Phi\big((\Concat(\bm\gamma))(t_n+r)\big)\,\dd\widetilde\eta^{\Delta t}(\bm\gamma)
=
\int \Phi(\gamma^n(r))\,\dd\widetilde\eta^{\Delta t}(\bm\gamma),
\]
by the definition of $\Concat$. Unfolding $\widetilde\eta^{\Delta t}$ from \eqref{eq:cylinder-measure} shows:
first $u_n$ has distribution $\widehat\mu_{t_n}$ (by iterating the endpoint kernel),
then $\gamma^n\sim\mathsf Q_{\Delta t}(u_n,\cdot)$, hence $\gamma^n(r)\sim K^r_{\Delta t}(u_n,\cdot)$.
Therefore
\[
\int \Phi\,\dd(E_{t_n+r})_\#\eta^{\Delta t}
=
\int_{L^2_x}\left(\int_{L^2_x}\Phi(w)\,\dd K^r_{\Delta t}(u,w)\right)\dd\widehat\mu_{t_n}(u),
\]
which is exactly \eqref{eq:marginal-identity}. This proves the first part.
Finally, since each segment is $AC([0,\Delta t];L^2_x)$ by Assumption~\ref{ass:segment-kernel}(iii),
the concatenated path is absolutely continuous on each $[t_n,t_{n+1}]$.
\end{proof}

\subsection{Canonical couplings and measurability}

\begin{definition}[Canonical couplings]\label{def:canonical-couplings}
For $s,t\in[0,T]$, define
\begin{equation}\label{eq:def-pist}
\pi^{\Delta t}_{s,t}:=(E_s,E_t)_\#\eta^{\Delta t}\in\Pcal(L^2_x\times L^2_x).
\end{equation}
\end{definition}

\begin{lemma}[Coupling property and measurability]\label{lem:coupling-measurable}
For all $s,t\in[0,T]$, $\pi^{\Delta t}_{s,t}\in\Pi(\widehat\mu_s,\widehat\mu_t)$.
Moreover, for every bounded continuous $\Psi:L^2_x\times L^2_x\to\R$, the map
\[
(s,t)\longmapsto \int_{L^2_x\times L^2_x}\Psi(u,v)\,\dd\pi^{\Delta t}_{s,t}(u,v)
\]
is measurable on $[0,T]^2$.
\end{lemma}

\begin{proof}
\textbf{Step 1: $\pi^{\Delta t}_{s,t}$ is a coupling.}
By definition, the first marginal of $\pi^{\Delta t}_{s,t}$ is $(E_s)_\#\eta^{\Delta t}=\widehat\mu_s$ and the second marginal is
$(E_t)_\#\eta^{\Delta t}=\widehat\mu_t$. Hence $\pi^{\Delta t}_{s,t}\in\Pi(\widehat\mu_s,\widehat\mu_t)$.

\smallskip
\noindent\textbf{Step 2: measurability in $(s,t)$.}
Fix bounded continuous $\Psi$. Let $\Gamma\sim\eta^{\Delta t}$ be the canonical random path.
Then, by definition of pushforward,
\[
\int \Psi\,\dd\pi^{\Delta t}_{s,t}=\E\big[\Psi(\Gamma(s),\Gamma(t))\big].
\]
For each sample path $\Gamma(\cdot)\in\Ccal([0,T];L^2_x)$, the map $(s,t)\mapsto(\Gamma(s),\Gamma(t))$ is continuous (hence Borel)
from $[0,T]^2$ into $L^2_x\times L^2_x$. Composing with continuous $\Psi$ yields a measurable function of $(s,t)$ for each path,
and taking expectation preserves measurability (Fubini/Tonelli applies since the function is bounded).
\end{proof}

\subsection{LM time-regularity from a uniform expected-speed bound}

\begin{assumption}[Uniform expected physical speed]\label{ass:uniform-phys-speed}
There exists $C_{\mathrm{spd}}>0$ such that under $\eta^{\Delta t}$ the canonical path $\Gamma$ is absolutely continuous in $L^2_x$ and
\begin{equation}\label{eq:phys-speed-assumption}
\E\|\dot\Gamma(t)\|_{L^2_x}\le C_{\mathrm{spd}}\qquad\text{for a.e.\ }t\in[0,T].
\end{equation}
\end{assumption}

\begin{theorem}[LM time-regularity from expected speed]\label{thm:LM-time-regularity}
Assume \cref{prop:global-path-measure,ass:uniform-phys-speed}. Then $\widehat\mu_\cdot$ is time-regular in the sense of
Definition~\ref{def:time-regular} with $L=1$ and constant $C_{\mathrm{spd}}$: for a.e.\ $s,t\in[0,T]$,
\begin{equation}\label{eq:LM-time-reg-conclusion}
\int_{L^2_x\times L^2_x}\|u-v\|_{H^{-1}}\,\dd\pi^{\Delta t}_{s,t}(u,v)\le C_{\mathrm{spd}}|t-s|.
\end{equation}
\end{theorem}

\begin{proof}
Fix $s,t\in[0,T]$ with $s\le t$.

\smallskip
\noindent\textbf{Step 1: reduce to an $L^2$ increment.}
By \eqref{eq:Hminus1-le-L2-secReg} and the definition of $\pi^{\Delta t}_{s,t}$,
\[
\int \|u-v\|_{H^{-1}}\,\dd\pi^{\Delta t}_{s,t}(u,v)
=
\E\|\Gamma(t)-\Gamma(s)\|_{H^{-1}}
\le
\E\|\Gamma(t)-\Gamma(s)\|_{L^2}.
\]

\smallskip
\noindent\textbf{Step 2: control the increment by integrating the speed.}
Since $\Gamma$ is absolutely continuous on each interval $[t_n,t_{n+1}]$,
write $[s,t]$ as a finite union of subintervals contained in step intervals and sum:
\[
\Gamma(t)-\Gamma(s)=\sum_{\ell}\int_{a_\ell}^{b_\ell}\dot\Gamma(r)\,\dd r
\quad\text{in }L^2_x,
\]
where each $[a_\ell,b_\ell]\subset[t_{n(\ell)},t_{n(\ell)+1}]$.

Take norms and use the triangle inequality for Bochner integrals:
\[
\|\Gamma(t)-\Gamma(s)\|_{L^2}\le \int_s^t \|\dot\Gamma(r)\|_{L^2}\,\dd r.
\]
Taking expectation and applying Tonelli yields
\[
\E\|\Gamma(t)-\Gamma(s)\|_{L^2}\le \int_s^t \E\|\dot\Gamma(r)\|_{L^2}\,\dd r.
\]

\smallskip
\noindent\textbf{Step 3: insert the uniform expected-speed bound.}
By Assumption~\ref{ass:uniform-phys-speed},
$\E\|\dot\Gamma(r)\|_{L^2}\le C_{\mathrm{spd}}$ for a.e.\ $r$, hence
\[
\E\|\Gamma(t)-\Gamma(s)\|_{L^2}\le \int_s^t C_{\mathrm{spd}}\,\dd r = C_{\mathrm{spd}}(t-s).
\]
Combining the steps gives \eqref{eq:LM-time-reg-conclusion}.
\end{proof}

\subsection{Sufficient conditions in sampler quantities}

The expected-speed assumption is an interface condition: it is what prevents the time-change from internal sampler time
to physical time from introducing a factor $1/\Delta t$ in the LM constant. We now make this scaling explicit.

\subsubsection{Time-change lemma}

\begin{lemma}[Time change]\label{lem:time-change-speed}
Suppose a segment $\gamma\in AC([0,\Delta t];L^2_x)$ is obtained from an internal-time path
$X\in AC([0,1];L^2_x)$ by the reparametrization $\gamma(r)=X_{r/\Delta t}$.
Assume $\dot X_\tau=V_\tau$ for a.e.\ $\tau\in[0,1]$. Then for a.e.\ $r\in(0,\Delta t)$,
\begin{equation}\label{eq:time-change-derivative-secReg}
\dot\gamma(r)=\frac{1}{\Delta t}\,V_{r/\Delta t}.
\end{equation}
\end{lemma}

\begin{proof}
Since $X\in AC([0,1];L^2_x)$,
\[
X_\tau = X_0 + \int_0^\tau V_\sigma\,\dd\sigma \quad\text{in }L^2_x.
\]
Substitute $\tau=r/\Delta t$:
\[
\gamma(r)=X_{r/\Delta t} = X_0 + \int_0^{r/\Delta t} V_\sigma\,\dd\sigma.
\]
Differentiate with respect to $r$ at points where $\tau\mapsto X_\tau$ is differentiable (a.e.\ $\tau$, hence a.e.\ $r$).
By the chain rule in Banach spaces,
\[
\dot\gamma(r)=\frac{\dd}{\dd r}X_{r/\Delta t}=\frac{1}{\Delta t}\dot X_{r/\Delta t}=\frac{1}{\Delta t}V_{r/\Delta t}.
\]
\end{proof}

\subsubsection{From internal-time velocity scaling to physical expected speed}

\begin{corollary}[Internal-time speed scaling implies \cref{ass:uniform-phys-speed}]
\label{cor:internal-speed-sufficient}
Assume there exists $C_{\mathrm{spd}}>0$ such that for each step $n$,
\begin{equation}\label{eq:internal-speed-sufficient}
\esssup_{\tau\in[0,1]}\E\big[\|V^{(n)}_\tau\|_{L^2}\,\big|\,U^{(n)}=u\big]\le C_{\mathrm{spd}}\,\Delta t
\qquad\text{for $\widehat\mu_{t_n}$-a.e.\ }u.
\end{equation}
Then Assumption~\ref{ass:uniform-phys-speed} is verified with the same $C_{\mathrm{spd}}$.
\end{corollary}

\begin{proof}
Fix a step $n$ and condition on the input state $U^{(n)}=u$.
Let $\gamma$ be the sampled segment and write it as a time-change of an internal path $X$ with velocity $V$.
By Lemma~\ref{lem:time-change-speed}, for a.e.\ $r\in(0,\Delta t)$,
\[
\E\big[\|\dot\gamma(r)\|_{L^2}\,\big|\,U^{(n)}=u\big]
=
\frac{1}{\Delta t}\,\E\big[\|V_{r/\Delta t}\|_{L^2}\,\big|\,U^{(n)}=u\big]
\le
\frac{1}{\Delta t}\,(C_{\mathrm{spd}}\Delta t)
=
C_{\mathrm{spd}},
\]
using \eqref{eq:internal-speed-sufficient}. Integrating out $u$ shows $\E\|\dot\gamma(r)\|_{L^2}\le C_{\mathrm{spd}}$ for a.e.\ $r$.
Since the global path $\Gamma$ is obtained by concatenating such segments, the same bound holds for a.e.\ physical time $t\in[0,T]$,
which is exactly \eqref{eq:phys-speed-assumption}.
\end{proof}

\subsubsection{Straightness: bounding the internal-time speed by chord + residual}

\begin{lemma}[Chord + pointwise straightness implies internal speed scaling]\label{lem:straightness-speed}
Let $X^{(n)}$ be an internal-time path on $[0,1]$ with velocity $V^{(n)}_\tau$ and define the chord
$D^{(n)}:=X^{(n)}_1-X^{(n)}_0$ and the residual $R^{(n)}_\tau:=V^{(n)}_\tau-D^{(n)}$.
Assume there exist constants $C_{\mathrm{ch}},C_{\mathrm{str}}\ge 0$ such that for $\widehat\mu_{t_n}$-a.e.\ input $u$,
\[
\E[\|D^{(n)}\|_{L^2}\mid U^{(n)}=u]\le C_{\mathrm{ch}}\Delta t,
\qquad
\esssup_{\tau\in[0,1]}\E[\|R^{(n)}_\tau\|_{L^2}^2\mid U^{(n)}=u]\le C_{\mathrm{str}}(\Delta t)^2.
\]
Then \eqref{eq:internal-speed-sufficient} holds with $C_{\mathrm{spd}}=C_{\mathrm{ch}}+\sqrt{C_{\mathrm{str}}}$.
\end{lemma}

\begin{proof}
Fix $\tau\in[0,1]$ and condition on $U^{(n)}=u$. Since $V^{(n)}_\tau=D^{(n)}+R^{(n)}_\tau$,
\[
\|V^{(n)}_\tau\|_{L^2}\le \|D^{(n)}\|_{L^2}+\|R^{(n)}_\tau\|_{L^2}.
\]
Take conditional expectation:
\[
\E[\|V^{(n)}_\tau\|_{L^2}\mid U^{(n)}=u]
\le
\E[\|D^{(n)}\|_{L^2}\mid U^{(n)}=u]
+
\E[\|R^{(n)}_\tau\|_{L^2}\mid U^{(n)}=u].
\]
Apply Cauchy--Schwarz to the residual term:
\[
\E[\|R^{(n)}_\tau\|_{L^2}\mid U^{(n)}=u]
\le
\Big(\E[\|R^{(n)}_\tau\|_{L^2}^2\mid U^{(n)}=u]\Big)^{1/2}
\le \sqrt{C_{\mathrm{str}}}\,\Delta t.
\]
Combine with the chord bound to obtain
\[
\E[\|V^{(n)}_\tau\|_{L^2}\mid U^{(n)}=u]\le (C_{\mathrm{ch}}+\sqrt{C_{\mathrm{str}}})\Delta t.
\]
Taking the essential supremum in $\tau$ yields \eqref{eq:internal-speed-sufficient} with $C_{\mathrm{spd}}=C_{\mathrm{ch}}+\sqrt{C_{\mathrm{str}}}$.
\end{proof}

\begin{remark}[Why the straightness bound is stated pointwise in $\tau$]
LM time-regularity requires a \emph{linear} bound in $|t-s|$. After the time-change
$\dot\gamma(r)=\Delta t^{-1}V_{r/\Delta t}$, a pointwise-in-$\tau$ control of $\E\|V_\tau\|$ yields a pointwise-in-$t$ control of
$\E\|\dot\gamma(t)\|$, hence a linear modulus. An integrated straightness functional $\int_0^1\E\|R_\tau\|^2\,\dd\tau$ controls only an
$L^2_\tau$ average and, by itself, yields at best a H\"older modulus in $t$ after Cauchy--Schwarz.
\end{remark}

\section{Path-space tightness from action bounds (H\"older modulus in $H^{-1}$)}
\label{sec:superposition}

Sections~\ref{sec:sampler} and~\ref{sec:regularity} use sampler-generated within-step trajectories to construct
\emph{canonical couplings} and verify the \emph{linear} LM time-regularity bound required for compactness in $d_T$.
In practice, however, one often has access to weaker, more ``energetic'' controls on interpolations--namely action bounds
of the form $\int_0^T\!\int |v_t|^p\,\dd\mu_t\,\dd t<\infty$ with $p>1$--either because the sampler interpolation satisfies
a continuity equation (Section~\ref{sec:sampler}) or because one controls integrated speed along sampled trajectories.

This section records the complementary consequence of such action bounds: they yield only a \emph{H\"older} modulus in time
of order $|t-s|^{1-1/p}$ for increments. Combined with compactness of $L^2$ balls in $H^{-1}(\T^d)$, this gives
tightness of the associated path measures in $\Ccal([0,T];H^{-1})$ and hence subsequential convergence on path space.
We emphasize: this H\"older control does \emph{not} replace LM time-regularity (which is linear in $|t-s|$), but it is
useful for extracting limits of sampler-induced \emph{path measures} and for organizing coupling constructions when passing
to limits (cf.\ the closure step used later in Section~\ref{sec:identify}).

\paragraph{No flow assumption: representation of continuity equations.}
To avoid assuming \emph{a priori} that a characteristic flow exists, we invoke the representation theory of continuity
equations: a narrowly continuous solution of the continuity equation with a Borel velocity field admits a (measurable)
characteristic representation, and $L^p$-integrability of the velocity implies that the velocity coincides with the time
derivative of characteristics in the $L^p$ sense. We use Proposition~8.1.8 in \cite{ambrosio2004gradient} as a black box.
In our setting, this is applied at finite resolution (e.g.\ after projection to a grid/spectral truncation
$L^2_{x_\Delta}\simeq\R^{N_\Delta}$); after reconstruction, the resulting paths live in $L^2(\T^d)$ and we then measure
increments in $H^{-1}$.

\subsection{From the continuity equation to a path measure}

We state a convenient consequence of \cite[Proposition~8.1.8]{ambrosio2004gradient} in the Euclidean (finite-dimensional)
setting; it is the only place where we use that the discretized state space is $\R^n$.

\begin{proposition}[Superposition principle for the continuity equation {\cite[Prop.~8.1.8]{ambrosio2004gradient}}]
\label{prop:ambrosio-representation}
Let $\mu_t$, $t\in[0,T]$, be a narrowly continuous family of Borel probability measures on $\R^n$ solving the continuity
equation
\[
\partial_t\mu_t+\nabla\cdot(v_t\,\mu_t)=0
\quad\text{in }\Dcal'((0,T)\times\R^n),
\]
with a Borel velocity field $v:[0,T]\times\R^n\to\R^n$ such that
\[
\int_0^T\int_{\R^n}|v_t(x)|\,\dd\mu_t(x)\,\dd t<\infty.
\]
Then there exists $\eta\in\Pcal\big(\Ccal([0,T];\R^n)\big)$ such that for every $t\in[0,T]$,
\[
(E_t)_\#\eta=\mu_t,
\]
and $\eta$ is concentrated on $AC([0,T];\R^n)$ with
\[
\dot\gamma(t)=v_t(\gamma(t))\quad\text{for }\eta\text{-a.e.\ }\gamma\text{ and a.e.\ }t\in(0,T).
\]
Moreover, if for some $p>1$,
\[
\int_0^T\int_{\R^n}|v_t(x)|^p\,\dd\mu_t(x)\,\dd t<\infty,
\]
then
\[
\int_{\Ccal([0,T];\R^n)}\int_0^T |\dot\gamma(t)|^p\,\dd t\,\dd\eta(\gamma)
=
\int_0^T\int_{\R^n}|v_t(x)|^p\,\dd\mu_t(x)\,\dd t.
\]
\end{proposition}

\begin{corollary}[Canonical path measure associated with a continuity equation]
\label{cor:path-measure-from-CE}
In the setting of Proposition~\ref{prop:ambrosio-representation}, any $\eta$ provided by the superposition principle is a
path measure with time marginals $(E_t)_\#\eta=\mu_t$. If $v\in L^p(\mu_t\dd t)$ for some $p>1$, then $\eta$ is
concentrated on $AC([0,T];\R^n)$ and satisfies the $L^p$ action identity
\[
\int\!\!\int_0^T |\dot\gamma(t)|^p\,\dd t\,\dd\eta(\gamma)
=
\int_0^T\int_{\R^n}|v_t(x)|^p\,\dd\mu_t(x)\,\dd t.
\]
\end{corollary}

\begin{proof}
By Proposition~\ref{prop:ambrosio-representation}, there exists a probability measure
$\eta\in\Pcal(\Ccal([0,T];\R^n))$ concentrated on $AC([0,T];\R^n)$ such that $(E_t)_\#\eta=\mu_t$ for every $t\in[0,T]$.
This is exactly the marginal identity. Assume in addition that $v\in L^p(\mu_t\dd t)$ for some $p>1$. Proposition~\ref{prop:ambrosio-representation} yields that
$\dot\gamma(t)=v_t(\gamma(t))$ for $\eta$-a.e.\ $\gamma$ and for a.e.\ $t\in(0,T)$, and that
\begin{equation}\label{eq:action-superposition}
\int_{\Ccal([0,T];\R^n)}\int_0^T |\dot\gamma(t)|^p\,\dd t\,\dd\eta(\gamma)
=
\int_0^T\int_{\R^n}|v_t(x)|^p\,\dd\mu_t(x)\,\dd t.
\end{equation}
For completeness, we justify \eqref{eq:action-superposition} from the pointwise identity
$\dot\gamma(t)=v_t(\gamma(t))$ and the marginal relation $(E_t)_\#\eta=\mu_t$. Indeed, since $|\dot\gamma(t)|^p$ is nonnegative and measurable on $\Ccal([0,T];\R^n)\times(0,T)$, Tonelli gives
\[
\int_{\Ccal([0,T];\R^n)}\int_0^T |\dot\gamma(t)|^p\,\dd t\,\dd\eta(\gamma)
=
\int_0^T\int_{\Ccal([0,T];\R^n)} |\dot\gamma(t)|^p\,\dd\eta(\gamma)\,\dd t.
\]
Using $\dot\gamma(t)=v_t(\gamma(t))$ for $\eta$-a.e.\ $\gamma$ and a.e.\ $t$,
\[
\int_{\Ccal([0,T];\R^n)} |\dot\gamma(t)|^p\,\dd\eta(\gamma)
=
\int_{\Ccal([0,T];\R^n)} |v_t(\gamma(t))|^p\,\dd\eta(\gamma)
=
\int_{\R^n} |v_t(x)|^p\,\dd (E_t)_\#\eta(x)
=
\int_{\R^n} |v_t(x)|^p\,\dd\mu_t(x),
\]
where the third equality is the pushforward change-of-variables formula and the last equality is $(E_t)_\#\eta=\mu_t$.
Integrating in $t$ yields \eqref{eq:action-superposition}.
\end{proof}

\begin{remark}[How this is used for sampler interpolations]
In Section~\ref{sec:sampler} we obtain (at finite resolution) a closed continuity equation for the within-step law
interpolation with a Borel drift given by a conditional expectation. Proposition~\ref{prop:ambrosio-representation} then
yields the existence of a path measure representing that law curve, without assuming any flow structure in advance.
After reconstruction to $L^2(\T^d)$, the H\"older/tightness arguments below apply verbatim.
\end{remark}

\subsection{$L^2$ balls are compact in $H^{-1}$}

We now switch back to $D=\T^d$ and the function space setting used throughout the paper.
The $H^{-1}$ norm is
\[
\|f\|_{H^{-1}}^2=(2\pi)^d\sum_{k\in\Z^d}(1+|k|^2)^{-1}|\widehat f(k)|^2.
\]

\begin{lemma}[$L^2$ balls are relatively compact in $H^{-1}$]
\label{lem:L2ball-compact-Hminus1}
For every $R>0$, the set $B_R:=\{u\in L^2(D;\R^d):\|u\|_{L^2}\le R\}$ is relatively compact in $H^{-1}(D;\R^d)$.
\end{lemma}

\begin{proof}
Fix $\varepsilon>0$. Let $P_{\le K}$ be the Fourier projector onto modes $|k|\le K$ and $P_{>K}=\Id-P_{\le K}$.
For $u\in B_R$,
\[
\|P_{>K}u\|_{H^{-1}}^2
=(2\pi)^d\sum_{|k|>K}(1+|k|^2)^{-1}|\widehat u(k)|^2
\le (1+K^2)^{-1}\|u\|_{L^2}^2
\le (1+K^2)^{-1}R^2.
\]
Choose $K$ so that $\|P_{>K}u\|_{H^{-1}}\le \varepsilon/2$ for all $u\in B_R$.
The range $E_K:=\{u=P_{\le K}u\}$ is finite-dimensional, hence $P_{\le K}B_R$ is totally bounded in $H^{-1}$.
Cover $P_{\le K}B_R$ by finitely many $H^{-1}$ balls of radius $\varepsilon/2$ with centers $w^1,\dots,w^M\in E_K$.
Then for each $u\in B_R$ some $w^\ell$ satisfies
\[
\|u-w^\ell\|_{H^{-1}}\le \|P_{\le K}u-w^\ell\|_{H^{-1}}+\|P_{>K}u\|_{H^{-1}}\le \varepsilon.
\]
Thus $B_R$ is totally bounded in $H^{-1}$, hence relatively compact.
\end{proof}

\subsection{Action implies a H\"older increment bound}

Let $\Gamma:=\Ccal([0,T];H^{-1}(D))$. The next lemma is deterministic and will be applied to $\eta$-a.e.\ path.

\begin{lemma}[Deterministic H\"older increment bound from $L^p$ action]\label{lem:det-holder-from-action}
Fix $p>1$ and let $\gamma\in AC([0,T];L^2(D))$. Then for all $0\le s\le t\le T$,
\begin{equation}\label{eq:det-holder-Hminus1}
\|\gamma(t)-\gamma(s)\|_{H^{-1}}
\le \|\gamma(t)-\gamma(s)\|_{L^2}
\le |t-s|^{1-1/p}\left(\int_s^t \|\dot\gamma(r)\|_{L^2}^p\,\dd r\right)^{1/p}.
\end{equation}
\end{lemma}

\begin{proof}
The first inequality is $\|f\|_{H^{-1}}\le\|f\|_{L^2}$.
Absolute continuity gives $\gamma(t)-\gamma(s)=\int_s^t \dot\gamma(r)\,\dd r$ in $L^2$, hence
\[
\|\gamma(t)-\gamma(s)\|_{L^2}\le \int_s^t \|\dot\gamma(r)\|_{L^2}\,\dd r.
\]
Apply H\"older on $[s,t]$ with exponents $p$ and $p'=\frac{p}{p-1}$:
\[
\int_s^t \|\dot\gamma(r)\|_{L^2}\,\dd r
\le (t-s)^{1/p'}\left(\int_s^t \|\dot\gamma(r)\|_{L^2}^p\,\dd r\right)^{1/p}
= |t-s|^{1-1/p}\left(\int_s^t \|\dot\gamma(r)\|_{L^2}^p\,\dd r\right)^{1/p}.
\]
\end{proof}

\subsection{Tightness in $\Ccal([0,T];H^{-1})$}

\begin{assumption}[Uniform $L^2$ support and $L^p$ action]\label{ass:uniform-action}
A family $\{\eta^m\}\subset \Pcal(\Gamma)$ is concentrated on $AC([0,T];L^2(D))$ paths and there exist $R>0$, $A_p<\infty$, $p>1$ such that:
\begin{enumerate}[label=(\roman*),leftmargin=2.2em]
\item $\eta^m\big(\{\gamma:\sup_{t\in[0,T]}\|\gamma(t)\|_{L^2}\le R\}\big)=1$ for all $m$;
\item $\displaystyle \int_\Gamma\int_0^T\|\dot\gamma(t)\|_{L^2}^p\,\dd t\,\dd\eta^m(\gamma)\le A_p$ for all $m$.
\end{enumerate}
\end{assumption}

\begin{remark}[How Assumption~\ref{ass:uniform-action} is verified]
In the finite-dimensional (discretized) setting, Corollary~\ref{cor:path-measure-from-CE} gives a canonical way to build
$\eta^m$ from a continuity equation and identifies the path action with $\int_0^T\!\int |v_t|^p\,\dd\mu_t\,\dd t$.
After reconstruction to $L^2(\T^d)$, item (ii) becomes an $L^p$ bound on reconstructed speeds.
\end{remark}

\begin{lemma}[Compact H\"older sets]\label{lem:compact-holder-set}
Fix $p>1$, $R>0$, and $H>0$. Let $\overline{B_R}^{\,H^{-1}}$ be the $H^{-1}$-closure of $B_R$.
Define
\[
\Kcal_{R,H}:=\left\{\gamma\in\Gamma:\ \gamma(t)\in\overline{B_R}^{\,H^{-1}}\ \forall t,\ 
\sup_{0\le s<t\le T}\frac{\|\gamma(t)-\gamma(s)\|_{H^{-1}}}{|t-s|^{1-1/p}}\le H\right\}.
\]
Then $\Kcal_{R,H}$ is compact in $\Gamma=\Ccal([0,T];H^{-1})$.
\end{lemma}

\begin{proof}
For each fixed $t$, $\{\gamma(t):\gamma\in\Kcal_{R,H}\}\subset\overline{B_R}^{\,H^{-1}}$.
By Lemma~\ref{lem:L2ball-compact-Hminus1}, $\overline{B_R}^{\,H^{-1}}$ is compact in $H^{-1}$, giving pointwise relative compactness.
The H\"older seminorm bound gives uniform equicontinuity in $H^{-1}$.
Closedness of $\Kcal_{R,H}$ under uniform convergence in $H^{-1}$ is immediate.
Arzel\`a--Ascoli yields compactness.
\end{proof}

\begin{proposition}[Tightness on path space]\label{prop:tightness-path}
Assume \cref{ass:uniform-action}. Then $\{\eta^m\}$ is tight in $\Pcal(\Gamma)$.
\end{proposition}

\begin{proof}
Let $\mathsf A(\gamma):=\int_0^T\|\dot\gamma(t)\|_{L^2}^p\,\dd t$.
By Lemma~\ref{lem:det-holder-from-action},
\[
\sup_{0\le s<t\le T}\frac{\|\gamma(t)-\gamma(s)\|_{H^{-1}}}{|t-s|^{1-1/p}}
\le \mathsf A(\gamma)^{1/p}.
\]
Hence, by Markov,
\[
\eta^m(\Kcal_{R,H}^c)\le \eta^m(\mathsf A>H^p)\le \frac{1}{H^p}\int \mathsf A\,\dd\eta^m \le \frac{A_p}{H^p}.
\]
Choose $H$ so that $A_p/H^p\le\varepsilon$. Then $\eta^m(\Kcal_{R,H})\ge 1-\varepsilon$ and $\Kcal_{R,H}$ is compact.
\end{proof}

\section{Identification: residuals and LM compactness imply an LM statistical solution}
\label{sec:identify}

This section turns LM compactness into an identification result: if an approximating sequence is compact in the LM topology
and satisfies the Euler hierarchy identities up to residuals that vanish along a subsequence, then every subsequential
limit is an LM statistical solution.

\subsection{Compactness and convergence of admissible observables}

\begin{theorem}[Compactness from LM inputs]\label{thm:compact}
Assume \cref{ass:LM-inputs}. Then there exists a subsequence (not relabeled) and a limit curve
$\mu_\cdot\in L^1([0,T);\Pcal(L^2_x))$ such that $\widehat\mu^{\Delta_j}_\cdot\to\mu_\cdot$ in $d_T$.
Moreover, expectations of all LM-admissible observables converge along this subsequence.
\end{theorem}

\begin{proof}
This is exactly the compactness and admissible-observable convergence mechanism of~\cite{LMP2021}:
uniform time-regularity, uniform $L^2$ bounds, and a uniform structure-function modulus yield relative compactness in $d_T$,
and admissible observables are stable under $d_T$ convergence.
\end{proof}

\subsection{Admissibility of hierarchy integrands}

\begin{lemma}[Hierarchy integrands are LM-admissible]\label{lem:poly-adm}
Fix $k\in\N$ and divergence-free $\phi_1,\dots,\phi_k\in C^\infty([0,T)\times D;\R^d)$, and set
$\phi(t,x)=\phi_1(t,x_1)\otimes\cdots\otimes\phi_k(t,x_k)$ on $D^k$.
Define, for $\xi=(\xi_1,\dots,\xi_k)\in(\R^d)^k$,
\begin{align*}
g_0(t,x,\xi)
&:= \partial_t\phi(t,x):(\xi_1\otimes\cdots\otimes\xi_k),\\
g_i(t,x,\xi)
&:= \nabla_{x_i}\phi(t,x):\big(\xi_1\otimes\cdots\otimes(\xi_i\otimes\xi_i)\otimes\cdots\otimes\xi_k\big),
\qquad i=1,\dots,k.
\end{align*}
Then each $g_0,g_1,\dots,g_k$ is LM-admissible in the sense of \cref{def:LM-adm}.
\end{lemma}

\begin{proof}
We verify \eqref{eq:adm-growth} and \eqref{eq:adm-lip} from \cref{def:LM-adm}.
Throughout, $C$ denotes a constant depending only on $k$, $d$, and finitely many $L^\infty$ norms of
$\partial_t\phi$ and $\nabla_{x_i}\phi$.

\smallskip
\noindent\textbf{Step 1: Growth.}
Since $\partial_t\phi$ is a bounded $k$-tensor,
\[
|g_0(t,x,\xi)|
\le C\prod_{j=1}^k |\xi_j|
\le C\prod_{j=1}^k \sqrt{1+|\xi_j|^2}
\le C\prod_{j=1}^k (1+|\xi_j|^2),
\]
which is \eqref{eq:adm-growth} for $g_0$.

Similarly, $\nabla_{x_i}\phi$ is bounded and the $i$-th slot is quadratic, hence
\[
|g_i(t,x,\xi)|
\le C\,|\xi_i|^2\prod_{j\ne i}|\xi_j|
\le C\,(1+|\xi_i|^2)\prod_{j\ne i}\sqrt{1+|\xi_j|^2}
\le C\prod_{j=1}^k(1+|\xi_j|^2),
\]
so \eqref{eq:adm-growth} holds for each $g_i$.

\smallskip
\noindent\textbf{Step 2: Lipschitz estimate for $g_0$.}
Fix $\xi,\xi'\in(\R^d)^k$ and define the telescoping sequence
\[
\xi^{(i)}:=(\xi_1',\dots,\xi_{i-1}',\xi_i,\dots,\xi_k),
\qquad i=1,\dots,k+1,
\]
so that $\xi^{(1)}=\xi$ and $\xi^{(k+1)}=\xi'$.
Then
\[
g_0(t,x,\xi)-g_0(t,x,\xi')
=\sum_{i=1}^k\bigl(g_0(t,x,\xi^{(i)})-g_0(t,x,\xi^{(i+1)})\bigr).
\]
Since $g_0$ is multilinear in $\xi_1,\dots,\xi_k$, changing only the $i$-th component yields
\[
\bigl|g_0(t,x,\xi^{(i)})-g_0(t,x,\xi^{(i+1)})\bigr|
\le C\,|\xi_i-\xi_i'|\prod_{j\ne i}\bigl(|\xi_j|+|\xi_j'|\bigr).
\]
Use $|\eta|+|\eta'|\le C\sqrt{1+|\eta|^2+|\eta'|^2}$ and $\Pi_i(\xi,\xi')\ge 1$ to get
\[
\prod_{j\ne i}\bigl(|\xi_j|+|\xi_j'|\bigr)
\le C\,\Pi_i(\xi,\xi')^{1/2}
\le C\,\Pi_i(\xi,\xi').
\]
Also $\sqrt{1+|\xi_i|^2+|\xi_i'|^2}\ge 1$, hence
\[
\bigl|g_0(t,x,\xi^{(i)})-g_0(t,x,\xi^{(i+1)})\bigr|
\le
C\,\Pi_i(\xi,\xi')\,\sqrt{1+|\xi_i|^2+|\xi_i'|^2}\,|\xi_i-\xi_i'|.
\]
Summing over $i$ gives \eqref{eq:adm-lip} for $g_0$.

\smallskip
\noindent\textbf{Step 3: Lipschitz estimate for $g_m$ ($m\in\{1,\dots,k\}$).}
Fix $m$. Again telescope one component at a time:
\[
g_m(t,x,\xi)-g_m(t,x,\xi')
=\sum_{i=1}^k\bigl(g_m(t,x,\xi^{(i)})-g_m(t,x,\xi^{(i+1)})\bigr).
\]
If $i\ne m$, then $g_m$ is linear in the $i$-th slot and quadratic only in the $m$-th slot, hence
\[
\bigl|g_m(t,x,\xi^{(i)})-g_m(t,x,\xi^{(i+1)})\bigr|
\le C\,|\xi_i-\xi_i'|\,(|\xi_m|^2+|\xi_m'|^2)\!\!\prod_{j\ne i,m}\bigl(|\xi_j|+|\xi_j'|\bigr).
\]
Use $(|\xi_m|^2+|\xi_m'|^2)\le 1+|\xi_m|^2+|\xi_m'|^2\le \Pi_i(\xi,\xi')$ (since $m\ne i$),
and the same bound on the product of the remaining factors as in Step 2, to conclude
\[
\bigl|g_m(t,x,\xi^{(i)})-g_m(t,x,\xi^{(i+1)})\bigr|
\le
C\,\Pi_i(\xi,\xi')\,\sqrt{1+|\xi_i|^2+|\xi_i'|^2}\,|\xi_i-\xi_i'|.
\]

If $i=m$, we use
\[
\xi_m\otimes\xi_m-\xi_m'\otimes\xi_m'
=(\xi_m-\xi_m')\otimes\xi_m+\xi_m'\otimes(\xi_m-\xi_m'),
\]
so
\[
\|\xi_m\otimes\xi_m-\xi_m'\otimes\xi_m'\|
\le (|\xi_m|+|\xi_m'|)\,|\xi_m-\xi_m'|
\le C\,\sqrt{1+|\xi_m|^2+|\xi_m'|^2}\,|\xi_m-\xi_m'|.
\]
Multiplying by $\prod_{j\ne m}(|\xi_j|+|\xi_j'|)\le C\,\Pi_m(\xi,\xi')$ as before yields
\[
\bigl|g_m(t,x,\xi^{(m)})-g_m(t,x,\xi^{(m+1)})\bigr|
\le
C\,\Pi_m(\xi,\xi')\,\sqrt{1+|\xi_m|^2+|\xi_m'|^2}\,|\xi_m-\xi_m'|.
\]

Summing the $k$ telescoping terms gives \eqref{eq:adm-lip} for $g_m$.
This completes the proof.
\end{proof}

\subsection{Residuals and identification}

We define residuals for approximate correlation measures $\nu^\Delta$ by the defect in the hierarchy identity.

\begin{definition}[Hierarchy residual]\label{def:hier-res}
For each $\Delta>0$, let $\nu^{\Delta}$ be the correlation hierarchy associated with $\widehat\mu^\Delta_\cdot$
(and $\bar\nu^\Delta$ the hierarchy of $\widehat\mu^\Delta_0$).
For divergence-free $\phi_1,\dots,\phi_k$, define
\begin{align}
\mathcal R^\Delta(\phi_1,\dots,\phi_k)
&:=
\int_0^T\int_{D^k}\Bigg[
\ip{\nu^{\Delta,k}_{t,x}}{\xi_1\otimes\cdots\otimes\xi_k}:\partial_t\phi(t,x)
\\
&\qquad\qquad\qquad
+\sum_{i=1}^k
\ip{\nu^{\Delta,k}_{t,x}}{\xi_1\otimes\cdots\otimes(\xi_i\otimes\xi_i)\otimes\cdots\otimes\xi_k}
:\nabla_{x_i}\phi(t,x)
\Bigg]\dd x\,\dd t \nonumber\\
&\qquad
+\int_{D^k}\ip{\bar\nu^{\Delta,k}_x}{\xi_1\otimes\cdots\otimes\xi_k}:\phi(0,x)\,\dd x. \nonumber\label{eq:residual}
\end{align}

\end{definition}

\begin{assumption}[Vanishing residuals]\label{ass:residuals}
Along a subsequence $\Delta_j\to0$:
\begin{enumerate}[label=(\roman*),leftmargin=2.2em]
\item $\widehat\mu^{\Delta_j}_0\Rightarrow\bar\mu$ weakly with finite second moment;
\item for every $k$ and divergence-free test tuple, $\mathcal R^{\Delta_j}(\phi_1,\dots,\phi_k)\to0$;
\item the incompressibility constraint \eqref{eq:LM-divfree} holds for $\nu^{\Delta_j,2}$ for a.e.\ $t$.
\end{enumerate}
\end{assumption}

\begin{remark}[Resolved versus full residuals]
The vanishing-residual assumption \cref{ass:residuals}(ii) is stated for the full Euler hierarchy.
Section~\ref{sec:training} provides a training-native bound for the \emph{resolved} residual
$\mathcal R^{\Delta,K}$ tested against $K$-band-limited fields. Passing from resolved to full residuals can be done
by letting $K\to\infty$ and controlling the unresolved tail via the same structure-function mechanism used in
Section~\ref{sec:approx}. This two-parameter limit is not expanded here.
\end{remark}

\begin{theorem}[Compactness + residual $\to0$ implies LM statistical solution]\label{thm:identify}
Assume \cref{ass:LM-inputs} and \cref{ass:residuals}. Then every $d_T$ limit of $\widehat\mu^{\Delta_j}_\cdot$ is an
LM statistical solution in the sense of \cref{def:LM-SS} with initial law $\bar\mu$.
\end{theorem}

\begin{proof}
By Theorem~\ref{thm:compact}, after extracting a subsequence (not relabeled) we have
$\widehat\mu^{\Delta_j}_\cdot\to\mu_\cdot$ in $d_T$. For each $t$, let $\nu^{\Delta_j}_t$ and $\nu_t$ denote the correlation
hierarchies associated with $\widehat\mu^{\Delta_j}_t$ and $\mu_t$ via the LM correspondence theorem.

\smallskip
\noindent\textbf{Step 1: hierarchy identity.}
Fix $k\in\N$ and divergence-free tests $\phi_1,\dots,\phi_k$ and the associated tensor product
$\phi(t,x)=\phi_1(t,x_1)\otimes\cdots\otimes\phi_k(t,x_k)$.
Let $g_0,g_1,\dots,g_k$ be the hierarchy integrands from Lemma~\ref{lem:poly-adm}.
By Lemma~\ref{lem:poly-adm}, each $g_i$ is LM-admissible.

By LM admissible-observable convergence (Theorem~2.4 in~\cite{LMP2021}), for each $i=0,1,\dots,k$ we can pass to the limit in
\[
\int_0^T\int_{D^k}\ip{\nu^{\Delta_j,k}_{t,x}}{g_i(t,x,\xi)}\,\dd x\,\dd t
\quad\longrightarrow\quad
\int_0^T\int_{D^k}\ip{\nu^{k}_{t,x}}{g_i(t,x,\xi)}\,\dd x\,\dd t.
\]
For the initial term, Assumption~\ref{ass:residuals}(i) gives $\widehat\mu^{\Delta_j}_0\Rightarrow\bar\mu$ with finite second moment,
hence the associated initial correlation measures $\bar\nu^{\Delta_j,k}$ converge to $\bar\nu^k$ (the hierarchy of $\bar\mu$) in the
sense needed to pass
\[
\int_{D^k}\ip{\bar\nu^{\Delta_j,k}_x}{\xi_1\otimes\cdots\otimes\xi_k}:\phi(0,x)\,\dd x
\;\longrightarrow\;
\int_{D^k}\ip{\bar\nu^{k}_x}{\xi_1\otimes\cdots\otimes\xi_k}:\phi(0,x)\,\dd x.
\]
Combining these limit passages with Assumption~\ref{ass:residuals}(ii), i.e.
$\mathcal R^{\Delta_j}(\phi_1,\dots,\phi_k)\to 0$, yields that the limiting hierarchy satisfies \eqref{eq:LM-hierarchy}.

\smallskip
\noindent\textbf{Step 2: incompressibility.}
By Assumption~\ref{ass:residuals}(iii), the incompressibility constraint \eqref{eq:LM-divfree} holds for $\nu^{\Delta_j,2}$ for a.e.\ $t$.
The integrand in \eqref{eq:LM-divfree} is a (polynomial) admissible observable in the LM sense (with $k=2$), hence admissible-observable
convergence passes \eqref{eq:LM-divfree} to the limit. Therefore \eqref{eq:LM-divfree} holds for $\nu^2_t$ for a.e.\ $t$, and $\mu_t$ is
concentrated on $L^2_\sigma$ for a.e.\ $t$.

\smallskip
\noindent\textbf{Step 3: time-regularity.}
By \cref{ass:LM-inputs}(ii), the sequence $\widehat\mu^{\Delta_j}_\cdot$ is time-regular with uniform constants $(C,L)$.
Since $\widehat\mu^{\Delta_j}_\cdot\to\mu_\cdot$ in $d_T$ and the uniform first-moment bound holds by \cref{ass:LM-inputs}(i),
Lemma~\ref{lem:time-reg-closed} implies that $\mu_\cdot$ is time-regular with the same $(C,L)$.

\smallskip
\noindent\textbf{Conclusion.}
Items (i)--(iii) in Definition~\ref{def:LM-SS} hold with initial law $\bar\mu$, so $\mu_\cdot$ is an LM statistical solution.
\end{proof}

\section{Training-native certification of vanishing hierarchy residuals}
\label{sec:training}

This section gives a quantitative route to verify vanishing hierarchy residuals from training-native regression errors.
The key point is that training and sampling are inherently \emph{finite-resolution}: the learned drift lives in a resolved
$L^2$ state space (grid / spectral truncation). Accordingly, certification is stated on resolved scales.

\subsection{Euler drift in weak form and its resolved projection}

Let $D=\T^d$ and write $\|u\|_2=\|u\|_{L^2(D)}$.
For divergence-free $\varphi\in C^\infty(D;\R^d)$, define the Euler drift functional
$\mathcal B^\star(u)\in H^{-1}(D;\R^d)$ by duality:
\begin{equation}\label{eq:Euler-drift}
\ip{\mathcal B^\star(u)}{\varphi}
:= -\int_D (u\otimes u):\nabla\varphi\,\dd x.
\end{equation}
The right-hand side is well-defined for $u\in L^2$ and yields $\mathcal B^\star(u)\in H^{-1}$.

\paragraph{Resolved drift.}
Fix a resolution $K\ge 1$ and let $P_{\le K}$ be the sharp Fourier projector from Section~\ref{sec:approx}, acting by Fourier
multipliers. Since $P_{\le K}$ has finite-dimensional range, it extends canonically to distributions and maps $H^{-1}$ into
$L^2$. Define the resolved Euler drift by
\begin{equation}\label{eq:Euler-drift-resolved}
\mathcal B^\star_K(u):=P_{\le K}\mathcal B^\star(u)\in L^2_x.
\end{equation}
For any test $\varphi\in C^\infty(D;\R^d)$,
\begin{equation}\label{eq:proj-duality}
\ip{\mathcal B^\star_K(u)}{\varphi}
=
\ip{\mathcal B^\star(u)}{P_{\le K}\varphi}.
\end{equation}
In particular, if $\varphi$ is $K$-band-limited (i.e.\ $P_{\le K}\varphi=\varphi$), then
\begin{equation}\label{eq:bandlimited-equality}
\ip{\mathcal B^\star_K(u)}{\varphi}=\ip{\mathcal B^\star(u)}{\varphi}.
\end{equation}

\paragraph{Learned (resolved) drift.}
A learned sampler (or learned probability-flow ODE) induces a measurable drift
\[
\mathcal B^\Delta:[0,T]\times L^2_x\to L^2_x,
\]
interpreted as an $L^2$ vector field on resolved scales. The training loss below compares $\mathcal B^\Delta$ to the resolved
target $\mathcal B^\star_K$ in $L^2$.

\subsection{Law-level product observables and the hierarchy residual as a drift defect}

Fix $k\in\N$ and divergence-free test fields
\[
\phi_1,\dots,\phi_k\in C_c^\infty([0,T)\times D;\R^d),
\]
and define the linear functionals
\begin{equation}\label{eq:Psi-def}
\Psi_{\phi}(t,u):=\int_D u(x)\cdot\phi(t,x)\,\dd x,
\end{equation}
and the $k$-fold product observable
\begin{equation}\label{eq:F-def}
\mathcal F_{\boldsymbol\phi}(t,u):=\prod_{j=1}^k \Psi_{\phi_j}(t,u),
\qquad \boldsymbol\phi:=(\phi_1,\dots,\phi_k).
\end{equation}

\begin{lemma}[Derivatives of the product observable]\label{lem:derivF}
For $u,w\in L^2_x$ and $t\in[0,T)$,
\begin{align}
\partial_t \mathcal F_{\boldsymbol\phi}(t,u)
&=\sum_{i=1}^k \left(\int_D u\cdot \partial_t\phi_i(t)\,\dd x\right)\prod_{j\ne i}\Psi_{\phi_j}(t,u),
\label{eq:dtF}\\
D_u\mathcal F_{\boldsymbol\phi}(t,u)[w]
&=\sum_{i=1}^k \left(\int_D w\cdot \phi_i(t)\,\dd x\right)\prod_{j\ne i}\Psi_{\phi_j}(t,u).
\label{eq:duF}
\end{align}
\end{lemma}

\begin{proof}
For each $i$, $\Psi_{\phi_i}(t,u)$ is linear in $u$ and smooth in $t$.
Differentiate the product \eqref{eq:F-def} in time:
\[
\partial_t\mathcal F_{\boldsymbol\phi}(t,u)=\sum_{i=1}^k \left(\partial_t\Psi_{\phi_i}(t,u)\right)\prod_{j\ne i}\Psi_{\phi_j}(t,u),
\]
and $\partial_t\Psi_{\phi_i}(t,u)=\int_D u\cdot\partial_t\phi_i(t)\,\dd x$, giving \eqref{eq:dtF}.

For the Fr\'echet derivative in $u$, expand
\[
\Psi_{\phi_i}(t,u+\varepsilon w)=\Psi_{\phi_i}(t,u)+\varepsilon\int_D w\cdot\phi_i(t)\,\dd x,
\]
insert into the product \eqref{eq:F-def}, and differentiate at $\varepsilon=0$ to obtain \eqref{eq:duF}.
\end{proof}

\subsubsection{A generator identity for drift-driven law curves}

Let $\rho_\cdot\in L^1_t(\Pcal(L^2_x))$ be a law curve.
We say that $\rho_\cdot$ is \emph{drift-driven by} $\mathcal B$ (in the cylindrical weak sense) if for every product observable
$\mathcal F_{\boldsymbol\phi}$ as above, the map $t\mapsto \int\mathcal F_{\boldsymbol\phi}(t,u)\,\dd\rho_t(u)$ is absolutely continuous and
\begin{equation}\label{eq:drift-driven-identity}
\frac{\dd}{\dd t}\int_{L^2_x}\mathcal F_{\boldsymbol\phi}(t,u)\,\dd\rho_t(u)
=
\int_{L^2_x}\Big(\partial_t\mathcal F_{\boldsymbol\phi}(t,u)+D_u\mathcal F_{\boldsymbol\phi}(t,u)[\mathcal B(t,u)]\Big)\,\dd\rho_t(u)
\end{equation}
for a.e.\ $t\in(0,T)$.
This identity holds, in particular, for deterministic ODE sampling curves (e.g.\ probability-flow ODEs) and for sampler
interpolations whenever a closed continuity equation with drift $\mathcal B$ holds in cylindrical test form.

Because each $\phi_i$ is compactly supported in $[0,T)$, $\mathcal F_{\boldsymbol\phi}(t,\cdot)$ vanishes for $t$ close to $T$.
Integrating \eqref{eq:drift-driven-identity} from $0$ to $T$ gives
\begin{equation}\label{eq:drift-driven-integrated}
\int_0^T\int_{L^2_x}\Big(\partial_t\mathcal F_{\boldsymbol\phi}(t,u)+D_u\mathcal F_{\boldsymbol\phi}(t,u)[\mathcal B(t,u)]\Big)\,\dd\rho_t(u)\,\dd t
+\int_{L^2_x}\mathcal F_{\boldsymbol\phi}(0,u)\,\dd\rho_0(u)
=0.
\end{equation}

\subsubsection{Resolved residual and exact drift-defect identity}

Take $\rho_\cdot=\widehat\mu^\Delta_\cdot$, and assume it is drift-driven by $\mathcal B^\Delta$ in the sense above.
Fix $K\ge 1$. Define the \emph{resolved} hierarchy residual by replacing $\mathcal B^\star$ with the resolved drift $\mathcal B^\star_K$:
\begin{equation}\label{eq:residual-lawlevel-resolved}
\mathcal R^{\Delta,K}(\phi_1,\dots,\phi_k)
:=
\int_0^T\int_{L^2_x}\Big(\partial_t\mathcal F_{\boldsymbol\phi}(t,u)+D_u\mathcal F_{\boldsymbol\phi}(t,u)[\mathcal B^\star_K(u)]\Big)\,\dd\widehat\mu^\Delta_t(u)\,\dd t
+\int_{L^2_x}\mathcal F_{\boldsymbol\phi}(0,u)\,\dd\widehat\mu^\Delta_0(u).
\end{equation}
If the tests are $K$-band-limited (i.e.\ $P_{\le K}\phi_i=\phi_i$ for all $i$), then \eqref{eq:bandlimited-equality} implies that
$\mathcal R^{\Delta,K}(\boldsymbol\phi)$ coincides with the Euler residual defined using $\mathcal B^\star$.

Now use \eqref{eq:drift-driven-integrated} with $\mathcal B=\mathcal B^\Delta$ and subtract from
\eqref{eq:residual-lawlevel-resolved}. The time-derivative and initial terms cancel, leaving the exact identity
\begin{equation}\label{eq:residual-as-drift-defect-resolved}
\mathcal R^{\Delta,K}(\phi_1,\dots,\phi_k)
=
\int_0^T\int_{L^2_x} D_u\mathcal F_{\boldsymbol\phi}(t,u)\big[\mathcal B^\star_K(u)-\mathcal B^\Delta(t,u)\big]\,
\dd\widehat\mu^\Delta_t(u)\,\dd t.
\end{equation}

\subsection{A fully explicit residual bound from $L^2$ drift regression on resolved scales}

Define the resolved drift regression loss (evaluated on the produced law curve)
\begin{equation}\label{eq:Ldrift-on-mu-resolved}
\mathcal L_{\mathrm{drift}}^{\Delta,K}
:=
\int_0^T\int_{L^2_x}\|\mathcal B^\Delta(t,u)-\mathcal B^\star_K(u)\|_2^2\,\dd\widehat\mu^\Delta_t(u)\,\dd t.
\end{equation}

\begin{proposition}[Residual bound from $L^2$ drift regression (resolved)]
\label{prop:residual-reg}
Assume the law curve has a uniform $2k$-moment bound:
\begin{equation}\label{eq:moment-2k}
M_{2k}:=\sup_{t\in[0,T]}\int_{L^2_x}\|u\|_2^{2k}\,\dd\widehat\mu^\Delta_t(u)<\infty.
\end{equation}
Then for every divergence-free $\phi_1,\dots,\phi_k\in C_c^\infty([0,T)\times D;\R^d)$,
\begin{equation}\label{eq:residual-bound-final}
\abs{\mathcal R^{\Delta,K}(\phi_1,\dots,\phi_k)}
\le
\sqrt{T}\,M_{2k}^{\frac{k-1}{2k}}\,
\Bigg(\sum_{i=1}^k \|\phi_i\|_{L^\infty_tL^2_x}\prod_{j\ne i}\|\phi_j\|_{L^\infty_tL^2_x}\Bigg)
\big(\mathcal L_{\mathrm{drift}}^{\Delta,K}\big)^{1/2}.
\end{equation}
\end{proposition}

\begin{proof}
Start from the exact identity \eqref{eq:residual-as-drift-defect-resolved} and insert \eqref{eq:duF}:
\[
\mathcal R^{\Delta,K}(\phi_1,\dots,\phi_k)
=
\int_0^T\int
\sum_{i=1}^k
\left(\int_D (\mathcal B^\star_K(u)-\mathcal B^\Delta(t,u))\cdot\phi_i(t)\,\dd x\right)
\prod_{j\ne i}\Psi_{\phi_j}(t,u)\,
\dd\widehat\mu^\Delta_t(u)\,\dd t.
\]
Fix an index $i$ and bound pointwise in $(t,u)$.

\smallskip
\noindent\textbf{Step 1: bound the drift-test pairing.}
By Cauchy--Schwarz in $x$,
\begin{equation}\label{eq:drift-pair-bound}
\left|\int_D (\mathcal B^\star_K(u)-\mathcal B^\Delta(t,u))\cdot\phi_i(t)\,\dd x\right|
\le \|\mathcal B^\star_K(u)-\mathcal B^\Delta(t,u)\|_2\,\|\phi_i(t)\|_2.
\end{equation}

\smallskip
\noindent\textbf{Step 2: bound the remaining linear factors.}
For each $j\ne i$, by Cauchy--Schwarz,
\begin{equation}\label{eq:Psi-bound}
|\Psi_{\phi_j}(t,u)|
\le \|u\|_2\,\|\phi_j(t)\|_2.
\end{equation}
Therefore,
\[
\left|\prod_{j\ne i}\Psi_{\phi_j}(t,u)\right|
\le \|u\|_2^{k-1}\prod_{j\ne i}\|\phi_j(t)\|_2.
\]

\smallskip
\noindent\textbf{Step 3: combine and integrate.}
Taking $L^\infty_tL^2_x$ norms of the tests and summing over $i$ yields
\begin{align}
|\mathcal R^{\Delta,K}(\phi_1,\dots,\phi_k)|
&\le
\Bigg(\sum_{i=1}^k \|\phi_i\|_{L^\infty_tL^2_x}\prod_{j\ne i}\|\phi_j\|_{L^\infty_tL^2_x}\Bigg)
\int_0^T\int \|\mathcal B^\star_K-\mathcal B^\Delta\|_2\,\|u\|_2^{k-1}\,\dd\widehat\mu^\Delta_t\,\dd t.
\label{eq:R-bound-reduced}
\end{align}

\smallskip
\noindent\textbf{Step 4: Cauchy--Schwarz in $(t,u)$.}
Apply Cauchy--Schwarz to the last integral:
\[
\int_0^T\int \|\mathcal B^\star_K-\mathcal B^\Delta\|_2\,\|u\|_2^{k-1}
\le
\left(\int_0^T\int \|\mathcal B^\star_K-\mathcal B^\Delta\|_2^2\right)^{1/2}
\left(\int_0^T\int \|u\|_2^{2k-2}\right)^{1/2}.
\]
The first factor is exactly $\big(\mathcal L_{\mathrm{drift}}^{\Delta,K}\big)^{1/2}$ by \eqref{eq:Ldrift-on-mu-resolved}.

\smallskip
\noindent\textbf{Step 5: control the $(2k-2)$-moment by the $(2k)$-moment.}
For each $t$, Lyapunov's inequality gives
\[
\int \|u\|_2^{2k-2}\,\dd\widehat\mu^\Delta_t(u)
\le
\left(\int \|u\|_2^{2k}\,\dd\widehat\mu^\Delta_t(u)\right)^{(k-1)/k}
\le M_{2k}^{(k-1)/k}.
\]
Integrate in time and take square roots:
\[
\left(\int_0^T\int \|u\|_2^{2k-2}\,\dd\widehat\mu^\Delta_t(u)\,\dd t\right)^{1/2}
\le \sqrt{T}\,M_{2k}^{(k-1)/(2k)}.
\]
Insert into \eqref{eq:R-bound-reduced} to obtain \eqref{eq:residual-bound-final}.
\end{proof}

\begin{remark}[Resolved certification and the capacity--coverage viewpoint]
At fixed resolution $K$, Proposition~\ref{prop:residual-reg} bounds the residual tested against $K$-band-limited fields
by a training-native drift regression loss. To recover the full (unprojected) Euler hierarchy, one can let $K\to\infty$;
the remaining obstruction is the unresolved tail (the same object controlled by structure functions in Section~\ref{sec:approx}).
\end{remark}

\subsection{Diffusion: probability-flow ODE and score-to-drift regression (finite-dimensional)}
\label{subsec:pfode}

We record the deterministic probability-flow representation and the exact $L^2$ identity linking score regression
to drift regression in a finite-dimensional discretization $x\in\R^n$ (the setting in which diffusion models are trained). For a broader treatment of diffusiom models, see \cite{karras2022elucidating}.

\subsubsection{Forward SDE and Fokker--Planck equation}

Consider the forward diffusion
\begin{equation}\label{eq:fwd-sde}
\dd X_\tau = a(X_\tau,\tau)\,\dd\tau + \sigma(\tau)\,\dd W_\tau,\qquad \tau\in[0,1],
\end{equation}
where $W_\tau$ is standard Brownian motion in $\R^n$ and $\sigma(\tau)>0$ is scalar.
Assume that $X_\tau$ admits a strictly positive density $p_\tau\in C^{1,2}([0,1]\times\R^n)$.

\begin{lemma}[Fokker--Planck]\label{lem:fokker-planck}
Under the above regularity, $p_\tau$ satisfies
\begin{equation}\label{eq:fokker-planck}
\partial_\tau p_\tau = -\nabla\cdot(a(\cdot,\tau)p_\tau) + \frac12\sigma(\tau)^2\,\Delta p_\tau
\end{equation}
in the classical sense.
\end{lemma}

\begin{proof}
It\^o's formula for $\psi\in C_c^\infty(\R^n)$ gives
\[
\dd \psi(X_\tau)=\nabla\psi(X_\tau)\cdot a(X_\tau,\tau)\,\dd\tau + \frac12\sigma(\tau)^2\Delta\psi(X_\tau)\,\dd\tau
+ \sigma(\tau)\nabla\psi(X_\tau)\cdot \dd W_\tau.
\]
Take expectation and integrate by parts to obtain \eqref{eq:fokker-planck}.
\end{proof}

\subsubsection{Probability-flow drift and equality of marginals}

Define the score $s_\tau(x):=\nabla\log p_\tau(x)$ and the probability-flow drift
\begin{equation}\label{eq:pf-drift}
b^\star(x,\tau):=a(x,\tau)-\frac12\sigma(\tau)^2\,s_\tau(x).
\end{equation}

\begin{lemma}[Fokker--Planck equals continuity equation with $b^\star$]\label{lem:fp-to-ce}
Let $p_\tau$ solve \eqref{eq:fokker-planck} with $p_\tau>0$ and define $b^\star$ by \eqref{eq:pf-drift}.
Then $p_\tau$ satisfies
\begin{equation}\label{eq:pf-continuity}
\partial_\tau p_\tau + \nabla\cdot(b^\star(\cdot,\tau)p_\tau)=0.
\end{equation}
\end{lemma}

\begin{proof}
Since $s_\tau=\nabla\log p_\tau$, we have $p_\tau s_\tau=\nabla p_\tau$ and $\nabla\cdot(p_\tau s_\tau)=\Delta p_\tau$.
Insert \eqref{eq:pf-drift} into $\nabla\cdot(b^\star p_\tau)$ to obtain \eqref{eq:pf-continuity}.
\end{proof}

\subsubsection{Score regression implies drift regression (exact identity)}

Let $s_\theta(x,\tau)$ be a learned score model and define the learned probability-flow drift
\begin{equation}\label{eq:pf-drift-learned}
b_\theta(x,\tau):=a(x,\tau)-\frac12\sigma(\tau)^2\,s_\theta(x,\tau).
\end{equation}
Then
\[
b_\theta(x,\tau)-b^\star(x,\tau)=-\frac12\sigma(\tau)^2\big(s_\theta(x,\tau)-s_\tau(x)\big),
\]
and integrating under $p_\tau$ yields the exact identity
\begin{equation}\label{eq:drift-vs-score-identity}
\E_{p_\tau}\|b_\theta(\cdot,\tau)-b^\star(\cdot,\tau)\|^2
=
\frac14\,\sigma(\tau)^4\,\E_{p_\tau}\|s_\theta(\cdot,\tau)-s_\tau(\cdot)\|^2.
\end{equation}
Integrating in $\tau$ gives
\begin{equation}\label{eq:drift-vs-score-integrated}
\int_0^1 \E_{p_\tau}\|b_\theta-b^\star\|^2\,\dd\tau
=
\frac14\int_0^1 \sigma(\tau)^4\,\E_{p_\tau}\|s_\theta-s_\tau\|^2\,\dd\tau.
\end{equation}

\begin{remark}[How this plugs into the resolved residual bound]
When the learned evolution used to form $\widehat\mu^\Delta_\cdot$ is the probability-flow ODE with drift $b_\theta$,
the drift regression loss in Proposition~\ref{prop:residual-reg} is an $L^2$ error on the resolved state space.
The identities \eqref{eq:drift-vs-score-identity}--\eqref{eq:drift-vs-score-integrated} show that standard score regression
controls this drift error quantitatively.
\end{remark}


\section{Applications: certificates and standard distributional scores in statistical solution framework}
\label{sec:applications}

This section records application-facing consequences of the law-level viewpoint developed in
Sections~\ref{sec:sampler}, \ref{sec:approx}, \ref{sec:rollout}, and \ref{sec:training}.
In many ML pipelines for probabilistic PDE forecasting, one trains a conditional generative sampler on a discretized
mapping (numerical solver or data generator) and then evaluates the resulting ensembles using distributional scores
(CRPS/energy score) or diffusion likelihood values as confidence proxies.
Statistical solutions provide a principled way to interpret such finite-grid diagnostics in a continuum setting:
they make explicit \emph{which} quantities are stable under refinement and \emph{why} those quantities correspond to
robust law-level statements (in the sense of Lanthaler--Mishra--Par\'es-Pulido~\cite{LMP2021}).
We highlight two complementary instances.

\begin{itemize}
    \item \textbf{Diffusion-based likelihood certificates for future states and rare-event detection.}
    In high-impact applications, the key question is often not merely \emph{average} forecast skill but whether a model
    can assign meaningful confidence to \emph{specific future states}--especially those associated with rare or
    extreme events--in a way that is principled and robust to discretization.
    A growing literature uses probability-flow ODE (PF-ODE, see Section~\ref{subsec:pfode}) constructions to extract likelihood-like quantities from a
    trained score model and deploy them as \emph{certificates of trust} or out-of-distribution (OOD) detectors
    \cite{heng2024out,abdi2025out,jarve2025probability,graber2025out}.
    In scientific forecasting, this is particularly natural: likelihoods (or likelihood surrogates) quantify how
    compatible a candidate next state is with the learned conditional law, hence directly probe whether the model is
    extrapolating beyond its learned statistical regime.
    We specialize to the task-aware certificate setting of~\cite{raonic2025towards}, where PF-ODE likelihoods of
    discretized PDE states (e.g.\ geophysical fields) are combined with error information from a possibly different,
    deterministic model to form a dual certificate. We place these patchwise, clipped likelihood certificates inside the
    LM observable framework and show how they admit clean refinement-limit statements
    (Section~\ref{subsec:app-likelihood}), and how--under an explicit strong convexity hypothesis--they can be converted into
    quantitative mean-square error control in the coupled pipeline (Section~\ref{subsec:app-like-to-err}).

    \item \textbf{CRPS / energy score as LM-admissible resolved observables.}
    Proper scores such as CRPS and its multivariate analogue (energy score) are now standard for training and evaluating
    ensemble forecasts in weather and turbulence modeling, including diffusion-based ensemble systems and large learned
    ensembles \cite{price2023gencast,price2025probabilistic,larsson2025diffusion,andrae2024continuous,mahesh2024huge,mahesh2024huge2},
    as well as CRPS-trained operational-style models \cite{lang2026aifs} and probabilistic operator-learning frameworks
    \cite{bulte2025probabilistic}. These scores are computed on discretized fields (often pointwise or patchwise), yet
    they can be expressed at the continuum level as expectations of \emph{resolved} Lipschitz observables of the law
    (e.g.\ mollified point evaluations or fixed low-dimensional projections).
    Consequently, convergence in the LM metric
    $d_T(\mu,\nu)=\int_0^T W_1(\mu_t,\nu_t)\,\dd t$ implies quantitative control of CRPS/energy-score discrepancies for any
    fixed resolved observable, providing a rigorous bridge from the sample-based metrics used in practice to a
    continuum notion of law convergence \cite{LMP2021}. We make this link explicit in Section~\ref{subsec:crps} by showing
    that time-integrated CRPS/energy-score gaps are bounded by $d_T$ (up to the observable Lipschitz constant), thereby
    subsuming common distributional evaluation criteria within the statistical-solution framework.
\end{itemize}

\subsection{Discretized pipelines as law operators and the canonical input coupling}
\label{subsec:app-pipeline}
Throughout this section we work on $D=\T^d$ and use the phase space $L^2_x:=L^2(D;\R^m)$ as in the main text.

We specialize the sampler-as-operator viewpoint of Section~\ref{sec:sampler} to a typical discretized training pipeline.
Let $\Delta>0$ denote a resolution parameter (grid spacing, truncation level, etc.) and let
$L^2_{x_\Delta} \simeq \R^{N_\Delta}$
be a discrete state space. Let $P_\Delta:L^2_x\to L^2_{x_\Delta}$ be a restriction/projection and
$R_\Delta:L^2_{x_\Delta}\to L^2_x$
a reconstruction operator. Assume the uniform stability bound
\begin{equation}\label{eq:app-R-stab}
\| R_\Delta z\|_{L^2_x}\le C_R\|z\|_{L^2_{x_\Delta}}\qquad\forall z\in L^2_{x_\Delta},
\end{equation}
with $C_R$ independent of $\Delta$.

Let $\bar\mu\in\Pcal_2(L^2_x)$ be the initial law and set $\bar\mu^\Delta:=(P_\Delta)_\#\bar\mu$.
Let $\Phi_t^\Delta:L^2_{x_\Delta}\to L^2_{x_\Delta}$ be a deterministic numerical map (data generator) advancing from time $0$ to time $t$.
It induces a deterministic kernel $K^{\Delta,\mathrm{num}}_t(z_0,\cdot)=\delta_{\Phi_t^\Delta(z_0)}$ and hence a pushforward on laws.
Define the numerical output law on $L^2_x$ by
\begin{equation}\label{eq:app-num-law}
\mu_t^\Delta := (R_\Delta)_\#\big(\Phi_t^\Delta\big)_\#\bar\mu^\Delta \in \Pcal_2(L^2_x).
\end{equation}

A conditional generative predictor (diffusion / flow matching / rectified flow) at time $t$ is modeled as a Markov kernel
$K^\Delta_{t,\theta}:L^2_{x_\Delta}\to\Pcal(L^2_{x_\Delta})$, $z_0\mapsto K^\Delta_{t,\theta}(z_0,\cdot)$, exactly as in
Section~\ref{sec:sampler}. Its induced output law on $L^2_x$ is
\begin{equation}\label{eq:app-diff-law}
\widehat\mu^\Delta_{t,\theta}
:=
(R_\Delta)_\#\left(\int_{L^2_{x_\Delta}} K^\Delta_{t,\theta}(z_0,\cdot)\,\dd\bar\mu^\Delta(z_0)\right)\in\Pcal_2(L^2_x).
\end{equation}

\paragraph{Canonical input coupling (pipeline coupling).}
The most common training/evaluation setup couples a numerical output and a model output by using the \emph{same} input.
Formally: sample $Z_0\sim\bar\mu^\Delta$, set $Z_t:=\Phi_t^\Delta(Z_0)$, and sample $\widehat Z_t\sim K^\Delta_{t,\theta}(Z_0,\cdot)$
conditionally on the same $Z_0$. Define reconstructed fields
$U_t^\Delta:=R_\Delta Z_t\in L^2_x$ and $V^\Delta_{t,\theta}:=R_\Delta \widehat Z_t\in L^2_x$, and set
\begin{equation}\label{eq:app-pi}
\pi^\Delta_{t,\theta}:=\Law(U_t^\Delta,V^\Delta_{t,\theta})\in\Pi(\mu_t^\Delta,\widehat\mu^\Delta_{t,\theta}).
\end{equation}
This is the discrete analogue of the canonical couplings produced from sampler path measures in Section~\ref{sec:regularity}:
it is induced directly by the shared randomness of the pipeline (here, the shared input $Z_0$).

\begin{lemma}[From pipeline mean-square error to $W_1$]\label{lem:app-mse-to-W1}
For every $t$,
\[
W_1(\mu_t^\Delta,\widehat\mu_{t,\theta}^\Delta)
\le W_2(\mu_t^\Delta,\widehat\mu_{t,\theta}^\Delta)
\le \left(\int_{L^2_x\times L^2_x}\|u-v\|_{L^2_x}^2\,\dd\pi^\Delta_{t,\theta}(u,v)\right)^{1/2}.
\]
\end{lemma}

\begin{proof}
$W_1\le W_2$ is standard. Since $\pi^\Delta_{t,\theta}\in\Pi(\mu_t^\Delta,\widehat\mu^\Delta_{t,\theta})$, the definition of $W_2$
as an infimum over couplings yields the stated upper bound.
\end{proof}

\begin{theorem}[Law convergence from numerical convergence + vanishing pipeline MSE]
\label{thm:app-law-conv}
Assume:
\begin{enumerate}[label=(\roman*),leftmargin=2.2em]
\item (\emph{Numerical law converges}) there exists $\mu_\cdot\in L^1_t(\Pcal(L^2_x))$ such that
$d_T(\mu^\Delta,\mu)\to 0$ as $\Delta\to0$;
\item (\emph{Vanishing pipeline MSE}) for some $\theta=\theta(\Delta)$,
\[
\int_0^T\int_{L^2_x\times L^2_x}\|u-v\|_{L^2_x}^2\,\dd\pi^\Delta_{t,\theta(\Delta)}(u,v)\,\dd t \to 0
\qquad\text{as }\Delta\to 0.
\]
\end{enumerate}
Then $d_T(\widehat\mu^\Delta_{\cdot,\theta(\Delta)},\mu)\to0$ as $\Delta\to0$.
\end{theorem}

\begin{proof}
By Lemma~\ref{lem:app-mse-to-W1} and Cauchy--Schwarz,
\[
\int_0^T W_1(\mu_t^\Delta,\widehat\mu^\Delta_{t,\theta(\Delta)})\,\dd t
\le \sqrt{T}\left(\int_0^T\int\|u-v\|_{L^2_x}^2\,\dd\pi^\Delta_{t,\theta(\Delta)}\,\dd t\right)^{1/2}\to0.
\]
Hence $d_T(\mu^\Delta,\widehat\mu^\Delta_{\cdot,\theta(\Delta)})\to0$. By the triangle inequality in $d_T$,
\[
d_T(\widehat\mu^\Delta_{\cdot,\theta(\Delta)},\mu)\le
d_T(\widehat\mu^\Delta_{\cdot,\theta(\Delta)},\mu^\Delta)+d_T(\mu^\Delta,\mu)\to0.
\]
\end{proof}

\begin{remark}[How this connects to the rest of the paper]
Theorem~\ref{thm:app-law-conv} is the application-level analogue of the general philosophy of Sections~\ref{sec:sampler} and
\ref{sec:training}: once a pipeline is expressed as an operator on laws, canonical couplings turn standard training losses
(mean-square error under shared inputs) into quantitative law convergence in the LM metric $d_T$.
As a consequence, \emph{any} evaluation quantity that can be written as a resolved/admissible observable of the law is
controlled by $d_T$--in particular the proper scores in Subsection~\ref{subsec:crps} and the clipped likelihood certificates
in Subsection~\ref{subsec:app-likelihood}.
This yields a principled continuum interpretation of distributional evaluation practices used in modern probabilistic weather
systems and diffusion-based ensemble generators \cite{price2023gencast,price2025probabilistic,lang2026aifs,larsson2025diffusion,mahesh2024huge,mahesh2024huge2}.
\end{remark}

\subsection{Diffusion likelihood certificates on patches as admissible observables}
\label{subsec:app-likelihood}

Diffusion-based probabilistic models often provide a \emph{likelihood-like} scalar (e.g.\ conditional log-likelihood or
negative log-likelihood) for a candidate output given an input, computed via a probability--flow ODE and the divergence of
the learned score. Such values are routinely used as ``confidence certificates'' in practice, but are typically analyzed
only at the discretized level. We show that, under a natural consistency hypothesis, clipped patchwise likelihood
certificates are LM-admissible observables and hence converge strongly along refinement.

\paragraph{Patch variables.}
Fix $k\in\N$ and write $x=(x_1,\dots,x_k)\in D^k$, $\xi=(\xi_1,\dots,\xi_k)\in(\R^m)^k$.
We consider an \emph{augmented} patch variable
\[
\widetilde\xi=(\xi^0,\xi^{\mathrm{num}},\xi^{\mathrm{diff}})\in ((\R^m)^3)^k,
\]
encoding an input patch $\xi^0$, a numerical output patch $\xi^{\mathrm{num}}$, and a model output patch
$\xi^{\mathrm{diff}}$.
Let $\widetilde\mu^\Delta_t:=\Law(U_0^\Delta,U_t^\Delta,V^\Delta_{t,\theta(\Delta)})$ be the induced law on
$L^2(D;(\R^m)^3)$ and let $\widetilde\nu^{k,\Delta}_{t,x}$ be its correlation measures valued in $\Pcal(((\R^m)^3)^k)$. We will repeatedly use the marginalization property of correlation measures (Lemma~\ref{lem:corr-marg})
to reduce $k$-patch integrals to one-point quantities.

\paragraph{Probability--flow ODE likelihood on patch space.}
For a conditional diffusion model, a convenient representation of likelihood-like quantities is given by the probability--flow
ODE. We do not re-derive the probability--flow identity here (see Subsection~\ref{subsec:pfode}); we only use that the model
produces a scalar functional
\[
\mathcal L^{\Delta,\mathrm{cond}}_{t,\theta,k}(x;\,a,b)
\qquad (t\in[0,T),\ x\in D^k,\ a,b\in(\R^m)^k),
\]
interpreted as a conditional log-likelihood of an output patch $b$ given an input patch $a$, computed from the learned score
along a probability--flow ODE trajectory.

\paragraph{Certificate consistency and polynomial regularity.}
To place these certificates in the LM observable framework, we assume two properties:
(i) a uniform polynomial Lipschitz bound (needed for admissibility) and (ii) a continuum-limit consistency as $\Delta\to0$
for fixed patch size $k$.

\begin{assumption}[Polynomial regularity and continuum consistency of conditional patch likelihoods]
\label{ass:app-like-consistency}
Fix $k\in\N$. There exists $C_{\mathrm{like}}>0$ independent of $\Delta$, $t$, $x$ such that for all
$t\in[0,T)$, $x\in D^k$, and all $(a,b),(a',b')\in (\R^m)^k\times(\R^m)^k$,
\begin{equation}\label{eq:app-like-poly}
\bigl|
\mathcal L^{\Delta,\mathrm{cond}}_{t,\theta(\Delta),k}(x;a,b)
-
\mathcal L^{\Delta,\mathrm{cond}}_{t,\theta(\Delta),k}(x;a',b')
\bigr|
\le
C_{\mathrm{like}}\,(1+\|a\|^2+\|b\|^2+\|a'\|^2+\|b'\|^2)\,(\|a-a'\|+\|b-b'\|).
\end{equation}
Moreover, there exists a limiting functional $\mathcal L^{\mathrm{cond}}_{t,k}(x;\cdot,\cdot)$ such that
for every $R>0$,
\begin{equation}\label{eq:app-like-limit}
\sup_{t\in[0,T),\,x\in D^k}\ \sup_{\|a\|+\|b\|\le R}
\bigl|
\mathcal L^{\Delta,\mathrm{cond}}_{t,\theta(\Delta),k}(x;a,b)
-
\mathcal L^{\mathrm{cond}}_{t,k}(x;a,b)
\bigr|
\;\longrightarrow\;0
\qquad\text{as }\Delta\to0.
\end{equation}
\end{assumption}

\paragraph{Clipped negative log-likelihood certificate.}
For $M>0$ define $\clip_M(r):=\max\{-M,\min\{r,M\}\}$ and set
\begin{equation}\label{eq:app-gnll}
g^{\Delta,M}_{\mathrm{nll}}(t,x,\widetilde\xi)
:=
-\clip_M\!\left(\mathcal L^{\Delta,\mathrm{cond}}_{t,\theta(\Delta),k}\bigl(x;\xi^0,\xi^{\mathrm{diff}}\bigr)\right),
\qquad
g^{M}_{\mathrm{nll}}(t,x,\widetilde\xi)
:=
-\clip_M\!\left(\mathcal L^{\mathrm{cond}}_{t,k}\bigl(x;\xi^0,\xi^{\mathrm{diff}}\bigr)\right).
\end{equation}

\begin{lemma}[Admissibility of clipped likelihood certificates]\label{lem:app-gnll-adm}
Under Assumption~\ref{ass:app-like-consistency}, for each fixed $k$ and $M$, the observable $g^{\Delta,M}_{\mathrm{nll}}$ is
LM-admissible in the sense of Definition~\ref{def:LM-adm}, with an admissibility constant uniform in $\Delta$.
The same holds for $g^M_{\mathrm{nll}}$.
\end{lemma}

\begin{proof}
Boundedness of clipping gives the growth condition immediately. The Lipschitz condition follows from the $1$-Lipschitz
property of $\clip_M$ and the polynomial Lipschitz estimate \eqref{eq:app-like-poly}, after rewriting the resulting bound
in the LM admissibility form (using $\Pi_i\ge1$).
\end{proof}

\begin{theorem}[Strong convergence of clipped likelihood certificates along refinement]
\label{thm:app-gnll-conv}
Assume the LM compactness hypotheses (uniform time-regularity, uniform $L^2$ support, uniform structure-function modulus)
hold for the augmented laws $\widetilde\mu^\Delta_{\cdot}:=\Law(U_0^\Delta,U_t^\Delta,V^\Delta_{t,\theta(\Delta)})$, and assume
\begin{equation}\label{eq:app-mu-tilde-conv}
d_T(\widetilde\mu^\Delta,\widetilde\mu)\to0
\qquad\text{as }\Delta\to0
\end{equation}
for some $\widetilde\mu_\cdot\in L^1_t(\Pcal(L^2(D;(\R^m)^3)))$ with correlation measures $\widetilde\nu^k_{t,x}$.
Then for each fixed $k$ and $M>0$,
\[
\int_0^T\int_{D^k}
\bigl|
\ip{\widetilde\nu^{k,\Delta}_{t,x}}{g^{\Delta,M}_{\mathrm{nll}}(t,x,\cdot)}
-
\ip{\widetilde\nu^{k}_{t,x}}{g^{M}_{\mathrm{nll}}(t,x,\cdot)}
\bigr|\,\dd x\,\dd t
\;\longrightarrow\;0.
\]
\end{theorem}

\begin{proof}
Decompose
\[
\ip{\widetilde\nu^{k,\Delta}_{t,x}}{g^{\Delta,M}_{\mathrm{nll}}}
-\ip{\widetilde\nu^{k}_{t,x}}{g^{M}_{\mathrm{nll}}}
=
\Big(\ip{\widetilde\nu^{k,\Delta}_{t,x}}{g^{\Delta,M}_{\mathrm{nll}}-g^{M}_{\mathrm{nll}}}\Big)
+\Big(\ip{\widetilde\nu^{k,\Delta}_{t,x}}{g^{M}_{\mathrm{nll}}}-\ip{\widetilde\nu^{k}_{t,x}}{g^{M}_{\mathrm{nll}}}\Big).
\]

\smallskip
\noindent\textbf{Step 1: the $g^M$ term vanishes by LM observable convergence.}
By Lemma~\ref{lem:app-gnll-adm}, $g^M_{\mathrm{nll}}$ is LM-admissible. Since
$d_T(\widetilde\mu^\Delta,\widetilde\mu)\to0$, LM strong convergence of admissible observables
(Theorem~2.4 of \cite{LMP2021}, as recorded in Theorem~\ref{thm:compact}) yields
\[
\int_0^T\!\!\int_{D^k}
\Big|
\ip{\widetilde\nu^{k,\Delta}_{t,x}}{g^{M}_{\mathrm{nll}}(t,x,\cdot)}
-
\ip{\widetilde\nu^{k}_{t,x}}{g^{M}_{\mathrm{nll}}(t,x,\cdot)}
\Big|\,\dd x\,\dd t \longrightarrow 0.
\]

\smallskip
\noindent\textbf{Step 2: control the $g^{\Delta,M}-g^M$ term by bounded-set convergence + tail.}
Since $\clip_M$ is $1$-Lipschitz and bounded by $M$, we have for all $(t,x,\widetilde\xi)$,
\[
\big|g^{\Delta,M}_{\mathrm{nll}}(t,x,\widetilde\xi)-g^{M}_{\mathrm{nll}}(t,x,\widetilde\xi)\big|
\le
\big|\mathcal L^{\Delta,\mathrm{cond}}_{t,\theta(\Delta),k}(x;\xi^0,\xi^{\mathrm{diff}})
-\mathcal L^{\mathrm{cond}}_{t,k}(x;\xi^0,\xi^{\mathrm{diff}})\big|,
\]
and also the crude bound
\[
\big|g^{\Delta,M}_{\mathrm{nll}}-g^{M}_{\mathrm{nll}}\big|\le 2M.
\]

Fix $R>0$ and write $\mathbf 1_R$ for the indicator of the bounded set
$\{\| \xi^0\|+\|\xi^{\mathrm{diff}}\|\le R\}$ in $(\R^m)^k\times(\R^m)^k$.
Then for all $(t,x)$,
\begin{align*}
\Big|\ip{\widetilde\nu^{k,\Delta}_{t,x}}{g^{\Delta,M}_{\mathrm{nll}}-g^{M}_{\mathrm{nll}}}\Big|
&\le
\ip{\widetilde\nu^{k,\Delta}_{t,x}}{\big|g^{\Delta,M}_{\mathrm{nll}}-g^{M}_{\mathrm{nll}}\big|\mathbf 1_R}
+
\ip{\widetilde\nu^{k,\Delta}_{t,x}}{\big|g^{\Delta,M}_{\mathrm{nll}}-g^{M}_{\mathrm{nll}}\big|(1-\mathbf 1_R)} \\
&\le
\sup_{\substack{t\in[0,T),\,x\in D^k\\ \|a\|+\|b\|\le R}}
\big|\mathcal L^{\Delta,\mathrm{cond}}_{t,\theta(\Delta),k}(x;a,b)-\mathcal L^{\mathrm{cond}}_{t,k}(x;a,b)\big|
\;+\; 2M\,\widetilde\nu^{k,\Delta}_{t,x}\big(\| \xi^0\|+\|\xi^{\mathrm{diff}}\|>R\big).
\end{align*}
Integrate in $(t,x)$ to obtain
\begin{align}
\int_0^T\!\!\int_{D^k}
\Big|\ip{\widetilde\nu^{k,\Delta}_{t,x}}{g^{\Delta,M}_{\mathrm{nll}}-g^{M}_{\mathrm{nll}}}\Big|\,\dd x\,\dd t
&\le
T|D|^k\,\delta_\Delta(R)
+
2M\int_0^T\!\!\int_{D^k}\widetilde\nu^{k,\Delta}_{t,x}\big(\| \xi^0\|+\|\xi^{\mathrm{diff}}\|>R\big)\,\dd x\,\dd t,
\label{eq:split-bounded-tail}
\end{align}
where
\[
\delta_\Delta(R):=
\sup_{\substack{t\in[0,T),\,x\in D^k\\ \|a\|+\|b\|\le R}}
\big|\mathcal L^{\Delta,\mathrm{cond}}_{t,\theta(\Delta),k}(x;a,b)-\mathcal L^{\mathrm{cond}}_{t,k}(x;a,b)\big|.
\]
By Assumption~\ref{ass:app-like-consistency}\eqref{eq:app-like-limit}, $\delta_\Delta(R)\to0$ as $\Delta\to0$ for every fixed $R$.

It remains to bound the tail term uniformly in $\Delta$. By Chebyshev and $(\|a\|+\|b\|)^2\le 2(\|a\|^2+\|b\|^2)$,
\begin{align*}
\widetilde\nu^{k,\Delta}_{t,x}\big(\| \xi^0\|+\|\xi^{\mathrm{diff}}\|>R\big)
&\le \frac{1}{R^2}\,\ip{\widetilde\nu^{k,\Delta}_{t,x}}{(\|\xi^0\|+\|\xi^{\mathrm{diff}}\|)^2}\\
&\le \frac{2}{R^2}\,\ip{\widetilde\nu^{k,\Delta}_{t,x}}{\|\xi^0\|^2+\|\xi^{\mathrm{diff}}\|^2}.
\end{align*}
Integrating in $x\in D^k$ and using the marginalization identity from Lemma~\ref{lem:corr-marg} (with $\psi(z)=|z^0|^2+|z^{\mathrm{diff}}|^2$)
yields
\[
\int_{D^k}\Big\langle \widetilde\nu^{k,\Delta}_{t,x},\|\xi^0\|^2+\|\xi^{\mathrm{diff}}\|^2\Big\rangle\,\dd x
=
|D|^{k-1}\int_D \Big\langle \widetilde\nu^{1,\Delta}_{t,y},|\xi^0|^2+|\xi^{\mathrm{diff}}|^2\Big\rangle\,\dd y.
\]

By the uniform $L^2$ support (energy) assumption for the augmented laws $\widetilde\mu^\Delta_t$,
the right-hand side is bounded uniformly in $(t,\Delta)$ by a constant $C_{k}$ (depending on $k$ and the support radius).
Hence
\[
\int_0^T\!\!\int_{D^k}\widetilde\nu^{k,\Delta}_{t,x}\big(\| \xi^0\|+\|\xi^{\mathrm{diff}}\|>R\big)\,\dd x\,\dd t
\le \frac{C_k\,T}{R^2}.
\]
Insert this into \eqref{eq:split-bounded-tail}:
\[
\int_0^T\!\!\int_{D^k}
\Big|\ip{\widetilde\nu^{k,\Delta}_{t,x}}{g^{\Delta,M}_{\mathrm{nll}}-g^{M}_{\mathrm{nll}}}\Big|\,\dd x\,\dd t
\le
T|D|^k\,\delta_\Delta(R)+\frac{2M C_k\,T}{R^2}.
\]
Now choose $R$ large so that $\frac{2M C_k\,T}{R^2}<\varepsilon/2$, then choose $\Delta$ small so that
$T|D|^k\,\delta_\Delta(R)<\varepsilon/2$. This proves the first term converges to $0$ in $L^1_{t,x}$.

\smallskip
\noindent\textbf{Conclusion.}
Both terms in the initial decomposition vanish in $L^1_{t,x}$, yielding the claimed convergence.
\end{proof}

\begin{remark}[Why fixed patches are the right continuum objects]
A global conditional likelihood on the full grid has dimension $N_\Delta\to\infty$ as $\Delta\to0$ and therefore is not a
single functional on $L^2_x$ in a one-parameter limit. The LM framework is naturally compatible with fixed-dimensional
observables: fixed $k$-point statistics, fixed projections, and resolved/mollified evaluations. Patchwise likelihood
certificates fit exactly into this class and therefore admit clean refinement-limit statements.
\end{remark}

\subsection{From likelihood certificates to $L^2$ error under a strong convexity hypothesis}
\label{subsec:app-like-to-err}

Likelihood values become \emph{quantitative} error certificates only under additional structure.
A clean sufficient condition is strong convexity of the conditional negative log-likelihood in the output variable,
anchored at the numerical truth.

For simplicity we state this at $k=1$ (one-point patches). Let
$V^\Delta_{t,\theta}(x;a,b):=-\mathcal L^{\Delta,\mathrm{cond}}_{t,\theta,1}(x;a,b)$.

\begin{assumption}[Strong convexity of conditional NLL in the output]\label{ass:app-strong-conv}
There exists $\lambda>0$ independent of $\Delta,t,x$ such that for a.e.\ $(t,x)$ and all $a\in\R^m$, $b\in\R^m$,
\[
V^\Delta_{t,\theta(\Delta)}(x;a,b)-V^\Delta_{t,\theta(\Delta)}(x;a,b^{\mathrm{true}})
\ge \frac{\lambda}{2}\,|b-b^{\mathrm{true}}|^2,
\qquad b^{\mathrm{true}}:=\xi^{\mathrm{num}}.
\]
\end{assumption}

Define the excess conditional NLL observable (one-point) by
\[
\mathrm{XNLL}^\Delta_{t}(x;\widetilde\xi)
:=
V^\Delta_{t,\theta(\Delta)}(x;\xi^0,\xi^{\mathrm{diff}})
-
V^\Delta_{t,\theta(\Delta)}(x;\xi^0,\xi^{\mathrm{num}}),
\qquad \widetilde\xi=(\xi^0,\xi^{\mathrm{num}},\xi^{\mathrm{diff}})\in(\R^m)^3.
\]

\begin{theorem}[Small excess NLL implies small $L^2$ mismatch in the pipeline]
\label{thm:app-xnll-to-mse}
Assume \cref{ass:app-strong-conv}. Let $\pi^\Delta_{t,\theta(\Delta)}$ be the pipeline coupling \eqref{eq:app-pi}.
Then for a.e.\ $t\in[0,T)$,
\begin{equation}\label{eq:xnll-to-mse}
\int_{L^2_x \times L^2_x}\|u-v\|_{2}^2\,\dd\pi^\Delta_{t,\theta(\Delta)}(u,v)
\le \frac{2C_R^2}{\lambda}\int_D 
\Big\langle \widetilde\nu^{1,\Delta}_{t,x}\,,\,\mathrm{XNLL}^\Delta_t(x;\cdot)\Big\rangle\,\dd x,
\end{equation}
with $C_R$ from \eqref{eq:app-R-stab}. In particular, small expected excess conditional NLL implies small mean-square pipeline error.
\end{theorem}

\begin{proof}
Fix a time $t$ such that the strong convexity condition in \cref{ass:app-strong-conv} holds for a.e.\ $x\in D$.

\smallskip
\noindent\textbf{Step 1: pointwise strong convexity gives a pointwise squared-error bound.}
Recall that $V^\Delta_{t,\theta(\Delta)}(x;a,b):=-\mathcal L^{\Delta,\mathrm{cond}}_{t,\theta(\Delta),1}(x;a,b)$ and
\[
\mathrm{XNLL}^\Delta_{t}(x;\widetilde\xi)
=
V^\Delta_{t,\theta(\Delta)}(x;\xi^0,\xi^{\mathrm{diff}})
-
V^\Delta_{t,\theta(\Delta)}(x;\xi^0,\xi^{\mathrm{num}}),
\qquad
\widetilde\xi=(\xi^0,\xi^{\mathrm{num}},\xi^{\mathrm{diff}})\in(\R^m)^3.
\]
By \cref{ass:app-strong-conv}, for a.e.\ $x$ and all $(\xi^0,\xi^{\mathrm{num}},\xi^{\mathrm{diff}})$,
\[
\mathrm{XNLL}^\Delta_{t}(x;\widetilde\xi)
=
V^\Delta_{t,\theta(\Delta)}(x;\xi^0,\xi^{\mathrm{diff}})
-
V^\Delta_{t,\theta(\Delta)}(x;\xi^0,\xi^{\mathrm{num}})
\;\ge\; \frac{\lambda}{2}\,|\xi^{\mathrm{diff}}-\xi^{\mathrm{num}}|^2.
\]
Equivalently,
\begin{equation}\label{eq:pointwise-xnll-bound}
|\xi^{\mathrm{diff}}-\xi^{\mathrm{num}}|^2 \;\le\; \frac{2}{\lambda}\,\mathrm{XNLL}^\Delta_{t}(x;\widetilde\xi).
\end{equation}

\smallskip
\noindent\textbf{Step 2: apply \eqref{eq:pointwise-xnll-bound} to the pipeline random fields and integrate in space.}
Under the pipeline coupling, we have the reconstructed numerical field $U_t^\Delta\in L^2_x$ and the reconstructed model
field $V_{t,\theta(\Delta)}^\Delta\in L^2_x$. At the discretized level these come from
\[
U_t^\Delta = R_\Delta Z_t, \qquad V_{t,\theta(\Delta)}^\Delta = R_\Delta \widehat Z_t,
\]
where $Z_t$ is the numerical output and $\widehat Z_t$ is the model output conditioned on the same input (see
\eqref{eq:app-pi}). For $k=1$ patches, $\xi^{\mathrm{num}}$ and $\xi^{\mathrm{diff}}$ represent the pointwise values
(of the discretized fields) at location $x$.

Applying \eqref{eq:pointwise-xnll-bound} pointwise in $x$ and integrating over $D$ gives
\begin{equation}\label{eq:disc-l2-from-xnll}
\|Z_t-\widehat Z_t\|_{L^2_{x,\Delta}}^2
\;\le\; \frac{2}{\lambda}\int_D \mathrm{XNLL}^\Delta_t\!\bigl(x;\widetilde\xi(x)\bigr)\,\dd x,
\end{equation}
where $\widetilde\xi(x)$ denotes the triple of $k=1$ patch variables
\[
\widetilde\xi(x) = \bigl(\xi^0(x),\,\xi^{\mathrm{num}}(x),\,\xi^{\mathrm{diff}}(x)\bigr)
\]
extracted from the augmented pipeline output $(U_0^\Delta,U_t^\Delta,V_{t,\theta(\Delta)}^\Delta)$ at location $x$.

\smallskip
\noindent\textbf{Step 3: use reconstruction stability to pass from discrete to $L^2_x$.}
By the stability assumption \eqref{eq:app-R-stab},
\[
\|U_t^\Delta - V_{t,\theta(\Delta)}^\Delta\|_{2}
=
\|R_\Delta(Z_t-\widehat Z_t)\|_{2}
\le C_R\,\|Z_t-\widehat Z_t\|_{L^2_{x,\Delta}}.
\]
Squaring and combining with \eqref{eq:disc-l2-from-xnll} yields the pointwise (in the underlying probability space) bound
\begin{equation}\label{eq:field-l2-from-xnll}
\|U_t^\Delta - V_{t,\theta(\Delta)}^\Delta\|_{2}^2
\le \frac{2C_R^2}{\lambda}\int_D \mathrm{XNLL}^\Delta_t\!\bigl(x;\widetilde\xi(x)\bigr)\,\dd x.
\end{equation}

\smallskip
\noindent\textbf{Step 4: take expectation under the pipeline coupling.}
By definition of the coupling $\pi^\Delta_{t,\theta(\Delta)}=\Law(U_t^\Delta,V_{t,\theta(\Delta)}^\Delta)$,
\[
\int_{L^2_x\times L^2_x}\|u-v\|_2^2\,\dd\pi^\Delta_{t,\theta(\Delta)}(u,v)
=
\E\|U_t^\Delta - V_{t,\theta(\Delta)}^\Delta\|_2^2.
\]
Taking expectation in \eqref{eq:field-l2-from-xnll} and using Tonelli gives
\begin{equation}\label{eq:expect-xnll}
\E\|U_t^\Delta - V_{t,\theta(\Delta)}^\Delta\|_{2}^2
\le \frac{2C_R^2}{\lambda}\int_D \E\Big[\mathrm{XNLL}^\Delta_t\!\bigl(x;\widetilde\xi(x)\bigr)\Big]\,\dd x.
\end{equation}

\smallskip
\noindent\textbf{Step 5: rewrite the pointwise expectation using the $k=1$ correlation measure.}
By construction, $\widetilde\nu^{1,\Delta}_{t,x}$ is the $k=1$ correlation measure of the augmented law
$\widetilde\mu_t^\Delta=\Law(U_0^\Delta,U_t^\Delta,V_{t,\theta(\Delta)}^\Delta)$ at location $x$, i.e.\ it is exactly the law
of the triple $\widetilde\xi(x)$ at that point. Hence,
\[
\E\Big[\mathrm{XNLL}^\Delta_t\!\bigl(x;\widetilde\xi(x)\bigr)\Big]
=
\Big\langle \widetilde\nu^{1,\Delta}_{t,x}\,,\,\mathrm{XNLL}^\Delta_t(x;\cdot)\Big\rangle.
\]
Insert this into \eqref{eq:expect-xnll} and use the identification in Step 4. This yields \eqref{eq:xnll-to-mse}.
\end{proof}

\begin{remark}[Interpretation]
The strong convexity hypothesis is a sufficient condition turning likelihood values into rigorous error certificates.
In practice, one may expect local/approximate convexity around typical outputs rather than a global uniform constant.
The statement above isolates the exact structural ingredient needed for a quantitative implication, and the LM framework
then propagates such certificate information to the continuum limit through admissible-observable convergence.
\end{remark}

\subsection{Relation to CRPS and common distributional scores used in ML PDE forecasting}
\label{subsec:crps}

CRPS and energy-score evaluations are now standard in probabilistic ML forecasting pipelines and tooling
\cite{lang2026aifs,bulte2025probabilistic,mahesh2024huge,mahesh2024huge2}.
The bounds below formalize these scores as Lipschitz resolved observables controlled by $d_T$, giving a direct quantitative
route from LM convergence to convergence of the same distributional metrics used in practice in large-scale weather systems
and other probabilistic PDE solvers \cite{price2023gencast,price2025probabilistic,larsson2025diffusion,andrae2024continuous}.

In the present law-level setting, these scores can be expressed as expectations of Lipschitz observables of (resolved)
one-point marginals, and are therefore controlled quantitatively by the LM topology.

\subsubsection{CRPS and the energy distance}

Let $P\in\Pcal_1(\R)$ and let $Y\in\R$ be an observation. The CRPS of $P$ at $Y$ admits the representation
\begin{equation}\label{eq:crps-point}
\mathrm{CRPS}(P,Y)
:= \E|X-Y|-\frac12\E|X-X'|,
\qquad X,X'\stackrel{\mathrm{iid}}{\sim}P.
\end{equation}
If $Q\in\Pcal_1(\R)$ is the law of $Y$ and $Y,Y'\stackrel{\mathrm{iid}}{\sim}Q$ are independent of $X,X'$, then the expected CRPS is
\begin{equation}\label{eq:crps-law}
\mathrm{CRPS}(P,Q)
:= \E_{Y\sim Q}\mathrm{CRPS}(P,Y)
= \E|X-Y|-\frac12\E|X-X'|-\frac12\E|Y-Y'|.
\end{equation}
The right-hand side is (one half of) the classical \emph{energy distance} between $P$ and $Q$, and vanishes iff $P=Q$.

\begin{lemma}[CRPS is controlled by $W_1$]\label{lem:crps-w1}
For all $P,Q\in\Pcal_1(\R)$,
\begin{equation}\label{eq:crps-w1}
\mathrm{CRPS}(P,Q)\le 2\,W_1(P,Q).
\end{equation}
More generally, for any $P,P',Q\in\Pcal_1(\R)$,
\begin{equation}\label{eq:crps-lip}
\big|\mathrm{CRPS}(P,Q)-\mathrm{CRPS}(P',Q)\big|\le 2\,W_1(P,P').
\end{equation}
\end{lemma}

\begin{proof}
Let $Y\sim Q$. Since $x\mapsto |x-Y|$ is $1$-Lipschitz, Kantorovich--Rubinstein duality gives
\[
\big|\E|X-Y|-\E|X'-Y|\big|\le W_1(P,P')
\quad\text{for }X\sim P,\ X'\sim P'.
\]
For the pairwise term, the function $h(x,x'):=|x-x'|$ is $1$-Lipschitz with respect to the product metric
$|(x,x')-(y,y')|:=|x-y|+|x'-y'|$. Coupling $P$ and $P'$ optimally by $\pi$ and using the product coupling $\pi\times\pi$
yields
\[
\Big|\E|X-\widetilde X|-\E|X'-\widetilde X'|\Big|\le 2\,W_1(P,P'),
\qquad (X,X')\sim\pi,\ (\widetilde X,\widetilde X')\sim\pi.
\]
Insert these two bounds into \eqref{eq:crps-law} to obtain \eqref{eq:crps-lip}. Choosing $P'=Q$ and using
$\mathrm{CRPS}(Q,Q)=0$ gives \eqref{eq:crps-w1}.
\end{proof}

\subsubsection{CRPS for field-valued laws via resolved observables}

In PDE forecasting, CRPS is usually computed on scalar components of the field evaluated at grid points.
At the continuum level, point evaluation is not continuous on $L^2$, so we define the analogous score through
\emph{resolved} (or mollified) scalar observables.

Let $\ell:L^2_x\to\R$ be a Lipschitz functional with constant $\Lip(\ell)$ (e.g.\ $\ell(u)=\ip{u}{\psi}$ for some
$\psi\in L^2_x$, or a mollified point evaluation). For a law $\mu_t\in\Pcal_1(L^2_x)$ define the pushforward
$P_t:=\ell_\#\mu_t\in\Pcal_1(\R)$. Given another law $\nu_t$ define $Q_t:=\ell_\#\nu_t$.
By \cref{lem:crps-w1} and the contraction property of $W_1$ under Lipschitz maps,
\begin{equation}\label{eq:crps-dt-control}
\mathrm{CRPS}(P_t,Q_t)
\le 2\,W_1(P_t,Q_t)
\le 2\,\Lip(\ell)\,W_1(\mu_t,\nu_t).
\end{equation}
Integrating in time yields the quantitative bound
\begin{equation}\label{eq:crps-dT}
\int_0^T \mathrm{CRPS}(\ell_\#\mu_t,\ell_\#\nu_t)\,\dd t
\le 2\,\Lip(\ell)\,d_T(\mu,\nu).
\end{equation}
In particular, convergence in $d_T$ implies convergence of these CRPS scores for any fixed resolved observable $\ell$.

\begin{remark}[Connection to gridpoint CRPS used in ML]
Gridpoint CRPS corresponds to choosing $\ell(u)$ as a discrete evaluation of a component at a grid node.
In the continuum limit, this is naturally modeled by mollified evaluations
$\ell_{\varepsilon,x}(u)=\ip{u}{\eta_\varepsilon(\cdot-x)e_j}$ with $\eta_\varepsilon$ a smooth kernel and $e_j$ a coordinate vector.
Each $\ell_{\varepsilon,x}$ is Lipschitz on $L^2$ with $\Lip(\ell_{\varepsilon,x})=\|\eta_\varepsilon\|_{L^2}$.
Thus \eqref{eq:crps-dT} gives a direct quantitative control of (mollified/resolved) pointwise CRPS in terms of the LM metric.
\end{remark}

\subsubsection{Energy score (multivariate CRPS) and other proper scores}

The multivariate analogue widely used in probabilistic forecasting is the \emph{energy score} on $\R^m$,
obtained by replacing absolute values in \eqref{eq:crps-point}--\eqref{eq:crps-law} by the Euclidean norm.
The same argument as \cref{lem:crps-w1} yields an identical Lipschitz control in $W_1$ (up to constants depending on the norm),
and \eqref{eq:crps-dT} extends verbatim to vector-valued resolved observables $\ell:L^2_x\to\R^m$.
Consequently, the LM observable convergence framework quantitatively covers standard distributional metrics used in ML PDE
forecasting, including CRPS/energy-score type evaluations.

\section{Conclusion}

This work develops a law-level analysis for modern \emph{probabilistic} PDE forecasters, with incompressible Euler as the
guiding example in the sense of measure-valued/statistical solutions \cite{LMP2021} and classical background
\cite{majda2002vorticity}. The motivation is practical and timely: recent progress in ML has produced highly capable
distributional solvers and ensemble generators for geophysical flows and turbulence, including diffusion-based ensemble
weather forecasting and its variants \cite{price2023gencast,price2025probabilistic,larsson2025diffusion,andrae2024continuous},
CRPS-trained ensemble models \cite{lang2026aifs}, huge neural-operator ensembles \cite{mahesh2024huge,mahesh2024huge2}, and a
growing ecosystem of probabilistic operator-learning and generative turbulence models
\cite{bulte2025probabilistic,kohl2023benchmarking,du2024conditional,gao2024bayesian,lienen2023zero,oommen2024integrating,boxho2025turbulent}.
These systems are typically analyzed on discretized state spaces with distributional scores (e.g.\ CRPS/energy score) and
sample diagnostics; our goal was to provide a continuum-compatible, quantitative framework that turns such practices into
statements about law evolutions and statistical-solution identities.

The starting point is that conditional samplers (flow matching / rectified flows / diffusion via probability-flow ODEs)
define Markov kernels on state space and hence Markov operators on laws. This shifts the analysis from trajectory
stability to stability and approximation of \emph{law evolutions}, aligning with the operator-centric perspective emerging
in generative PDE modeling \cite{chen2025bridging,armegioiu2025rectified,molinaro2024generative}. On the quantitative side,
we proved a Wasserstein stability mechanism in which the growth of $W_2$ is controlled by a distance-weighted average
strain evaluated along coupled pairs, rather than by a worst-case Lipschitz constant. We then isolated the finite-resolution
obstruction: one-step law error decomposes into a resolved mismatch and an unresolved tail, with the latter controlled by
structure-function bounds (equivalently, spectral tails). These ingredients combine into explicit rollout bounds via a
discrete Gr\"onwall recursion, separating stability amplification from injected one-step defects--precisely the components
that empirically drive long-horizon degradation in neural PDE solvers \cite{lippe2308pde,schiff2024dyslim} and in
autoregressive generative rollouts for turbulent flows \cite{kohl2023benchmarking}.

On the qualitative side, we placed sampler-induced law curves into the Lanthaler--Mishra--Par\'es-Pulido framework
\cite{LMP2021}. Under uniform energy admissibility, a uniform structure-function modulus, and LM time-regularity,
compactness holds in the LM topology $d_T$, and admissible observables converge strongly; this provides a principled route
to interpreting distributional metrics used in practice as controlled observables of the law (including CRPS-type scores
through resolved/mollified evaluations, cf.\ Section~\ref{subsec:crps}). If, in addition, the Euler hierarchy identities
hold up to residuals vanishing along a sequence, then every subsequential limit is an LM statistical solution. We provided
a training-native route to verifying residual smallness on resolved scales: for drift-driven evolutions, hierarchy
residuals reduce exactly to drift-defect expectations and can be bounded by $L^2$ drift regression losses; for diffusion
models, standard score regression controls the corresponding probability-flow drift error via an explicit identity. This
connects the learning objectives used across probabilistic forecasting and downscaling pipelines
\cite{lang2026aifs,larsson2025diffusion,aich2024conditional,mardani2025residual,jin2023downscaling} to the weak identities
that define Euler at the level of correlations.

Several extensions are natural. First, making the two-parameter limit $(\Delta,K)\to(0,\infty)$ fully explicit would turn
resolved certification into a complete proof of vanishing full-hierarchy residuals, with structure-function control
providing the unresolved-scale closure. Second, extending the stability and certification mechanisms to other conservation
laws and to dissipative settings (e.g.\ Navier--Stokes, closures, or solver-defined reference dynamics) would broaden the
scope of law-level guarantees for probabilistic PDE solvers in the regimes targeted by current generative surrogates
\cite{price2023gencast,larsson2025diffusion,du2024conditional,gao2024bayesian}. Third, incorporating observation operators
and data-assimilation constraints into the same law-level framework would directly connect these results to operational
ensemble workflows \cite{price2025probabilistic,lang2026aifs,mahesh2024huge2}. Finally, it would be interesting to further
tighten the bridge to optimal-transport tools and Wasserstein gradient-flow viewpoints \cite{figalli2021invitation} that
naturally align with transport-based samplers.

Overall, the picture that emerges is that law-level analysis provides a principled bridge between generative sampling
mechanisms and rigorous statistical-solution notions for fluid dynamics: stability and approximation can be quantified in
Wasserstein distance, compactness can be enforced and checked through structure functions, and training losses can be
converted into certification of the weak identities that define the target PDE at the level of correlations.

\printbibliography
\end{document}